\documentclass[11pt]{article}

\usepackage[utf8]{inputenc}
\usepackage{amsmath,amssymb,amsthm}
\usepackage{caption, subcaption}
\usepackage{geometry}
\usepackage{graphicx}
\usepackage{thmtools, thm-restate}
\usepackage{xcolor}
\usepackage[normalem]{ulem}
\usepackage{cancel}
\usepackage[sort]{natbib} 
\usepackage{float}
\usepackage{enumitem}
\usepackage{listings}
\usepackage{hyperref}
\hypersetup{colorlinks,citecolor=blue,urlcolor=blue,linkcolor=blue}
\usepackage{enumitem}

\allowdisplaybreaks

\graphicspath{{../Fig}}

\newtheorem{theorem}{Theorem}
\newtheorem{lemma}{Lemma}
\newtheorem{corollary}{Corollary}
\newtheorem{proposition}{Proposition}

\newtheorem{algorithm}{Algorithm}
\newtheorem{estimator}{Estimator}
\newtheorem{subestimator}{Estimator}[estimator]

\theoremstyle{definition}

\theoremstyle{remark}
\newtheorem{assumption}{Assumption}
\newtheorem{remark}{Remark}

\def\inprob{\stackrel{p}{\rightarrow}}
\def\indist{\rightsquigarrow}
\def\inas{\stackrel{a.s.}{\longrightarrow}}

\newcommand\ind{\protect\mathpalette{\protect\independenT}{\perp}}
\def\independenT#1#2{\mathrel{\rlap{$#1#2$}\mkern4mu{#1#2}}}

\DeclareSymbolFont{bbold}{U}{bbold}{m}{n}
\DeclareSymbolFontAlphabet{\mathbbold}{bbold}
\newcommand{\one}{\mathbbold{1}}
\def\bbE{\mathbb{E}}
\def\bbP{\mathbb{P}}
\def\bbR{\mathbb{R}}
\def\bbV{\mathbb{V}}
\def\cov{\hspace{0.05em}\text{cov}\hspace{0.05em}}

\title{Double cross-fit doubly robust estimators: \\ Beyond series regression}
\author{
  Alec McClean\textsuperscript{1},
  Sivaraman Balakrishnan\textsuperscript{2},
  Edward H. Kennedy\textsuperscript{2},
  and Larry Wasserman\textsuperscript{2} \\\\
  \textsuperscript{1}Division of Biostatistics, NYU Grossman School of Medicine \\
  \textsuperscript{2}Department of Statistics \& Data Science, Carnegie Mellon University \\\\
  \texttt{hadera01@nyu.edu, \{siva, edward, larry\}@stat.cmu.edu}
}
\date{}

\begin{document}
\maketitle
	
\begin{abstract}
    Doubly robust estimators with cross-fitting have gained popularity in causal inference due to their favorable structure-agnostic error guarantees. However, when additional structure, such as H\"{o}lder smoothness, is available then more accurate ``double cross-fit doubly robust'' (DCDR) estimators can be constructed by splitting the training data and undersmoothing nuisance function estimators on independent samples.  We study a DCDR estimator of the Expected Conditional Covariance, a functional of interest in causal inference and conditional independence testing. We first provide a structure-agnostic error analysis for the DCDR estimator with no assumptions on the nuisance functions or their estimators.  Then, assuming the nuisance functions are H\"{o}lder smooth, but without assuming knowledge of the true smoothness level or the covariate density, we establish that DCDR estimators with several linear smoothers are $\sqrt{n}$-consistent and asymptotically normal under minimal conditions and achieve fast convergence rates in the non-$\sqrt{n}$ regime. When the covariate density and smoothnesses are known, we propose a minimax rate-optimal DCDR estimator based on undersmoothed kernel regression.  Moreover, we show an undersmoothed DCDR estimator satisfies a slower-than-$\sqrt{n}$ central limit theorem, and that inference is possible even in the non-$\sqrt{n}$ regime. Finally, we support our theoretical results with simulations, providing intuition for double cross-fitting and undersmoothing, demonstrating where our estimator achieves $\sqrt{n}$-consistency while the usual ``single cross-fit'' estimator fails, and illustrating asymptotic normality for the undersmoothed DCDR estimator.
\end{abstract}
    
\section{Introduction} \label{sec:intro}

In statistical estimation, the goal often is to construct low-dimensional functionals of an unknown data-generating distribution. Causal effects, such as the average treatment effect, the local average treatment effect, and the average treatment effect on the treated, are prime examples of low-dimensional functionals. Typically, estimators for these functionals are built as summary statistics of nuisance function estimates, such as the propensity score or outcome regression function. In recent decades, doubly robust estimators based on influence functions and semiparametric efficiency theory have gained prominence due to their favorable statistical properties, including robustness to model misspecification and improved efficiency \citep{vanderlaan2003unified, tsiatis2006semiparametric, kennedy2024semiparametric}. Crucially, these estimators can be \emph{cross-fit}, whereby the nuisance functions are estimated on a separate sample from that used to evaluate the functional estimator. Cross-fitting avoids restrictive Donsker-type conditions and enables flexible machine learning methods for nuisance estimation \citep{chernozhukov2018double, zheng2010asymptotic, robins2008higher}.  A widely used approach minimizes the mean squared error (MSE) of the nuisance estimators on a training set and then applies cross-fitting to construct the functional estimator. We refer to this method as the \emph{single cross-fit doubly robust-MSE (SCDR-MSE) estimator} (see, e.g., \citet{kennedy2024semiparametric} for a review).

\medskip

The SCDR-MSE estimator is attractive in practice: when the nuisance estimators’ MSE converges at an $n^{-1/4}$ rate (along with mild regularity conditions), it attains $\sqrt{n}$-consistency and asymptotic normality. This result is particularly appealing because it ensures that generic machine learning algorithms---trained solely to minimize MSE---can yield valid statistical inference under minimal assumptions. Recent theoretical work has further demonstrated that the SCDR-MSE estimator is minimax optimal in a particular structure-agnostic setting, meaning that no estimator can outperform it without additional knowledge of the nuisance functions' structure \citep{balakrishnan2023fundamental, jin2024structure}.

\medskip

\color{black}
However, a key limitation of the SCDR-MSE estimator is that it remains agnostic to any additional structure in the nuisance functions. While this generality ensures robustness, it can lead to suboptimal performance when smoother or lower-complexity nuisance functions permit faster convergence rates. A growing body of work has explored refinements to address this issue under smoothness assumptions. Higher-order estimators---originally proposed by \citet{robins2008higher}---utilize additional influence function corrections to reduce bias and achieve optimal convergence rates under smoothness assumptions \citep{robins2009semiparametric, liu2023new, vandervaart2014higher, robins2017minimax, liu2021adaptive, bonvini2024doubly}.  Meanwhile, \citet{mcgrath2022undersmoothing} demonstrated that SCDR and plug-in estimators incorporating undersmoothed orthogonal wavelet estimators could also match these efficiency gains, aligning with broader findings on cross-fitting and undersmoothing in semiparametric estimation \citep{gine2008simple, paninski2008undersmoothed, newey1998undersmoothing, vanderlaan2022efficient}. \color{black} Despite these theoretical advances, practical implementation remains a challenge.

\medskip

An alternative approach, first proposed by \citet{newey2018cross}, is the \emph{double cross-fit doubly robust (DCDR) estimator}. This estimator retains the doubly robust framework but introduces an additional layer of cross-fitting, where the nuisance estimators are trained on \emph{separate, independent} samples. Double cross-fitting is a simple yet effective modification that, as recent work suggests, can lead to rate-optimal estimation in both the $\sqrt{n}$- and non-$\sqrt{n}$-regimes, particularly when combined with undersmoothing \citep{mcgrath2022undersmoothing, fisher2023threeway, kennedy2023towards}.  However, several important questions remain about its theoretical guarantees and practical applicability:
\begin{enumerate}
    \itemsep0.05in
    \item Most analyses of the DCDR estimator rely on smoothness assumptions, leaving open the question of whether it retains its favorable properties in a structure-agnostic setting.
    \item Existing results with smoothness assumptions primarily focus on series regression nuisance estimators, raising the question of whether similar guarantees extend to other common estimators like k-Nearest-Neighbors or local polynomial regression. 
    \item While recent work has shown that the DCDR estimator can attain minimax-optimal convergence rates in the non-$\sqrt{n}$ regime, it remains unclear whether valid inference procedures exist. 
    \item Finally, empirical validation is lacking: theoretical guarantees suggest that DCDR should perform well, but little is known about how it compares to the standard SCDR-MSE estimator. 
\end{enumerate}
This paper addresses these gaps in the literature.

\color{black}
\subsection{Structure of the paper and our contributions}

We estimate the Expected Conditional Covariance (ECC), a causal effect \citep{li2011higher, diaz2023nonagency} which is also relevant to conditional independence testing \citep{shah2020hardness}, using a DCDR estimator. After providing further background in Section~\ref{sec:setup}, the structure of the paper and our main contributions are as follows:

\begin{enumerate}
    \itemsep0.05in
    \item \textbf{Structure-agnostic analysis (Section~\ref{sec:modelfree}).} 
    We derive a new asymptotically linear expansion for the DCDR estimator with minimal assumptions on the nuisance functions or their estimators. 

    \item \textbf{H\"older smoothness and local averaging estimators (Section~\ref{sec:unknown}).} 
    Under H\"older smoothness assumptions for the nuisance estimates, we construct a DCDR estimator using Nearest Neighbors and local polynomial regression estimators, complementing earlier results with series regression \citep{newey2018cross, mcgrath2022undersmoothing}. 

    \item \textbf{Known density and non-$\sqrt{n}$ inference (Section~\ref{sec:known}).} 
    Supposing the covariate density and smoothness levels are known, we develop a new DCDR estimator with kernel regression nuisance estimators that is rate-optimal for non-$\sqrt{n}$ convergence, complementing previous results with orthogonalized wavelet estimators \citep{mcgrath2022undersmoothing}. Then, we establish a slower-than-$\sqrt{n}$ central limit theorem, the first of its kind for a cross-fit doubly robust estimator, building on prior results with higher-order estimators \citep{robins2016asymptotic}.

    \item \textbf{Empirical validation (Section~\ref{sec:simulations}).} 
    Through simulations we illustrate our theoretical results. For example, Figure~\ref{fig:smoothness-intro} reinforces our convergence analysis from Section~\ref{sec:unknown}. It shows QQ plots over 100 simulations. With H\"{o}lder smooth nuisance functions having smoothness less than half the dimension, a DCDR estimator with undersmoothed local polynomial regressions (orange circles) approximates a normal distribution very closely while the SCDR-MSE estimator (blue triangles) does not.  Other results in simulations provide intuition for our structure-agnostic results from Section~\ref{sec:modelfree} and verify our slower-than-$\sqrt{n}$ CLT from Section~\ref{sec:known}.

    \item \textbf{Discussion and future directions (Section~\ref{sec:discussion}).} 
    We conclude by discussing practical implications of our results, extensions of our theoretical analysis to a wider class of estimators, and other avenues for future work.
\end{enumerate}

\vspace{-0.2in}
\begin{figure}[H]
    \centering
    \includegraphics[height=2.2in]{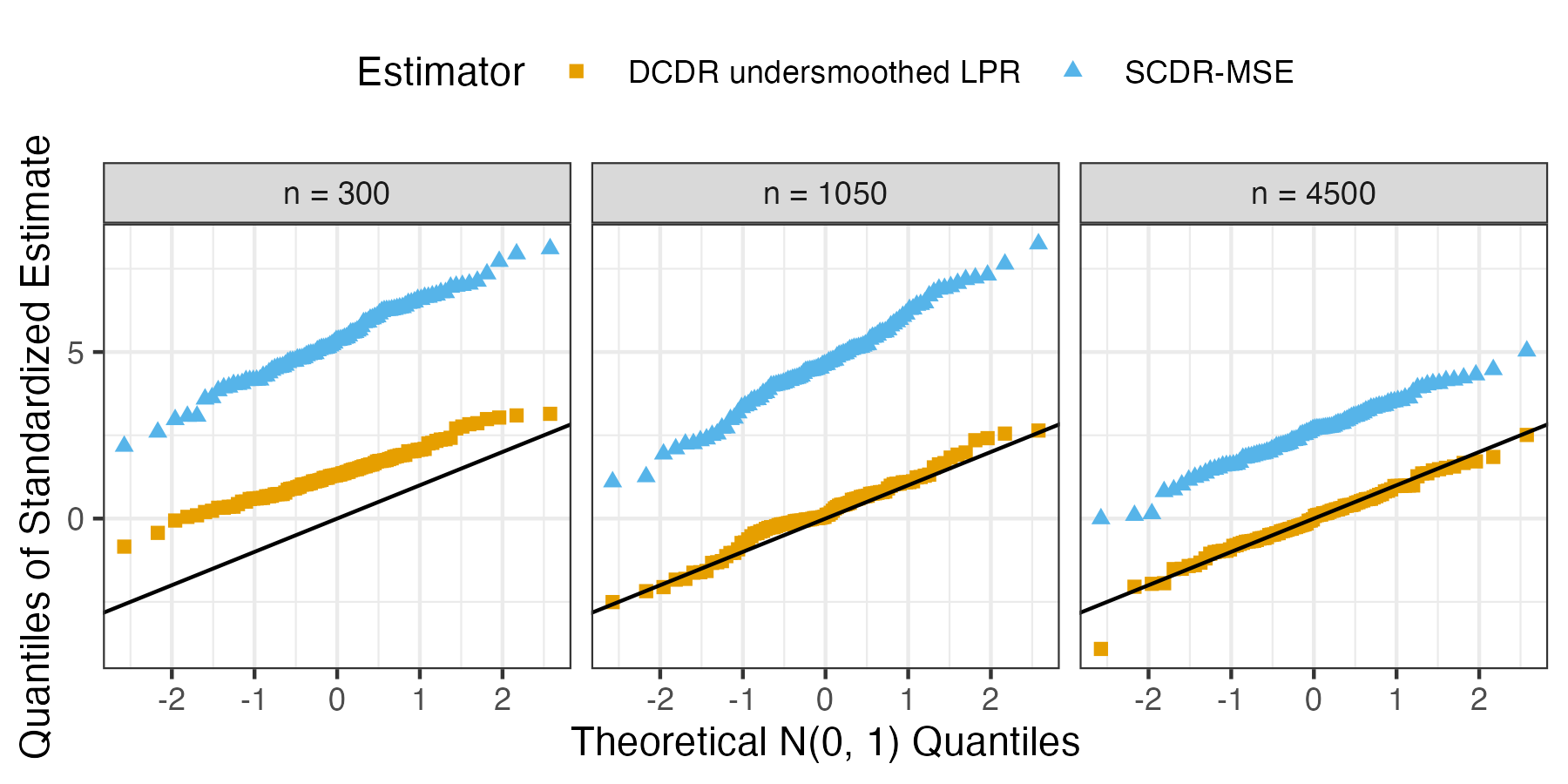}
    \caption{QQ-plots of $100$ standardized DCDR estimates with undersmoothed local polynomial regressions and $100$ standardized SCDR-MSE estimates over sample size (columns) with H\"{o}lder$(0.35)$ smooth nuisance functions and dimension $1$.}
    \label{fig:smoothness-intro}
\end{figure}


\color{black}
\subsection{Notation} \label{sec:notation}

We denote expectation by $\bbE$, variance by $\bbV$, covariance by $\cov$, and sample averages by $\bbP_n(f) = \frac{1}{n} \sum_{i=1}^{n} f(Z_i)$. For $x \in \bbR^d$, $\lVert x \rVert^2$ is the squared Euclidean norm, while $\lVert f \rVert_{\bbP}^2 = \int_{\mathcal{Z}} f(z)^2 d\bbP(z)$ and $\lVert f \rVert_\infty = \sup_{z \in \mathcal{Z}} |f(z)|$ denote the squared $L_2(\bbP)$ and supremum norms. If $\widehat f$ is an estimated function, then $\bbE \lVert \widehat f \rVert_\bbP^2$ is the expectation of $\lVert \widehat f \rVert_\bbP^2$ over the training data used to construct $\widehat f$. We use $a \lesssim b$ to mean $a \leq Cb$ for some constant $C$, and $a \asymp b$ to mean $b \lesssim a$ and $a \lesssim b$. We use $a \wedge b$ and $a \vee b$ for minimum and maximum. Convergence is denoted by $\indist$ (distribution), $\inprob$ (probability), and $\inas$ (almost sure).  Standard probabilistic order notation includes $o_\bbP(\cdot)$, $O_\bbP(\cdot)$, $o(1)$, and $O(1)$.  

\medskip

A function $f: \bbR^d \to \bbR$ is Hölder$(s)$ smooth if it is $\lfloor s \rfloor$-times continuously differentiable (where $\lfloor s \rfloor$ is the largest integer strictly smaller than $s$) with bounded partial derivatives and satisfies
$$
| D^m f(x) - D^m f(x') | \lesssim \lVert x - x' \rVert^{s - \lfloor s \rfloor}
$$
for all $x, x'$ and $m$ with $\sum_{j=1}^{d} m_j = \lfloor s \rfloor$, where $D^m$ is the multivariate partial derivative operator.

\medskip

We denote generic nuisance functions by $\eta$, datasets of $n$ observations by $D$ with subscripts (e.g., $D_\eta$ for training data for estimating $\eta$), and covariates by $X^n$ with similar subscripts.

\section{Setup and background} \label{sec:setup}

In this section, we describe the data generating process and the ECC, review known lower bounds for estimating the ECC over H\"{o}lder smoothness classes, revisit the existing literature on plug-in, doubly robust, and higher-order estimators, and explicitly define the double cross-fit doubly robust estimator for the ECC.

\medskip

We assume we observe a dataset comprising $3n$ independent and identically distributed data points $\{ Z_i \}_{i=1}^{3n}$ drawn from a distribution $\mathcal{P}$.  Here, $Z_i$ is a tuple $\{ X_i, A_i, Y_i \}$ where $X \in \bbR^d$ are covariates and $A \in \bbR$ and $Y \in \bbR$. We denote $\pi(X) = \bbE(A \mid X)$ and $\mu(X) = \bbE(Y \mid X)$ and collectively refer to them as nuisance functions.  In causal inference, often $A$ denotes binary treatment status, while $Y$ is the outcome of interest.  In that case, $\pi$ is referred to as the propensity score. Typically, $\bbE(Y \mid A=a, X)$ is referred to as the outcome regression, but we will refer to $\mu$ as the outcome regression function. 

\medskip
\color{black}
We focus on estimating the ECC:
$$
\psi_{ecc} = \bbE \{ \cov (A, Y \mid X) \} = \bbE(AY) - \bbE \{ \pi(X) \mu(X) \}.
$$
The ECC appears in the causal inference literature in the numerator of the variance weighted average treatment effect \citep{li2011higher}, as a measure of causal influence \citep{diaz2023nonagency}, and in derivative effects under stochastic interventions \citep{zhou2022marginal}.  Additionally, the ECC has appeared in the conditional independence testing literature \citep{shah2020hardness}.  Prior work on optimal DCDR estimators has also focused on the ECC \citep{mcgrath2022undersmoothing, newey2018cross, fisher2023threeway}.

\begin{remark} \label{rem:cycling}
    For our theoretical analysis, we assume there are $3n$ observations in total so we have $n$ observations for each independent fold.  When estimating the ECC with the DCDR estimator, we split the data into three folds: two for training and one for estimation.  Since our focus is on asymptotic rates, we ignore the constant factor lost from splitting the data.  But, with iid data, one can cycle the folds, repeat the estimation, and take the average to retain full sample efficiency. Indeed, our simulation results in Section~\ref{sec:simulations} illustrate such an approach.
\end{remark}

\subsection{Assumptions and lower bounds on estimation rates}

We start with the two assumptions we impose throughout.

\begin{assumption} \label{asmp:dgp} \textbf{(Bounded first and second moments for $A$ and $Y$)} $\mu(X)$ and $\pi(X)$ satisfy $|\mu(X)| < \infty, |\pi(X)| < \infty$, and the conditional second moments of $A$ and $Y$ are bounded above and below; i.e, $0 < \bbV (A \mid X = x), \bbV(Y \mid X = x) < \infty$ for all $x \in \mathcal{X}$.
\end{assumption}

\begin{assumption} \label{asmp:bdd_density} \textbf{(Bounded covariate density)} The covariates $X$ are continuous and have support $\mathcal{X}$, a compact subset of $\bbR^d$, and the covariate density $f(x)$ satisfies $0 < c \leq f(x) \leq C < \infty$ for all $x \in \mathcal{X}$.
\end{assumption}

 While we focus on continuous $X$ with density relative to the Lebesgue measure, our approach can be straightforwardly extended to discrete covariates. For discrete covariates, one can construct a separate DCDR estimator for each covariate value and then aggregate across these values, weighting by their estimated probabilities. These probabilities can be estimated at a $\sqrt{n}$-rate using a simple count estimator. 

\medskip

\color{black}
We require no further assumptions until Section~\ref{sec:unknown}. In Sections~\ref{sec:unknown} and~\ref{sec:known}, we analyze the DCDR estimator when the data generating process satisfies $\pi \in \text{H\"{o}lder}(\alpha)$ and $\mu \in \text{H\"{o}lder}(\beta)$.  Under H\"{o}lder smoothness, and when the covariate density is sufficiently smooth, \citet{robins2008higher} and \citet{robins2009semiparametric} proved that the minimax rate satisfies
\begin{equation} \label{eq:minimax_rate}
    \inf_{\widehat \psi} \sup_{\mathcal{P}_{\alpha, \beta}} \bbE | \widehat \psi - \psi_{ecc} | \gtrsim \begin{cases} n^{-1/2} &\text{ if } \frac{\alpha + \beta}{2} > d / 4, \\
        n^{-\frac{2 \alpha + 2\beta}{2 \alpha + 2 \beta + d}} &\text{otherwise.} \end{cases}
\end{equation}
The minimax rate exhibits an ``elbow'' phenomenon: $\sqrt{n}$-convergence is possible when the average smoothness of the nuisance functions is larger than one quarter the dimension; otherwise, the lower bound on the minimax rate is slower than $\sqrt{n}$ and depends on the average smoothness of the nuisance functions and the dimension of the covariates.  Importantly, these rates depend on the covariate density being smooth enough that it does not affect the estimation rate; when the covariate density is non-smooth, minimax rates for the ECC are not yet known.

\subsection{Plug-in, doubly robust, and higher-order estimators}

In this section, we describe plug-in, doubly robust, and higher-order estimators in further detail.  A plug-in estimator for the ECC can be constructed based on the representation
$$
\bbE \{ \cov (A, Y \mid X) \} = \bbE(AY) - \bbE \{ \pi(X) \mu(X) \}
$$
or
$$
\bbE \{ \cov (A, Y \mid X) \} = \bbE \big[ A \{ Y - \mu(X) \} \big].
$$
In either case, an estimator can be constructed according to the ``plugin principle'', by plugging in estimates for the relevant nuisance functions and taking the empirical average.  These estimators are often intuitive and easy to construct and when the nuisance functions are H\"{o}lder smooth and the estimators are appropriately undersmoothed they can be rate-optimal \citep{mcgrath2022undersmoothing}. However, without additional structure and careful undersmoothing, they can inherit biases from their nuisance function estimators. This has inspired an extensive literature on doubly robust estimators, which are also referred to as ``first-order'', ``double machine learning'', or ``one-step'' estimators.

\medskip

Doubly robust estimators are based on semiparametric efficiency theory and the efficient influence function (EIF), which acts like a functional derivative in the first-order von Mises expansion of the functional \citep{van1996weak, tsiatis2006semiparametric}.  For the ECC, the un-centered EIF is 
\begin{equation} \label{eq:ecc_eif}
    \varphi(Z) = \{ A - \pi(X) \} \{ Y - \mu(X) \}.
\end{equation}
The doubly robust estimator is constructed by estimating the nuisance functions, plugging their values into the formula for the un-centered EIF, and taking the empirical average:
$$
\widehat \psi_{dr} = \bbP_n \left[ \{ A - \widehat \pi(X) \} \{ Y - \widehat \mu(X) \} \right].
$$
Other doubly robust estimators such as the targeted maximum likelihood estimator are also common in the literature \citep{vanderlaan2011targeted}. They provide similar asymptotic guarantees as the doubly robust estimator, and are often referred to as ``doubly robust'' when their bias can be bounded by the product of the root mean squared errors of the nuisance function estimators under only mild regularity conditions. Doubly robust estimators are typically combined with two extra steps: (1) the nuisance estimators are constructed on a separate sample from that used to evaluated $\widehat \psi_{dr}$, and (2) the MSE of the nuisance estimates is minimized. We refer to this approach as the single cross-fit doubly robust-MSE (SCDR-MSE) estimator. It has strong error guarantees. Indeed, it is the optimal estimator when only MSE rates can be guaranteed for the nuisance estimators \citep{balakrishnan2023fundamental, jin2024structure}.
 However, when additional structure like H\"{o}lder smoothness is available, then better SCDR estimators can be constructed by undersmoothing the nuisance estimators (rather than minimizing MSE); see \citet{mcgrath2022undersmoothing} for a comprehensive analysis. \color{black}

\medskip

Higher-order estimators are based on a higher-order von Mises expansion of the functional of interest \citep{robins2008higher, li2011higher}.  Just as doubly robust estimators correct the bias of plug-in estimators, higher-order estimators correct the bias of doubly robust estimators.  For the ECC, the second-order estimator is
$$
\widehat \psi_{hoif} = \widehat \psi_{dr} - \frac{1}{n(n-1)} \sum_{i \neq j} \left\{ A_i - \widehat \pi(X_i) \right\} b(X_i)^T \widehat \Sigma^{-1} b(X_j)  \left\{ Y_j - \widehat \mu(X_j) \right\}
$$
where $b(X)$ is a basis with dimension growing with sample size and $\widehat \Sigma = \bbP_n \{ b(X) b(X)^T \}$ is the Gram matrix.  Higher-order estimators capitalize on the additional structure available when the nuisance functions are smooth, enabling them to achieve the minimax rate in some settings \citep{robins2008higher, robins2009semiparametric}. Recent research has developed adaptive and more numerically stable extensions of higher-order estimators \citep{liu2021adaptive, liu2023new}.

\subsection{Double cross-fit doubly robust estimator}

We focus on a double cross-fit doubly robust (DCDR) estimator, which is a simple adaptation of the SCDR estimator, whereby the nuisance estimators are trained on separate independent samples.

\begin{algorithm} \label{alg:dcdr} \textbf{\emph{(DCDR Estimator for the ECC)}}
    Let $(D_\mu, D_\pi, D_\varphi)$ denote three independent samples of $n$ observations of $Z_i = (X_i, A_i, Y_i)$. Then:
    \begin{enumerate}
        \item Train an estimator $\widehat \mu$ for $\mu$ on $D_\mu$ and train an estimator $\widehat \pi$ for $\pi$ on $D_\pi$.
        \item On $D_\varphi$, estimate the un-centered efficient influence function values $\widehat \varphi (Z) = \{ A - \widehat \pi(X) \} \{ Y - \widehat \mu(X) \}$ using the estimators from step 1, and construct the DCDR estimator $\widehat \psi_n$ as the empirical average of $\widehat \varphi(Z)$ over the estimation data  $D_\varphi$: 
        $$
        \widehat \psi_n = \bbP_{n} \{ \widehat \varphi(Z) \} \equiv \frac1n \sum_{Z_i \in D_\varphi} \widehat \varphi(Z_i).
        $$
    \end{enumerate}
\end{algorithm}

This estimator has received some attention in prior work: \citet{newey2018cross} first proposed it and combined it with regression splines nuisance estimators, showing that the resulting DCDR estimator can be $\sqrt{n}$-consistent under minimal smoothness conditions.  \citet{fisher2023threeway} and \citet{kennedy2023towards} extended the approach to estimate heterogeneous effect estimation, while \citep{mcgrath2022undersmoothing} developed a comprehensive analysis of the DCDR estimator with orthogonalized wavelet nuisance estimators under smoothness assumptions. Nonetheless, as we outlined in the introduction, there are several questions remaining about the properties of the DCDR estimator. The rest of this paper analyzes the DCDR estimator in detail.


\color{black}
\section{Structure-agnostic analysis} \label{sec:modelfree}

In this section, we derive a structure-agnostic asymptotically linear expansion for the DCDR estimator which holds with generic nuisance functions and estimators.  To the best of our knowledge, this is the first such structure-agnostic analysis. Then, we provide a nuisance-function-agnostic decomposition of the remainder term from the asymptotically linear expansion.  Finally, we discuss, informally, how these results reveal that undersmoothing the nuisance function estimators can lead to faster convergence rates for the DCDR estimator. 

\medskip

Our first result is a structure-agnostic asymptotically linear expansion of the DCDR estimator.  It does not require any assumptions about the nuisance functions or their estimators beyond Assumptions~\ref{asmp:dgp} and~\ref{asmp:bdd_density}.  

\begin{restatable}{lemma}{lemexpansion}\label{lem:expansion} \textbf{\emph{(Structure-agnostic linear expansion)}}
    Under Assumptions~\ref{asmp:dgp} and~\ref{asmp:bdd_density}, if $\psi_{ecc}$ is estimated with the DCDR estimator $\widehat \psi_n$ from Algorithm~\ref{alg:dcdr}, then
    \begin{align*}
        \widehat \psi_n - \psi_{ecc} &= (\bbP_{n} - \bbP) \{ \varphi(Z) \} + R_{1,n} + R_{2,n} \\
        \text{ where } R_{1,n} &\leq \lVert b_\pi \rVert_\bbP \lVert b_\mu \rVert_\bbP \text{ and } R_{2,n} = O_{\bbP} \left( \sqrt{\frac{\bbE \lVert \widehat \varphi - \varphi \rVert_\bbP^2 + \rho (\Sigma_n)}{n} } \right), 
    \end{align*}
    $b_\eta \equiv b_\eta(X) = \bbE \{ \widehat \eta(X) - \eta(X) \mid X \}$ is the pointwise bias of the estimator $\widehat \eta$, $\rho(\Sigma_n)$ denotes the spectral radius of $\Sigma_n$, and
    $$
    \Sigma_n = \bbE \left( \cov \left[ \Big\{  \widehat b_\varphi (X_1) , ...,  \widehat b_\varphi (X_n) \Big\}^T \mid X_\varphi^n \right] \right) 
    $$
    where $\widehat b_{\varphi} (X_i) = \bbE \{ \widehat \varphi(Z_i) - \varphi(Z_i) \mid X_i, D_\pi, D_\mu \}$ is the conditional bias of $\widehat \varphi$ and $X_\varphi^n$ denotes the covariates in the estimation sample.
\end{restatable}

All proofs are delayed to the appendix.  Here, we provide some intuition for the result.  Crucially, the proof of Lemma~\ref{lem:expansion} analyzes the randomness of the DCDR estimator over \emph{both the estimation and training data}.  By contrast, the analysis of the SCDR estimator is usually conducted \emph{conditionally on the training data}. The unconditional analysis of the DCDR estimator allows us to leverage the independence of the training samples, thereby bounding the bias of the DCDR estimator by the product of integrated biases of the nuisance function estimators.  Without accounting for the randomness over the training data, this is not possible. Therefore, conditional on the training data, the DCDR estimator would only have the same guarantees as the SCDR estimator. However, the unconditional analysis also requires accounting for the covariance over the training data between summands of the DCDR estimator because, without conditioning on the training data, the nuisance function estimators are random, and $\widehat \varphi(Z_i) \cancel{\ind} \widehat \varphi(Z_j)$ and $\cov_{i \neq j} \{ \widehat \varphi(Z_i), \widehat \varphi(Z_j) \} \neq 0$.  These non-zero covariances are accounted for by the new spectral radius term in the second remainder term, $\rho(\Sigma_n)$, which we analyze in further detail in Proposition~\ref{prop:spectral}.

\medskip

Lemma~\ref{lem:expansion} is useful because of its generality, and we use it throughout the rest of the paper.  Beyond Assumptions~\ref{asmp:dgp} and~\ref{asmp:bdd_density}, Lemma~\ref{lem:expansion} requires no assumptions for the nuisance functions or their estimators.  This is in contrast to previous results, which focus on specific linear smoothers for the nuisance function estimators \citep{newey2018cross, mcgrath2022undersmoothing, kennedy2023towards, fisher2023threeway}.  In Section~\ref{sec:unknown}, we use Lemma~\ref{lem:expansion} to analyze the DCDR estimator with linear smoothers.  Before that, we analyze the spectral radius term in Lemma~\ref{lem:expansion} without assuming any structure on the nuisance functions or their estimators, but leveraging the specific structure of the ECC.

\begin{remark}
    \citet{mcgrath2022undersmoothing} improved upon the bias term in Lemma~\ref{lem:expansion} using special properties of wavelet estimators, and the bias of their estimator scales like the minimum of two bias products. We demonstrate that a similar phenomenon occurs for local polynomial regression in Section~\ref{sec:known}. 
\end{remark}

\begin{restatable}{proposition}{propspectral} \label{prop:spectral}
    \emph{\textbf{(Spectral radius bound)}} Under Assumptions~\ref{asmp:dgp} and~\ref{asmp:bdd_density},  if $\psi_{ecc}$ is estimated with the DCDR estimator $\widehat \psi_n$ from Algorithm~\ref{alg:dcdr}, then 
    \begin{align*}
        \frac{\rho(\Sigma_n)}{n} \leq \frac{\bbE \lVert \widehat \varphi - \varphi \rVert_\bbP^2}{n} &+ \left( \lVert b_\pi^2 \rVert_\infty + \lVert s_\pi^2 \rVert_\infty \right) \bbE \Big[ \big| \cov \{ \widehat \mu(X_i), \widehat \mu(X_j) \mid X_i, X_j \} \big| \Big] \\
        &+ \left( \lVert b_\mu^2 \rVert_\infty + \lVert s_\mu^2 \rVert_\infty \right) \bbE \Big[ \big| \cov \{ \widehat \pi(X_i), \widehat \pi(X_j) \mid X_i, X_j \} \big| \Big] 
    \end{align*}
    where $\lVert b_\eta^2 \lVert_\infty = \sup_{x \in \mathcal{X}} \bbE \{ \widehat \eta(X) - \eta(X) \mid X = x \}^2$ and $\lVert s_\eta^2 \rVert_\infty = \sup_{x \in \mathcal{X} } \bbV \{ \widehat \eta(X) \mid X = x\}$ are uniform squared bias and variance bounds.
\end{restatable}

Here, we describe Proposition~\ref{prop:spectral} in further detail.  The first term on the right hand side comes from the diagonal of $\Sigma_n$, and is equal to the variance terms already observed in Lemma~\ref{lem:expansion}.  The second and third terms come from the off-diagonal terms in $\Sigma_n$.  The expected absolute covariance, $\bbE \Big[ \big| \cov \{ \widehat \eta(X_i), \widehat \eta(X_j) \mid X_i, X_j \} \big| \Big]$, measures the covariance over the training data of an estimator's predictions at two \emph{independent} test points.  For many estimators, we anticipate that $\bbE \Big[ \big| \cov \{ \widehat \eta(X_i), \widehat \eta(X_j) \mid X_i, X_j \} \big| \Big] \lesssim n^{-1}$. In Section~\ref{sec:unknown}, we demonstrate this to be the case for k-Nearest Neighbors and local polynomial regression. In Appendix~\ref{app:series}, we demonstrate this result for regression splines and orthogonalized wavelet estimators.  It is not immediately clear whether this result can be established for general classes of machine leaning estimators. Nonetheless, in Appendix~\ref{app:random-forest} we make a first step in this direction and establish it holds, up to a polylog factor, for a centered random forest estimator \citep{biau2012analysis}. Centered random forests differ from Breiman’s original random forest proposal \citep{breiman2001random} and from random forests typically used in practice. The key distinction is that tree partitions are constructed independently of the data, greatly simplifying the theoretical analysis. Similar simplifications are common in the theoretical literature (see, e.g., \citet{biau2016random} for an overview), but extending such results to random forests commonly implemented in practice remains substantially more challenging and represents an exciting direction for future research.

\medskip

\color{black}
Like Lemma~\ref{lem:expansion}, Proposition~\ref{prop:spectral} is useful because of its generality: it applies to any nuisance functions and nuisance function estimators.  Although Proposition~\ref{prop:spectral} relies specifically on the functional being the ECC, we anticipate that similar results apply for other functionals.  

\medskip

Further investigation of Proposition~\ref{prop:spectral} reveals when undersmoothing the nuisance function estimators will lead to the fastest convergence rate.  The EIF of the ECC, like many functionals, is Lipschitz in terms of its nuisance functions, so $\widehat \varphi - \varphi \lesssim | \widehat \pi - \pi | + | \widehat \mu - \mu |$ and $\lVert \widehat \varphi - \varphi \rVert_\bbP \lesssim \lVert \widehat \pi - \pi \rVert_\bbP + \lVert \widehat \mu - \mu \rVert_\bbP$.  Moreover, the compactness of the support of $X$ in Assumption~\ref{asmp:bdd_density} implies that the supremum mean squared errors of the nuisance function estimators scale at the typical pointwise rate. Therefore, if the expected covariance term scales inversely with sample size such that $\bbE \Big[ \big| \cov \{ \widehat \eta(X_i), \widehat \eta(X_j) \mid X_i, X_j \} \big| \Big] = O_\bbP( n^{-1} )$, then $\frac{\rho(\Sigma_n)}{n} = O_\bbP \left(\frac{\bbE \lVert \widehat \varphi - \varphi \rVert_\bbP^2}{n} \right) = O_\bbP\left(\frac{ \lVert b_{\pi}^2 \rVert_\infty + \lVert s_{\pi}^2 \rVert_\infty +   \lVert b_{\mu}^2 \rVert_\infty + \lVert s_{\mu}^2 \rVert_\infty }{n} \right)$, and so
	
\begin{equation} \label{eq:remainder}
    R_{2,n} = O_{\bbP} \left(  \sqrt{\frac{ \lVert b_{\pi}^2 \rVert_\infty + \lVert s_{\pi}^2 \rVert_\infty +   \lVert b_{\mu}^2 \rVert_\infty + \lVert s_{\mu}^2 \rVert_\infty  }{n} } \right). 
\end{equation}

Balancing $R_{2,n}$ in \eqref{eq:remainder} with the bias $R_{1,n}$ in Lemma~\ref{lem:expansion} requires constructing nuisance function estimators such that $\lVert b_\pi \rVert_\bbP^2 \lVert b_\mu \rVert_\bbP^2 \asymp  \frac{\lVert s_{\pi}^2 \rVert_\infty + \lVert s_{\mu}^2 \rVert_\infty}{n}$.  A natural way to achieve such a balance is by undersmoothing both $\widehat \pi$ and $\widehat \mu$ so their squared bias is smaller than their variance. 

\medskip

In this section, we have demonstrated a structure-agnostic linear expansion for the DCDR estimator and presented a nuisance-function-agnostic decomposition of its remainder term. In the next section, we assume the nuisance functions are H\"{o}lder smooth and construct DCDR estimators with local averaging linear smoothers, and we use Lemma~\ref{lem:expansion} and Proposition~\ref{prop:spectral} to demonstrate the DCDR estimator's efficiency guarantees.

\begin{remark}
    An important question is whether the results in this section have practical implications. For brevity, we defer further details to Appendix~\ref{app:practical}, where we investigate when and how one might conduct undersmoothing with generic machine learning estimators. There, we further develop the intuition from \eqref{eq:remainder} and observe that if $\bbE \Big[ \big| \cov \{ \widehat \eta(X_i), \widehat \eta(X_j) \mid X_i, X_j \} \big| \Big] \lesssim n^{-1}$, other mild regularity conditions are satisfied, and the nuisance estimators have monotone bias-variance tradeoffs in terms of their tuning parameters (e.g., increasing a tuning parameter always decreases bias and increases variance) then these results imply that \emph{undersmoothing the nuisance estimators as much as possible} leads to the fastest convergence rate for the DCDR estimator. 
\end{remark}


\color{black}
\section{H\"{o}lder smoothness and local averaging estimators} \label{sec:unknown}

In this section, we assume the nuisance functions are H\"{o}lder smooth and construct DCDR estimators without requiring knowledge of the smoothness or covariate density.  When the nuisance functions are estimated with local polynomial regression, we show the DCDR estimator is $\sqrt{n}$-consistent and asymptotically normal under minimal conditions and, in the non-$\sqrt{n}$ regime, converges at the conjectured minimax rate with unknown and non-smooth covariate density \citep{robins2008higher}. Additionally, when the nuisance functions are estimated with k-Nearest-Neighbors, we demonstrate that the DCDR estimator is $\sqrt{n}$-consistent when the nuisance functions are H\"{o}lder smooth of order at most one and are sufficiently smooth compared to the dimension of the covariates. First, we formally state the H\"{o}lder smoothness assumptions for the nuisance functions.

\begin{assumption} \label{asmp:holder} \textbf{(H\"{o}lder smooth nuisance functions)} The nuisance functions $\pi$ and $\mu$ are H\"{o}lder smooth, with $\pi \in \text{H\"{o}lder}(\alpha)$ and $\mu \in \text{H\"{o}lder}(\beta)$.
\end{assumption}

We focus on local averaging estimators in this section, and next we review k-Nearest-Neighbors and local polynomial regression.  In Appendix~\ref{app:series}, we review series regression, and establish results like those in this section for regression splines and wavelet estimators.  Those results are already known \citep{newey2018cross, fisher2023threeway, mcgrath2022undersmoothing}, but we provide them for completeness and because we use different proof techniques from those considered previously. Moreover, in Appendix~\ref{app:random-forest}, we establish similar results for a centered random forest estimator \citep{biau2012analysis}.

\subsection{Local averaging estimators}

We define the estimators for $\mu$ using $D_\mu$. The estimators for $\pi$ follow analogously with $D_\pi$, replacing $Y$ by $A$.

\begin{estimator} \label{est:knn} \emph{\textbf{(k-Nearest-Neighbors)}}
    The k-Nearest-Neighbors estimator for $\mu(X) = \bbE(Y \mid X)$ is
    \begin{equation} \label{eq:knn}
        \widehat \mu(x) = \frac1k \sum_{Z_i \in D_\mu} \one \big( \lVert X_i - x\rVert \leq \lVert X_{(k)} (x) - x  \rVert \big) Y_i,
    \end{equation}
    where $X_{(k)} (x)$ is the $k^{\text{th}}$ nearest neighbor of $x$ in $X_\mu^n$. 
\end{estimator}

The k-Nearest-Neighbors estimator is simple.  However, as we see subsequently, it is unable to adapt to higher smoothness in the nuisance functions, as in nonparametric regression \citep{gyorfi2002distribution}.

\begin{estimator} \label{est:lpr} \emph{\textbf{(Local polynomial regression)}}
    The local polynomial regression estimator for $\mu(X) = \bbE(Y \mid X)$ is
    \begin{equation}\label{eq:lpr}
        \widehat \mu(x) = \sum_{Z_i \in D_\mu} \left\{ \frac{1}{nh^d} b(0)^T \widehat Q^{-1} b \left( \frac{X_i - x}{h} \right) K \left( \frac{X_i - x}{h} \right) \right\} Y_i
    \end{equation}
    where
    $$
    \widehat{Q} = \frac{1}{nh^d} \sum_{X_i \in X_\mu^n} b \left( \frac{X_i - x}{h} \right) K \left( \frac{X_i - x}{h} \right) b \left( \frac{X_i - x}{h} \right)^T ,
    $$
    $b: \bbR^d \to \bbR^{p}$ where $p = {d + \lceil d /2 \rceil \choose \lceil d / 2 \rceil}$ is a vector of orthogonal basis functions consisting of all powers of each covariate up to order $\lceil d/2 \rceil$ and all interactions up to degree $\lceil d/2 \rceil$ polynomials (see, \citet{masry1996multivariate}, \citet{belloni2015some} Section 3), $\lceil d / 2 \rceil$ denotes the smallest integer strictly larger than $d / 2$, $K: \bbR^d \to \bbR$ is a bounded kernel with support on $[-1, 1]^d$, and $h$ is a bandwidth parameter.  If the matrix $\widehat Q$ is not invertible, $\widehat \mu(x) = 0$.
\end{estimator}

Local polynomial regression has been extensively studied \citep{ruppert1994multivariate, masry1996multivariate, fan2018local, tsybakov2009introduction}. There are two notable features to this version of the estimator.  First, the basis is expanded to order $\lceil d/2 \rceil$, the smallest integer strictly larger than $d/2$, rather than the smoothness of the regression function.  Therefore, the estimator does not require knowledge of the true smoothness, but the expansion of the basis to degree $\lceil d /2 \rceil$ still ensures the bias of the DCDR estimator is $o_\bbP(n^{-1/2})$ in the $\sqrt{n}$-regime.  Second, the estimator is explicitly defined even when the local Gram matrix, $\widehat Q$, is not invertible --- $\widehat \mu(x) = 0$. This ensures the bias of the estimator is bounded when $\widehat Q$ is not invertible. 
    
\medskip

Unlike k-Nearest-Neighbors, local polynomial regression can optimally estimate functions of higher smoothness. In Appendix~\ref{app:nf_ests}, we provide bias and variance bounds for both estimators, which follow from standard results in the relevant literature \citep{biau2015lectures, tsybakov2009introduction, gyorfi2002distribution}.  However, two nuances arise in this analysis because the bias and variance bounds account for randomness over the training data. First, the pointwise variance, $\bbV \{ \widehat \eta(x) \}$, scales at the typical conditional (on the training data) mean squared error rate; e.g., for local polynomial regression, $\bbV \{ \widehat \mu(x) \} \lesssim h^{-2\beta} + \frac{1}{nh^d}$. It may be possible to improve this with more careful analysis, but because this will not affect the behavior of the DCDR estimator --- which uses undersmoothed nuisance function estimators --- we leave this to future work. Second, for local polynomial regression, the local Gram matrix $\widehat Q$ may not be invertible. Therefore, it is necessary to show that non-invertibility occurs with asymptotically negligible probability if the bandwidth $h$ decreases slowly enough, which is possible using a matrix Chernoff inequality (see, \citet{tropp2015introduction} Section 5).	

\medskip

Next, we show the covariance terms from Proposition~\ref{prop:spectral} can decrease inversely with sample size for both estimators, i.e., $\bbE \Big[ \big| \cov \big\{ \widehat \eta(X_i), \widehat \eta(X_j) \mid X_i, X_j  \big\}\big| \Big]$, and demonstrate the efficiency guarantees of the DCDR estimator.

\subsection{$\sqrt{n}$-consistency under minimal conditions}

The efficiency of the DCDR estimator depends on how quickly the expected absolute covariance $\bbE \Big[ \big| \cov \{ \widehat \eta(X_i), \widehat \eta(X_j) \mid X_i, X_j \} \big| \Big]$ decreases.  Therefore, first, we show that this term can decrease inversely with sample size for k-Nearest-Neighbors and local polynomial regression.

\begin{restatable}{lemma}{lemcovariance}\label{lem:covariance}
    \emph{\textbf{(Covariance bound)}} 
    Suppose Assumptions~\ref{asmp:dgp},~\ref{asmp:bdd_density}, and~\ref{asmp:holder} hold.  Moreover, assume that each estimator balances squared bias and variance or is undersmoothed.  Then, both k-Nearest-Neighbors and local polynomial regression satisfy 
    \begin{equation} \label{eq:knn_cov}
        \bbE \Big[ \big| \cov \{ \widehat \eta(X_i), \widehat \eta(X_j) \mid X_i, X_j \} \big| \Big] = O_\bbP \left( \frac1n \right)
    \end{equation}
    for $\eta \in \{ \pi, \mu \}$.
\end{restatable}

Lemma~\ref{lem:covariance} demonstrates that the expected absolute covariance can decrease inversely with sample size for both k-Nearest-Neighbors and local polynomial regression. The result follows from a localization argument --- if the estimation points $X_i$ and $X_j$ are well separated, then $\widehat \eta(X_i)$ and $\widehat \eta(X_j)$ share no training data and therefore their covariance is zero; otherwise, the covariance is upper bounded by the variance. Lemma~\ref{lem:covariance} guarantees that the expected absolute covariance decreases inversely with sample size if the estimators balance squared bias and variance or are undersmoothed. It may be possible to improve this result so that it also applies to oversmoothed estimators, but because we focus only on undersmoothed nuisance function estimators subsequently, we leave that to future work. 

\medskip

The following result establishes that the DCDR estimator achieves $\sqrt{n}$-consistency and asymptotic normality under minimal conditions and fast convergence rates in the non-$\sqrt{n}$ regime.

\begin{restatable}{theorem}{thmsemiparametric} \label{thm:semiparametric} \emph{\textbf{(Convergence guarantees)}}
    Suppose Assumptions~\ref{asmp:dgp},~\ref{asmp:bdd_density}, and~\ref{asmp:holder} hold, and $\psi_{ecc}$ is estimated with the DCDR estimator $\widehat \psi_n$ from Algorithm~\ref{alg:dcdr}. If the nuisance functions $\widehat \mu$ and $\widehat \pi$ are estimated with local polynomial regression (Estimator~\ref{est:lpr}) with bandwidths satisfying $h_\mu, h_\pi \asymp \left( \frac{n}{\log n} \right)^{-1/d}$, then 
    \begin{equation}
        \begin{cases}
            \sqrt{\frac{n}{\bbV \{ \varphi(Z) \}}} (\widehat \psi_n - \psi_{ecc}) \indist N(0, 1) &\text{ if } \frac{\alpha + \beta}{2} > d/4\text{, and} \\[10pt]
            \bbE | \widehat \psi_n - \psi_{ecc} | = O_\bbP \left( \frac{n}{\log n} \right)^{-\frac{\alpha + \beta}{d}} &\text{ otherwise.}
        \end{cases}
    \end{equation}
    If the nuisance functions $\widehat \mu$ and $\widehat \pi$ are estimated with k-Nearest-Neighbors (Estimator~\ref{est:knn}) and $k_\mu, k_\pi \asymp \log n$, then
    \begin{equation}
        \begin{cases}
            \sqrt{\frac{n}{\bbV \{ \varphi(Z) \}}} (\widehat \psi_n - \psi_{ecc}) \indist N(0, 1) &\text{ if } \frac{\alpha + \beta}{2} > d/4 \text{ and } \alpha, \beta \leq 1\text{, and} \\[10pt]
            \bbE | \widehat \psi_n - \psi_{ecc} | \lesssim \left( \frac{n}{\log n} \right)^{-\frac{(\alpha \wedge 1) + (\beta \wedge 1)}{d}} &\text{ otherwise.}
        \end{cases}
    \end{equation}
\end{restatable}

Theorem~\ref{thm:semiparametric} shows that the DCDR estimator with undersmoothed local polynomial regression is $\sqrt{n}$-consistent and asymptotically normal under minimal conditions.  Further, it attains (up to a log factor) the convergence rate $n^{-\frac{\alpha + \beta}{d}}$ in probability in the non-$\sqrt{n}$ regime.  This is slower than the known lower bound for estimating the ECC when the covariate density is appropriately smooth, but has been conjectured to be the minimax rate when the covariate density is non-smooth \citep{robins2009semiparametric}.  A similar but weaker result holds for k-Nearest-Neighbors estimators, whereby the DCDR estimator achieves $\sqrt{n}$-consistency and asymptotic normality when the nuisance functions are H\"{o}lder smooth of order at most one but are sufficiently smooth compared to the dimension of the covariates. A simple example is if the nuisance functions are Lipschitz (i.e., $\alpha = \beta = 1$) and the dimension of the covariates is less than four ($d < 4$).  

\medskip

The DCDR estimator based on local polynomial regression in Theorem~\ref{thm:semiparametric} is not minimax optimal because the bandwidth is constrained so that the local Gram matrix is invertible with high probability, thereby limiting the convergence rate of the bias of the local polynomial regression estimators and, by extension, the bias of the DCDR estimator.  By replacing the Gram matrix with its expectation (assuming it is known), an estimator could be undersmoothed even further for a faster bias convergence rate. In the next section we propose such an estimator --- the ``covariate-density-adapted'' kernel regression.  We illustrate that the DCDR estimator with covariate-density-adapted kernel regression can be minimax optimal.  Moreover, we establish asymptotic normality in the non-$\sqrt{n}$ regime by undersmoothing the DCDR estimator so its variance dominates its squared bias, but it converges to a normal limiting distribution around the ECC at a slower-than-$\sqrt{n}$ rate.

\begin{remark}
    When the DCDR estimator achieves $\sqrt{n}$-consistency and asymptotic normality, Slutsky's theorem and Theorem~\ref{thm:semiparametric} imply that inference can be conducted for the ECC with Wald-type $1-\alpha$ confidence intervals, $\widehat \psi_n \pm \Phi^{-1}(1-\alpha/2) \sqrt{\frac{\widehat{\bbV}\{\varphi(Z)\}}{n}},$ where $\widehat{\bbV}\{\varphi(Z)\}$ is any consistent estimator for $\bbV \{ \varphi(Z) \}$ (e.g., the sample variance of $\widehat \varphi(Z)$).
\end{remark}

\begin{remark} \label{rem:scaling}
    Although the primary contribution of this analysis is theoretical, Theorem~\ref{thm:semiparametric} (along with related results for series regression discussed in Appendix~\ref{app:series} and in \citet{mcgrath2022undersmoothing,newey2018cross}) carries practical implications. Specifically, the nuisance estimators we consider---including regression splines and orthogonalized wavelet estimators---satisfy several beneficial properties, further investigated in Appendix~\ref{app:practical}. Consequently, achieving the fastest convergence rate for the DCDR estimator corresponds to undersmoothing the nuisance estimators as aggressively as possible (i.e., choosing bandwidth $h \asymp n^{-1/d}$). It is possible to do this in a principled manner. For instance, with local polynomial regression, choose a small fixed (with sample size) number of neighboring training points to construct an estimate at each test point.  In our simulations, we adopt this approach. We select an adaptive bandwidth based on the distance to the $10^{th}$ nearest neighbor in the training data, fixing $k=10$ across sample sizes. In simulations, this enabled $\sqrt{n}$-convergence and asymptotic normality when the ratio of smoothness to dimension is $0.35$, as illustrated in Figure~\ref{fig:smoothness-intro}.
\end{remark}

\section{Minimax optimality and asymptotic normality in the non-$\sqrt{n}$ regime} \label{sec:known}

In this section, we assume the covariate density is known and examine the behavior of the DCDR estimator with covariate-density-adapted kernel regression estimators for the nuisance functions.  For the results in this section, we require, in addition to previous assumptions, that the covariate density is known and sufficiently smooth.

\begin{assumption} \label{asmp:density_smooth} \textbf{(Known, lower bounded, and smooth covariate density)} The covariate density $f$ is known and $f \in \text{H\"{o}lder}(\gamma)$, where $\gamma \geq \alpha \vee \beta$. 
\end{assumption}

Under Assumption~\ref{asmp:density_smooth}, we demonstrate the DCDR estimator is minimax optimal.  First, we define the covariate-density-adapted kernel regression estimator:
\begin{estimator} \label{est:cdalpr} \emph{\textbf{(Covariate-density-adapted kernel regression)}} The covariate-density-adapted kernel regression estimator for $\mu(X) = \bbE(Y \mid X)$ is 
\begin{equation} \label{eq:cdalpr}
    \widehat \mu(x) = \sum_{Z_i \in D_\mu}  \frac{K_\mu \left( \frac{X_i - x}{h_\mu} \right)}{nh_\mu^d f(X_i)} Y_i,
\end{equation}
where $h_\mu$ is the bandwidth and $K_\mu$ is a kernel (to be chosen subsequently). The estimator for $\pi(X) = \bbE(A \mid X)$ is defined analogously on $D_\pi$.
\end{estimator}

This estimator uses the known covariate density in the denominator of \eqref{eq:cdalpr}. As a result, no constraint on the bandwidth is required, and the estimator can be undersmoothed more than the local polynomial regression estimator in Estimator~\ref{est:lpr}. \citet{mcgrath2022undersmoothing} proposed a similar adaptation of an orthogonalized wavelet estimator.  As they showed for the wavelet estimator, the known covariate density in Estimator~\ref{est:cdalpr} could be replaced by the estimated covariate density, and our subsequent results would follow if the covariate density were sufficiently smooth (smoother than in Assumption~\ref{asmp:density_smooth}) and its estimator sufficiently accurate. Other work has considered the setting where one has access to an auxiliary ``unsupervised'' dataset of only covariates where one could construct an accurate estimator of the covariate density, which is an adaptation that could be useful in practice \citep{liu2020nearly}. However, because the properties of the resulting DCDR estimator are not well understood when the covariate density is not sufficiently smooth, we leave analyzing estimators incorporating the estimated covariate density to future work. 

\medskip
\color{black}
The subsequent analysis combines two versions of covariate-density-adapted kernel regression, with different kernels.
\begin{subestimator} \label{est:cdalpr_higher} \textbf{\emph{(Higher-order covariate-density-adapted kernel regression)}} 
    The higher-order covariate-density-adapted kernel regression has symmetric and bounded kernel $K$ that is of order $\lceil \alpha + \beta \rceil$ and satisfies $K(x/h) \lesssim \one( \lVert x \rVert \leq h)$, $\int K(x) dx = 1$, $\int K(x)^2 dx \asymp 1$, and $\int \lVert x \rVert^{\alpha + \beta} K(x)dx \lesssim 1$ \citep{gyorfi2002distribution, tsybakov2009introduction}.
\end{subestimator} 
This version of the estimator uses a higher-order localized kernel, which allows it to adapt to the sum of the smoothnesses of the nuisance functions.  See, e.g., Section 5.3, \citet{gyorfi2002distribution} and Section 1.2.2, \citet{tsybakov2009introduction} for a review of higher-order kernels and how to construct bounded kernels of arbitrary order.  To complement this estimator, we require a smooth estimator.
\begin{subestimator} \label{est:cdalpr_smooth} \emph{\textbf{(Smooth covariate-density-adapted kernel regression)}}
    The smooth covariate-density-adapted kernel regression has continuous and bounded kernel $K$ satisfying $K(x/h) \lesssim \one \left(\lVert x \rVert \leq h \right), \int K(x) dx = 1$, $\int K(x)^2 dx \asymp 1$.
\end{subestimator}
Because the kernel in the smooth estimator is localized and \emph{continuous}, it allows the DCDR estimator to adapt to the sum of smoothnesses of the nuisance functions through the higher-order kernel estimator. For this purpose, the smooth kernel must be continuous, but need not control higher-order bias terms. Therefore, a simple kernel is adequate, such as the Epanechnikov kernel --- $K(x) = \frac{3}{4} \left(1 - \lVert x \rVert^2 \right) \one \left(\lVert x \rVert \leq 1 \right)$. 

\subsection{Minimax optimality}

The following result shows that the DCDR estimator using covariate-density-adapted kernel regression estimators is minimax optimal.  

\begin{restatable}{theorem}{thmminimax}\label{thm:minimax} \emph{\textbf{(Minimax optimality)}} 
    Suppose Assumptions~\ref{asmp:dgp},~\ref{asmp:bdd_density},~\ref{asmp:holder}, and~\ref{asmp:density_smooth} hold.  If $\psi_{ecc}$ is estimated with the DCDR estimator $\widehat \psi_n$ from Algorithm~\ref{alg:dcdr}, one nuisance function is estimated with the smooth covariate-density-adapted kernel regression (Estimator~\ref{est:cdalpr_smooth}) with bandwidth decreasing at any rate such that the estimator is consistent, and the other nuisance function is estimated with the higher-order covariate-density-adapted kernel regression (Estimator \ref{est:cdalpr_higher}) with bandwidth that scales at $n^{\frac{-2}{2\alpha + 2\beta + d}}$, then
    \begin{equation}
        \begin{cases}
            \sqrt{\frac{n}{\bbV \{ \varphi(Z) \}}} (\widehat \psi_n - \psi_{ecc}) \indist N(0, 1) &\text{ if } \frac{\alpha + \beta}{2} > d/4, \\[10pt]
            \bbE | \widehat \psi_n - \psi_{ecc} | = O_\bbP \left( n^{-\frac{2\alpha + 2\beta}{2\alpha + 2\beta + d}} \right) &\text{ otherwise.}
        \end{cases}
    \end{equation}
\end{restatable}

Theorem~\ref{thm:minimax} establishes that the DCDR estimator with covariate-density-adapted kernel regression estimators is $\sqrt{n}$-consistent and asymptotically normal under minimal conditions and minimax optimal in the non-$\sqrt{n}$ regime.  The result relies on knowledge of the smoothness of the nuisance functions, as well as shrinking one of the two bandwidths faster than $n^{-1/d}$. The proof relies on the smoothing properties of convolutions and an adaptation of Theorem 1 from \citet{gine2008simple}, as well as results from \citet{gine2008uniform} and Chapter 4 of \citet{gine2021mathematical}.  While Theorem~\ref{thm:minimax} is the first result applied to local averaging estimators such as kernel regression, \citet{mcgrath2022undersmoothing} proved the same result using approximate wavelet kernel projection estimators for the nuisance functions. Their result relies on the orthogonality (in expectation) of the wavelet estimator's predictions and residuals.  

\begin{remark}
    To guarantee asymptotic normality in the $\sqrt{n}$-regime, it is necessary that the smooth covariate-density-adapted estimator is consistent. If one were only interested in convergence rates, as in \citet{mcgrath2022undersmoothing}, one could replace the smooth estimator by any smooth estimator with bounded variance. Indeed, supposing without loss of generality that $\widehat \mu$ were the higher-order kernel estimator, one could set $\widehat \pi = 0$ and  instead implement the plug-in estimator for the ECC from \citet{newey2018cross}, given by $\widehat \psi = \mathbb{P}_n \big[ A \big\{ Y- \widehat \mu(X) \big\} \big]$. This plug-in approach requires only single cross-fitting since it involves just one nuisance estimator. In contrast to the DCDR estimator analysis presented in previous sections, this plug-in estimator leverages the benefits of undersmoothing in a manner more consistent with the classical literature, where the specialized construction of the nuisance estimator enables adaptation to the underlying smoothness properties (e.g., \citet{gine2008simple}).
\end{remark}

\subsection{Slower-than-$\sqrt{n}$ CLT}

In addition to minimax optimality, asymptotic normality is possible in the non-$\sqrt{n}$ regime.  The DCDR estimator in Theorem~\ref{thm:minimax} balances bias and variance; intuitively, if the DCDR estimator were undersmoothed one might expect it to converge to a Normal distribution centered at the ECC at a sub-optimal slower-than-$\sqrt{n}$ rate.  We demonstrate this in the next result.  First, we incorporate two further assumptions.
\begin{assumption} \label{asmp:boundedAY}
    \textbf{(Boundedness)} There exists $M > 0$ such that $|A| < M$ and $|Y| < M$.
\end{assumption}
\begin{assumption} \label{asmp:continuous_cond_var}
    \textbf{(Continuous conditional variance)} $\bbV(A \mid X = x)$ and $\bbV(Y \mid X = x)$ are continuous in $x$. 
\end{assumption}
Assumption~\ref{asmp:boundedAY} asserts that $A$ and $Y$ are bounded.  Assumption~\ref{asmp:continuous_cond_var} dictates that the conditional variances of $A$ and $Y$ are continuous in $X$, which is used to show that the limit of the standardizing variance in \eqref{eq:limiting_var}, below, exists. It may be possible to relax these assumptions with more careful analysis.  Nonetheless, with them it is possible to establish the following result.	
\begin{restatable}{theorem}{thminference}\label{thm:inference} \emph{\textbf{(Slower-than-$\sqrt{n}$ CLT)}}
    Under the conditions of Theorem~\ref{thm:minimax}, suppose $\frac{\alpha + \beta}{2} < \frac{d}{4}$ and Assumptions~\ref{asmp:boundedAY} and~\ref{asmp:continuous_cond_var} hold.  Suppose $\widehat \mu$ is the undersmoothed nuisance function estimator with bandwidth $h_\mu$ scaling at $n^{-\frac{2 + \varepsilon}{2\alpha + 2\beta + d}}$ for $0 < \varepsilon < \frac{4(\alpha + \beta)}{d}$ while $\widehat \pi$ is the smooth consistent estimator. Then,
    \begin{equation} \label{eq:slow_clt}
        \sqrt{\frac{n}{\bbV \{\widehat \varphi(Z)\mid D_\pi, D_\mu \}}} (\widehat \psi_n - \psi_{ecc}) \indist N(0, 1).
    \end{equation}
    Moreover,
    \begin{equation} \label{eq:limiting_var}
        n h_\mu^d \bbV \{\widehat \varphi(Z)\mid D_\pi, D_\mu \} \inas \bbE\left\{ \frac{\bbV(A \mid X) Y^2}{f(X)} \right\} \bbE\left\{\frac{K_\mu(X)^2}{f(X)}\right\}, 
    \end{equation}
    where $K_\mu$ is the kernel for $\widehat \mu$. If the roles of $\widehat \mu$ and $\widehat \pi$ were reversed, then \eqref{eq:slow_clt} holds and
    \begin{equation} \label{eq:limiting_var_pi}
        n h_\pi^d \bbV \{\widehat \varphi(Z)\mid D_\pi, D_\mu \} \inas\bbE\left\{ \frac{\bbV(Y \mid X) A^2}{f(X)} \right\} \bbE\left\{\frac{K_\pi(X)^2}{f(X)}\right\}.
    \end{equation}
\end{restatable}

Theorem~\ref{thm:inference} shows that the DCDR estimator can be suitably undersmoothed in the non-$\sqrt{n}$ regime so the DCDR estimator is sub-optimal but converges to a Normal distribution around the ECC.  Moreover, Theorem~\ref{thm:inference} establishes that the conditional variance by which the error is standardized converges almost surely to a constant which can be estimated from the data.  Therefore, Wald-type confidence intervals for the ECC can be constructed using \eqref{eq:slow_clt} and \eqref{eq:limiting_var} or \eqref{eq:limiting_var_pi}.  As far as we are aware, this is the first result demonstrating slower-than-$\sqrt{n}$ inference for a cross-fit estimator of a causal functional.  

\medskip

Here, we give some intuition for the result, which might best be understood through its unorthodox denominator in the standardization term in \eqref{eq:slow_clt}: the \emph{conditional variance of the estimated efficient influence function}.  This denominator is unorthodox both because it includes an \emph{estimated} efficient influence function and because it is a \emph{conditional} variance.  The estimated efficient influence function arises because $\widehat \psi_n$ is undersmoothed to such an extent that its scaled variance, $\mathbb{V} \left( \sqrt{n} \widehat \psi_n \right)$, is growing with sample size.  Similarly, $\bbV \{\widehat \varphi(Z)\mid D_\pi, D_\mu \}$ is also growing at the same rate with sample size, and thus standardizing by this term appropriately concentrates the variance of the standardized statistic, $\sqrt{\frac{n}{\bbV \{\widehat \varphi(Z)\mid D_\pi, D_\mu \}}} (\widehat \psi_n - \psi_{ecc})$.  Indeed, \eqref{eq:limiting_var} demonstrates that $\bbV\{ \widehat \varphi(Z) \mid D_\pi, D_\mu \}$ is growing with sample size because $nh_\mu^d \to 0$ as $n \to \infty$ by the assumption on the bandwidth. This result relies on a bound for higher moments of a U-statistic (Proposition 2.1, \citet{gine2000exponential}) which guarantees control of the sum of off-diagonal terms in  $\bbV\{ \widehat \varphi(Z) \mid D_\pi, D_\mu \}$. 

\medskip

Meanwhile, the \emph{conditional} variance is required so that a normal limiting distribution can be attained. While the non-$\sqrt{n}$ regime is often characterized by non-normal limiting distributions, a normal limiting distribution can be established applying the Berry-Esseen inequality (Theorem 1.1, \citet{bentkus1996berry}) after conditioning on the training data and showing that the standardized statistic satisfies a conditional central limit theorem almost surely and, therefore, an unconditional central limit theorem.  

\medskip

This approach --- using sample splitting to conduct inference --- is an old method which has recently been examined in several contexts, including, for example, estimating U-statistics \citep{robins2016asymptotic, kim2024dimension}, estimating variable importance measures \citep{rinaldo2019bootstrapping}, high-dimensional model selection \citep{wasserman2009high}, and post-selection inference \citep{meinshausen2010stability, dezeure2015high}. Earlier references include \citet{cox1975note}, \citet{hartigan1969using}, and \citet{moran1973dividing}.

\medskip
    
While this section and previous sections have established several theoretical results for the DCDR estimator, in the next section we investigate and illustrate these properties via simulation.

\section{Simulations} \label{sec:simulations}

In this section, we study the behavior of double cross-fit doubly robust (DCDR) estimators and compare them to single cross-fit doubly robust with MSE-minimizing nuisance estimators (SCDR-MSE).  First, we provide evidence for why double cross-fitting leads to undersmoothing the nuisance estimators for optimal convergence rates, reinforcing our theoretical analysis in Section~\ref{sec:modelfree}. Then, we construct H\"{o}lder smooth nuisance functions and examine when the distribution of standardized SCDR-MSE and DCDR estimates converge to standard Gaussians, and the coverage of Wald-style confidence intervals, reinforcing our theoretical results from Sections~\ref{sec:unknown} and \ref{sec:known}. Finally, we examine the Monte Carlo error of the estimators to understand whether the additional cross-fitting for double cross-fitting harms the overall performance of the estimator.
   
\medskip

\noindent All code and analysis is available at \href{https://github.com/alecmcclean/DCDR}{https://github.com/alecmcclean/DCDR}

\subsection{Intuition for undersmoothing}

As discussed in Section~\ref{sec:modelfree}, under certain covariance conditions on the nuisance estimators, double cross-fitting must be coupled with undersmoothed nuisance estimators for faster convergence rates. We reinforce this intuition here. We consider a data generating process where $X$ is uniform, $A = Y$, and both nuisance functions are the Doppler function (see Figure~\ref{fig:doppler}). Formally, the data generating process is 
\begin{align}
    X &\sim \text{Unif}(0,1), \\
    \pi(X) = \mu(X) &=  \sqrt{ X (1 - X) } \sin \left( \frac{2.1 \pi}{X + 0.05} \right), \label{eq:doppler} \\
    A = Y &= \pi(X) + \varepsilon, \varepsilon \sim N(0, \psi_{ecc} = 0.1).
\end{align}
Because $A = Y$, the ECC is the variance of the error noise in $A$ and $Y$. We chose $\psi_{ecc} = 0.1$ to give a strong signal-to-noise ratio for the estimators.

\begin{figure}[H]
    \centering
    \includegraphics[height = 2.5in]{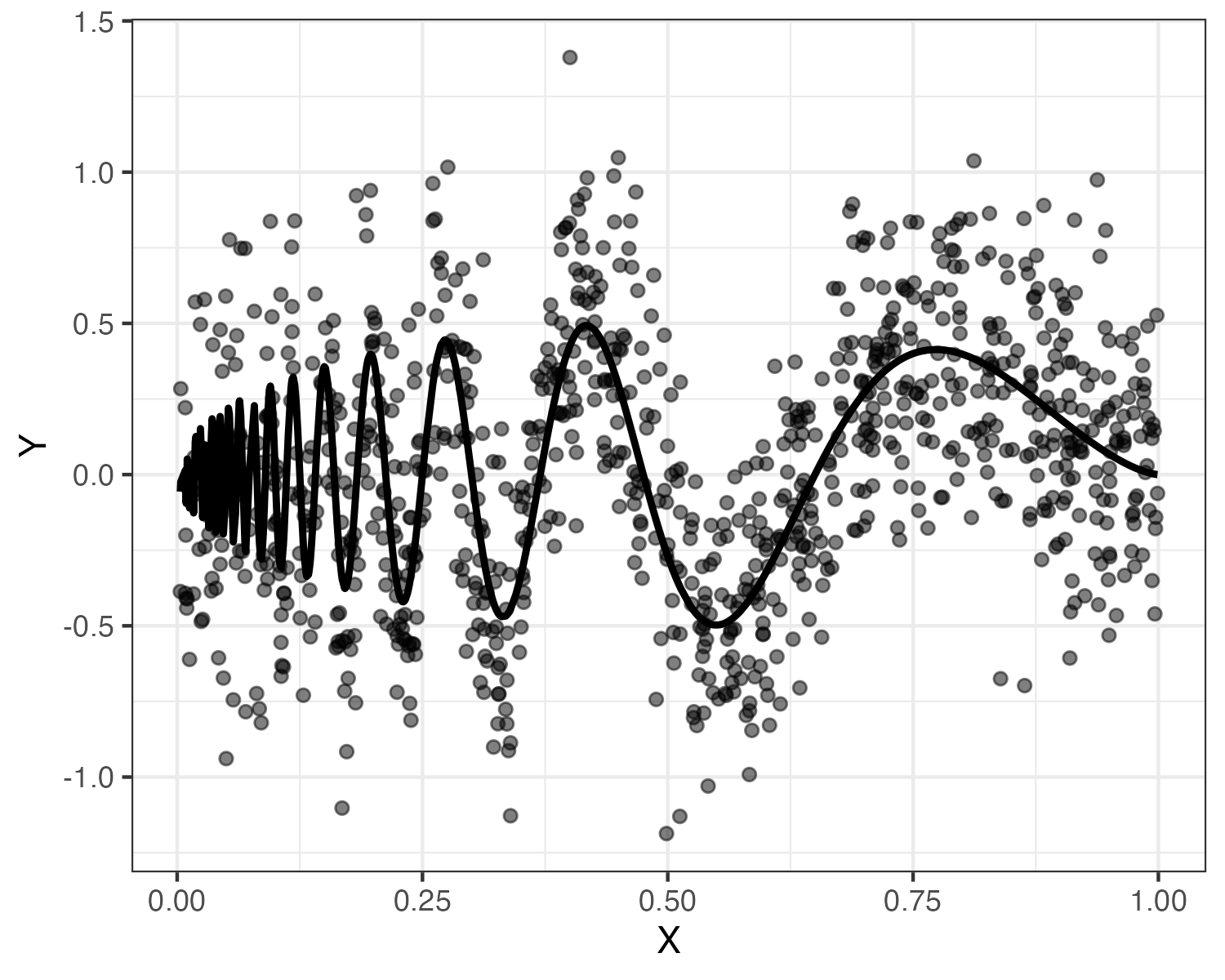}
    \caption{The Doppler function with $N(0, 0.1)$ random noise as in \eqref{eq:doppler}; this nuisance function was used for Figure~\ref{fig:changing_k}.}
    \label{fig:doppler}
\end{figure}

\medskip
We generated $500$ datasets with three folds of sizes $\{ 50, 100, 200, 500, 1000, 2000 \}$ and estimated each nuisance function with k-Nearest-Neighbors for $k$ from $1$ to $30$.  We estimated the ECC with the DCDR estimator and the SCDR estimator; for the SCDR estimator we trained the nuisance functions on the same fold and discarded the unused third fold (see Remark~\ref{rem:splitting}).  For each $k$, we computed the average mean squared error (MSE) of the nuisance function estimators and the DCDR and SCDR estimators over $500$ datasets. 

\medskip

To understand when undersmoothing is optimal, we calculated the optimal $k$ corresponding to the lowest average MSE over $500$ datasets for the DCDR, SCDR, and nuisance function estimators.  Figure~\ref{fig:changing_k} displays the optimal number of neighbors (y-axis) for each fold size (x-axis), with different colors denoting  estimator/estimand combinations.  For instance, the green point in the bottom left corner signifies that $k = 2$ gave the lowest average MSE over $500$ repetitions for the DCDR estimator estimating the ECC with datasets with folds of size $50$.  The black points and line represent the optimal $k$ for $\widehat \pi$ estimating $\pi$, orange represents $\widehat \mu$ estimating $\mu$, blue represents the SCDR estimator estimating the ECC, and green represents the DCDR estimator estimating the ECC (blue, orange, and black are the same line for the most part, so the blue line completely obscures the orange and partially obscures the black).  Figure~\ref{fig:changing_k} demonstrates the anticipated phenomenon: the optimal number of neighbors is lower for the DCDR estimator compared to the SCDR estimator and the nuisance function estimators, and it increases at a slower rate as sample size increases.  Equivalently, the optimal $k$ for the DCDR estimator corresponds to undersmoothed nuisance function estimators while the optimal $k$ for the SCDR estimator corresponds to optimal nuisance function estimators.

\begin{remark}
    The Doppler function is highly non-smooth and, therefore, is well suited to a structure-agnostic analysis like in Section~\ref{sec:modelfree}. Therefore, our first set of simulations are targeted to shed light on those results, rather than on subsequent results with smoothness assumptions in Sections~\ref{sec:unknown} and \ref{sec:known}.  To that end, we consider nuisance estimators using the same number of neighbors $k$. However, as our results in Section~\ref{sec:known} established, and our simulations in the next sections illustrate, an optimal DCDR estimator might consider different numbers of neighbors for each nuisance estimator, depending on the smoothness of the underlying nuisance function.
\end{remark}

\color{black}
\begin{remark} \label{rem:splitting}
    Figure~\ref{fig:changing_k} does not describe whether the SCDR estimator or DCDR estimator is more accurate, nor is that the goal of the analysis for Figure~\ref{fig:changing_k}.  Because we discarded a third of the data available to the SCDR estimator, it is not possible to compare the estimators directly.  Instead, Figure~\ref{fig:changing_k} shows that the DCDR estimator requires undersmoothed nuisance function estimators for optimal accuracy, while the SCDR estimator requires optimal nuisance function estimators.   In the next set of simulations, we cycle the folds so that the estimators can be directly compared.
\end{remark}

\begin{figure}[H]
    \centering    
    \includegraphics[height = 3in]{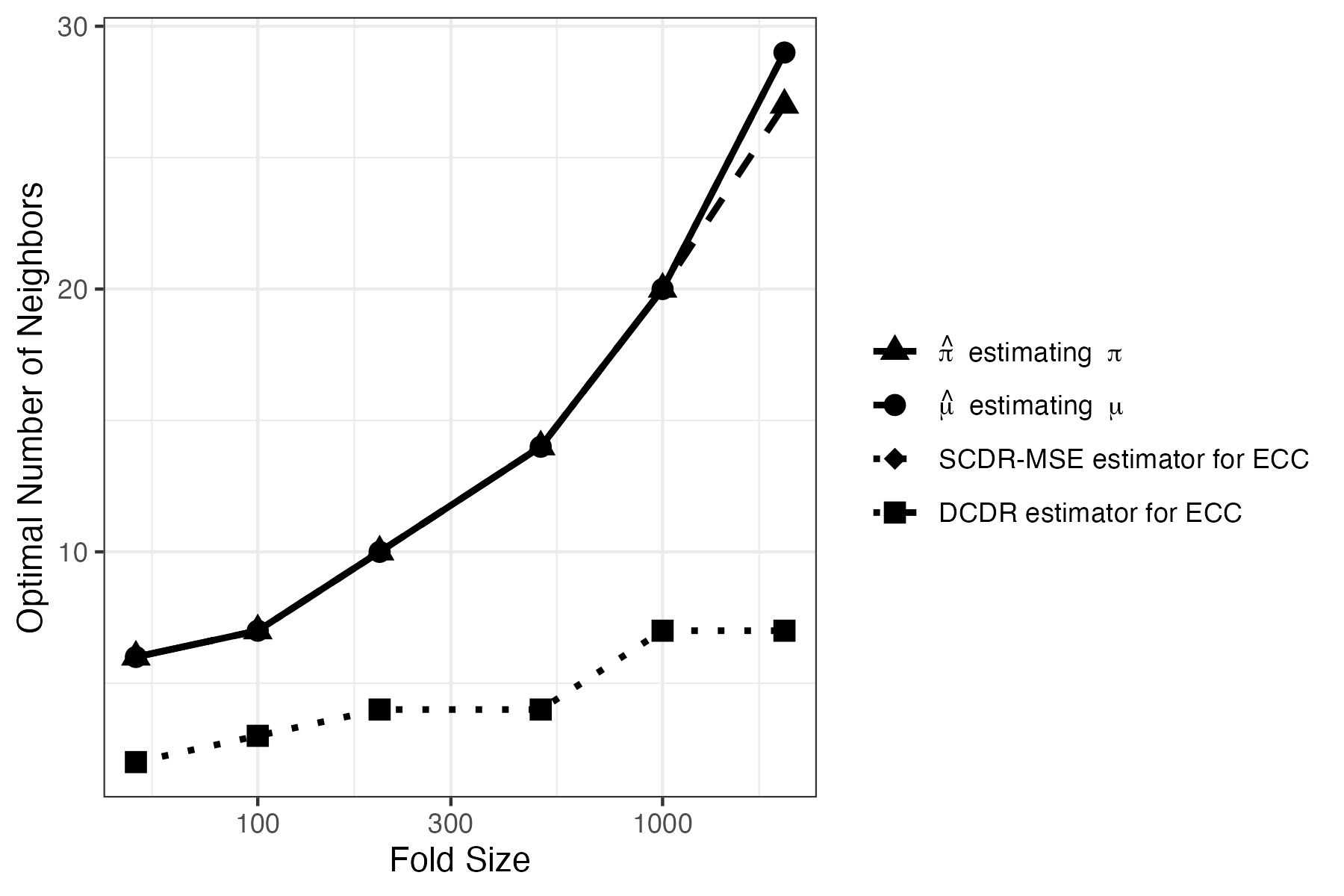}
    \caption{Fold size (x-axis) versus optimal number of neighbors (y-axis), where optimal is in terms of average MSE over 500 datasets; triangles and circles indicate the k-Nearest-Neighbors estimators for $\pi(X)$ and $\mu(X)$, respectively, while diamonds indicate the SCDR estimator for the ECC and squares indicate the DCDR estimator for the ECC.}
    \label{fig:changing_k}
\end{figure}

\subsection{Inference and coverage}

Theorem~\ref{thm:semiparametric} in Section~\ref{sec:unknown} and Theorem~\ref{thm:inference} in Section~\ref{sec:known} provided convergence guarantees for the DCDR estimator under smoothness assumptions. Here, we demonstrate these results via simulation.

\medskip

To facilitate our analysis, we constructed suitably smooth nuisance functions.  Specifically, we consider both 1-dimensional and 4-dimensional covariates uniform on the unit cube, $\psi_{ecc} = 10$, and $\pi$ and $\mu$ H\"{o}lder smooth. Throughout, we set both nuisance functions $\pi$ and $\mu$ to be of the same smoothness such that $\alpha = \beta = s$, and we control the smoothness $s$.  To construct appropriately smooth functions, we employed the lower bound minimax construction for regression (see, \citet{tsybakov2009introduction}, pg. 92).  These functions vary with sample size, and Figure~\ref{fig:example_holder} provides an illustration for $d = 1$, with smoothness levels $s \in \{ 0.1, 0.35, 0.6\}$ and dataset sizes $N \in \{ 100, 1000, 5000 \}$.   To generate 4-dimensional H\"{o}lder smooth functions, we added four functions that are univariate H\"{o}lder smooth in each dimension. 

\begin{figure}
    \centering
    \includegraphics[width = 0.8\textwidth]{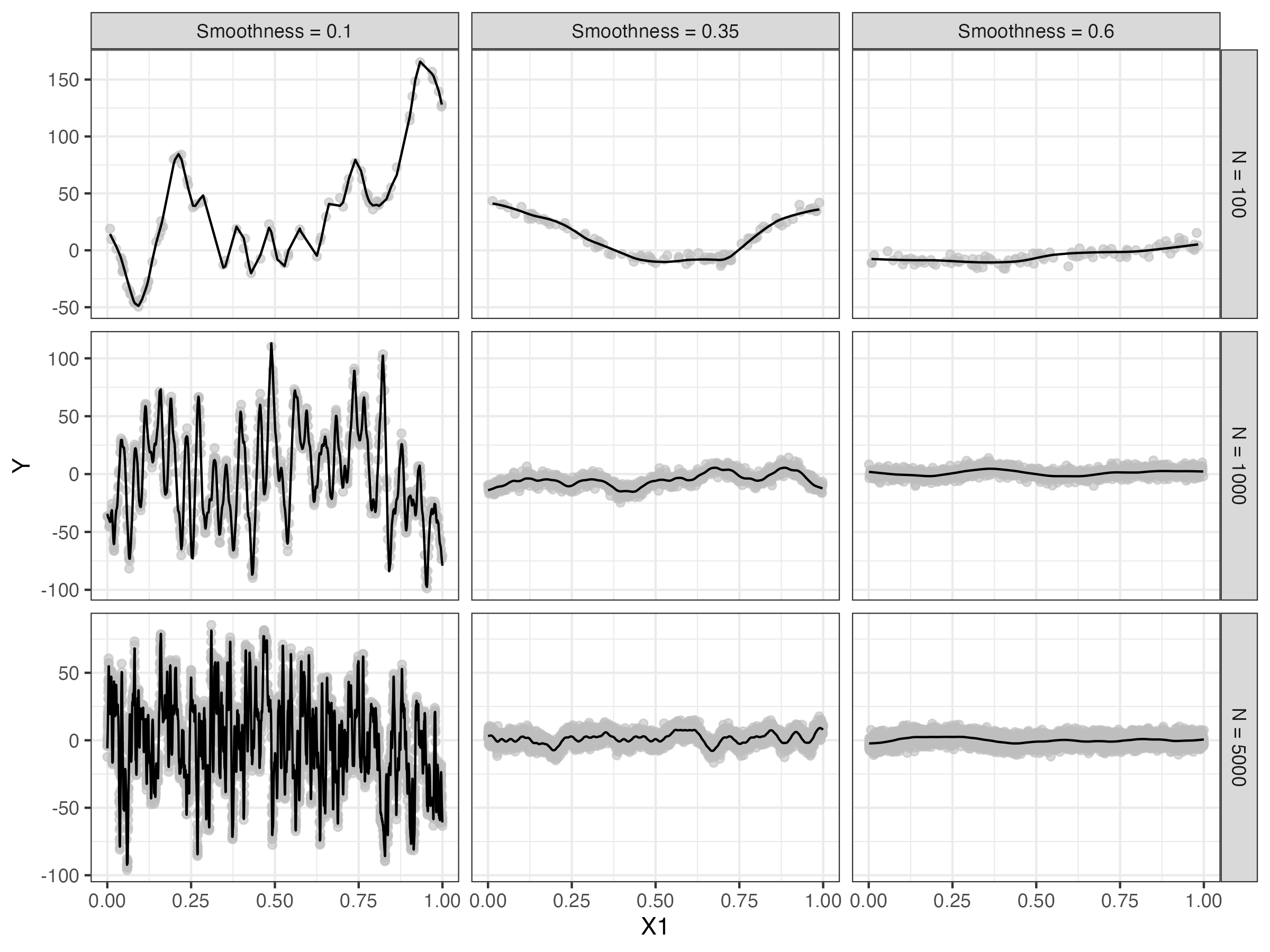}
    \caption{Example Holder smooth functions (black) of order $s \in \{0.1, 0.35, 0.6\}$ smoothness for $n \in \{100, 1000, 5000 \}$ observed data points (grey) with $N(0, 10)$ random noise.}
    \label{fig:example_holder}
\end{figure}

\medskip

We generated datasets for fold sizes $\{ 100, 200, 350, 700, 1500, 3000 \}$ where each dataset consisted of three folds.  When $d = 1$, we constructed nuisance functions with smoothnesses $\{ 0.1, 0.35, 0.6 \}$, and when $d = 4$ with smoothnesses $\{ 0.6, 1.5, 2.5 \}$.  For each fold size-dimension-smoothness combination, we generated $100$ datasets and constructed three estimators:
\begin{enumerate}
    \itemsep0.05in
    \item \textbf{SCDR-MSE} For $d \in \{1, 4\}$, we constructed the SCDR estimator with covariate-density-adapted kernel regressions (Estimator~\ref{est:cdalpr}), where we tuned the bandwidth at the optimal rate with sample size, using the smoothness of the underlying nuisance functions to do so. This is an approximation of the typical SCDR-MSE estimator, which we use as a benchmark to compare with the DCDR estimators.
     
    \item \textbf{DCDR undersmoothed local polynomial regression} For only $d = 1$, we constructed the DCDR estimator with undersmoothed local polynomial regression (Estimator~\ref{est:lpr}) \textit{without leveraging knowledge of the covariate density or the smoothness of the nuisance functions}. To undersmooth the nuisance estimators, we constructed adaptive bandwidths using the $10$ nearest neighbors to each estimation point in the training data. This is an ad-hoc method to scale the bandwidths at the appropriate rate $\left( \log n / n \right)^{-1/d}$, as in Theorem~\ref{thm:semiparametric}.
    
    \item \textbf{DCDR known density and smoothness} For $d \in \{1, 4\}$, we constructed the DCDR estimator with covariate-density-adapted kernel regressions (Estimator~\ref{est:cdalpr}), using knowledge of the covariate density and smoothness. We tuned the bandwidth of one nuisance estimator so it was consistent and undersmoothed to such a degree that the DCDR estimator itself was undersmoothed and could achieve a Gaussian limiting distribution even in the non-$\sqrt{n}$ regime, as in Theorem~\ref{thm:inference}. 
\end{enumerate}
For all estimators, we used two folds to construct nuisance estimators and the third fold to construct the functional estimator; then, we cycled the folds two times, repeated the process, and averaged across the full sample. Hence, all estimators were constructed using the full sample. \color{black} For all estimators, we constructed Wald-type 95\% confidence intervals for the ECC using the sample variance of the estimated efficient influence functions to estimate the limiting variance.

\medskip

Figures~\ref{fig:qq} and \ref{fig:coverage} show the inferential properties of the estimators.  Figure~\ref{fig:qq} contains QQ plots for the standardized statistics for different smoothnesses (rows) and fold sizes (columns) for dimension equal to one.  The black circles represent the \textit{DCDR known density and smoothness} estimator, while the orange squares represent the \textit{DCDR undersmoothed local polynomial} estimator, and the blue triangles represent the \textit{SCDR-MSE} estimator. The diagonal line is $y = x$.  Figure~\ref{fig:coverage} displays the coverage of the associated Wald-type confidence intervals, with the dimension and smoothness varying by column, and the sample size on the x-axis. 

\medskip

The results in Figures~\ref{fig:qq} and \ref{fig:coverage} confirm that non-$\sqrt{n}$ inference is possible, as in Theorem~\ref{thm:inference}.  As the sample size increases (moving across the panels in Figure~\ref{fig:qq}), the quantiles of the \textit{DCDR known density and smoothness} estimates in black converge to the quantiles of the standard normal distribution. Additionally, as sample size increases (moving across the x-axis in Figure~\ref{fig:coverage}), the coverage of the confidence intervals approach appropriate coverage.  These findings align with what was anticipated by the limiting distribution result in Theorem~\ref{thm:inference}.  This occurs \emph{even when} $s < d/4$.  

\medskip

Figures \ref{fig:qq} and \ref{fig:coverage} also confirm that the \textit{DCDR undersmoothed local polynomial regression} estimator facilitates $\sqrt{n}$-convergence and inference under the minimal smoothness condition, when $s > d/4$, as in Theorem~\ref{thm:semiparametric}, when the \textit{SCDR-MSE} estimator does not. This is corroborated in the middle rows of Figures~\ref{fig:qq} and \ref{fig:coverage}, where the quantiles of the \textit{DCDR undersmoothed local polynomial regression} estimator converge to the quantiles of the standard normal for $s/d = 0.35$ and the confidence intervals achieve appropriate coverage.  However, the quantiles diverge and the confidence intervals fail to achieve appropriate coverage when $s < d/4$, as shown by the orange squares in the top rows.  Meanwhile, as a benchmark, Figures~\ref{fig:qq} and~\ref{fig:coverage} illustrate that the SCDR-MSE only achieves $\sqrt{n}$-convergence when $s > d/2$. When $s > d/2$, the SCDR-MSE quantiles in the bottom row of Figure~\ref{fig:qq} converge closely to the normal quantiles, and do not converge otherwise.  The same phenomenon occurs for the confidence intervals in Figure~\ref{fig:coverage}, which do not achieve appropriate coverage when $s < d/2$.  In summary, these results support the theoretical conclusion that the DCDR estimators are $\sqrt{n}$-consistent and asymptotically normal in sufficiently non-smooth regimes ($d/4 < s < d/2$) where the SCDR-MSE estimator is not.

\begin{remark} \label{rem:drawback}
    The SCDR-MSE estimator we consider is a useful benchmark against which to compare the DCDR estimators because it is a reasonable stand-in for the typical modern SCDR-MSE pipeline, whereby one constructs nuisance estimators to minimize MSE. However, it is important to note that the SCDR estimator could instead be coupled with undersmoothed linear smoothers to achieve $\sqrt{n}$-convergence under the weakest smoothness conditions, but this is not demonstrated here \citep{mcgrath2022undersmoothing}. 
\end{remark}

\subsection{Monte Carlo error}

In this section, we compare the efficiency of the estimators. The results come from the same data generating process and estimators as in the previous section, and the results are in Figure~\ref{fig:monte_carlo}. Figure~\ref{fig:monte_carlo} shows the point estimates and 95\% confidence intervals for squared bias, variance, and MSE over $100$ simulations; the lower bound of the 95\% confidence intervals was excluded if it equaled zero. 

\medskip

The DCDR estimators both perform well compared to the \textit{SCDR-MSE} estimator. When $s/d = 0.1$ (top row), the \textit{DCDR known density and smoothness} has the highest variance but the lowest bias, which makes sense because this estimator is undersmoothed to guarantee a slower-than-$\sqrt{n}$ CLT. Interestingly, the lower bias outweighs the higher variance, and the \textit{DCDR known density and smoothness} estimator has the lowest MSE. Meanwhile, for $s/d = 0.35$ (middle row) and $s/d = 0.6$ (bottom row), the \textit{DCDR undersmoothed local polynomial regression} estimator performs the best, with the lowest bias and variance and therefore the lowest MSE.  

\begin{remark}
    Although the \textit{SCDR-MSE} estimator we consider is a useful benchmark for evaluating the DCDR estimators, it may be possible to construct a \textit{SCDR-MSE} estimator with better finite-sample performance by adaptively choosing the bandwidth using cross-validation rather than choosing the asymptotically optimal bandwidth as we do here. 
\end{remark}

\begin{figure}[H]
    \centering
    \includegraphics[width=0.95\textwidth]{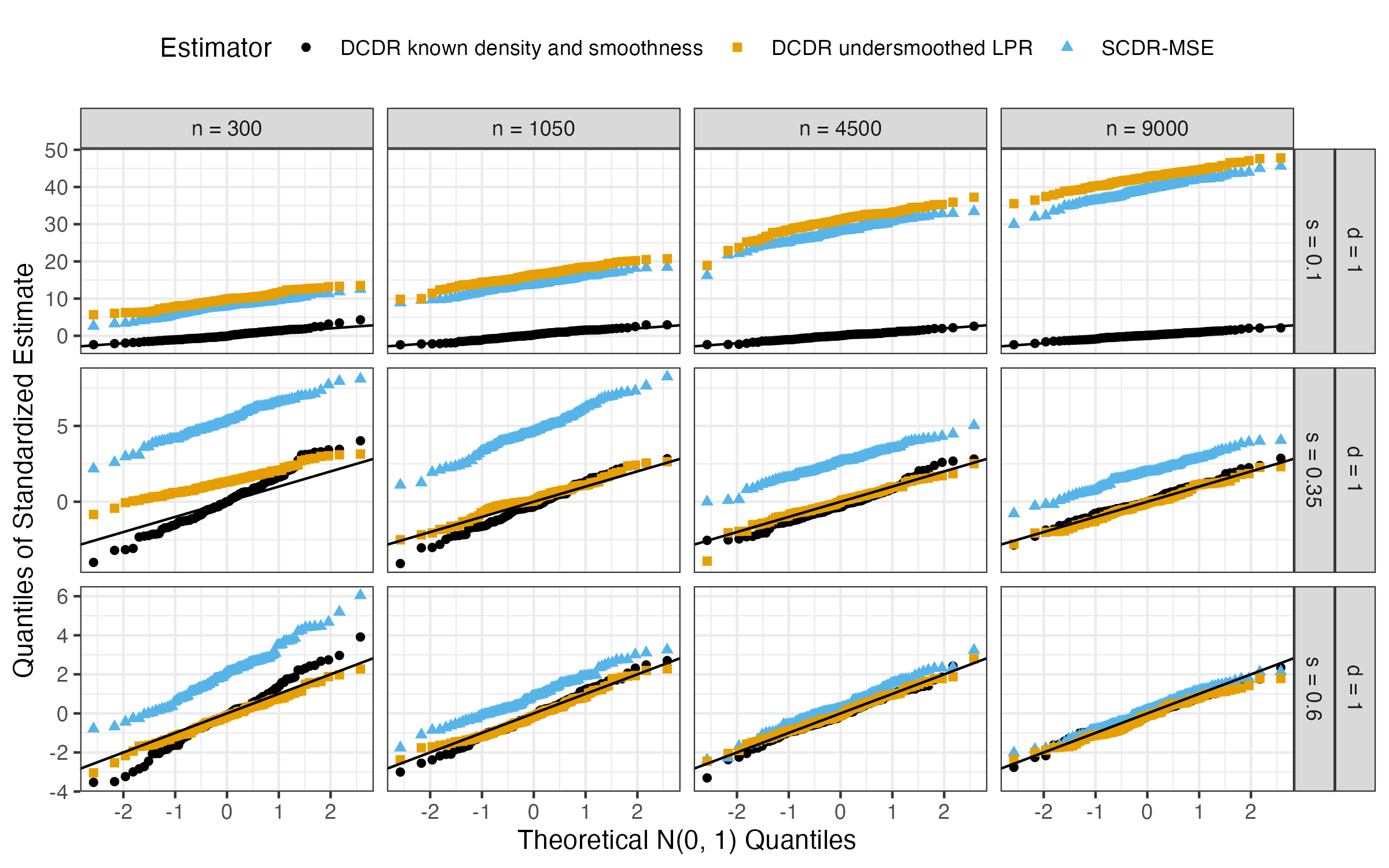}
    \caption{QQ Plots for the standardized statistics for different dimensions and smoothnesses (columns) and fold sizes (rows). Black circles represent the \textit{DCDR known density and smoothness} estimator, orange squares represent the \textit{DCDR undersmoothed local polynomial regression} estimator, and blue triangles represent the \textit{SCDR-MSE} estimator. The diagonal line is $y = x$.}
    \label{fig:qq}
\end{figure}

\begin{figure}[H]    
    \centering
    \includegraphics[height=4in]{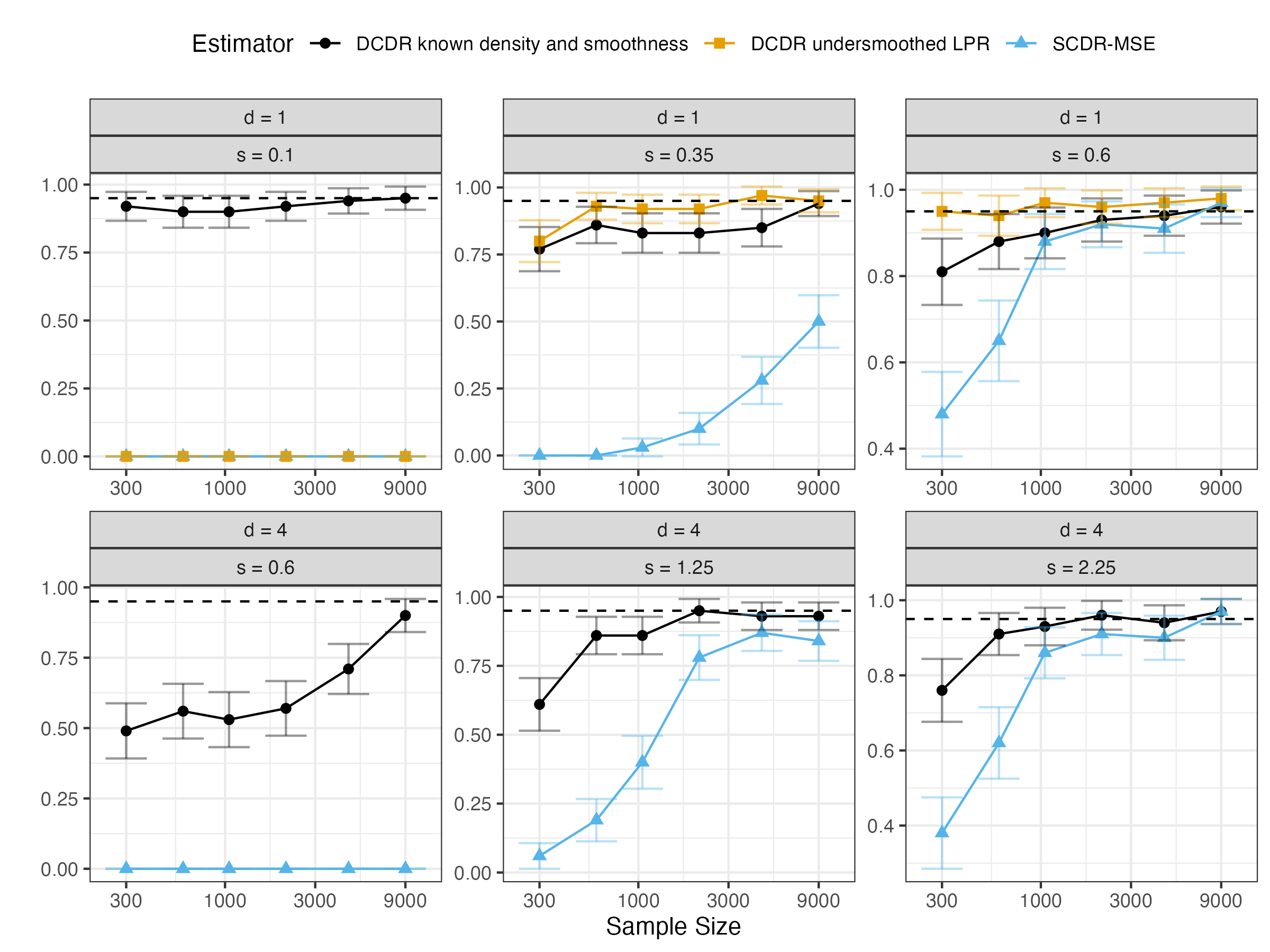}
    \caption{Points represent the coverage of 95\% confidence intervals over $100$ datasets constructed for different dimensions and smoothnesses (panels) and fold sizes (x-axis).  Error bars represent 95\% confidence intervals for the coverage of Wald-type confidence intervals.  Black circles represent the \textit{DCDR known density and smoothness} estimator, orange squares represent the \textit{DCDR undersmoothed local polynomial regression} estimator, and blue triangles represent the \textit{SCDR-MSE} estimator.}
    \label{fig:coverage}
\end{figure}

\begin{figure}[htbp]
    \centering
    \includegraphics[height=4in]{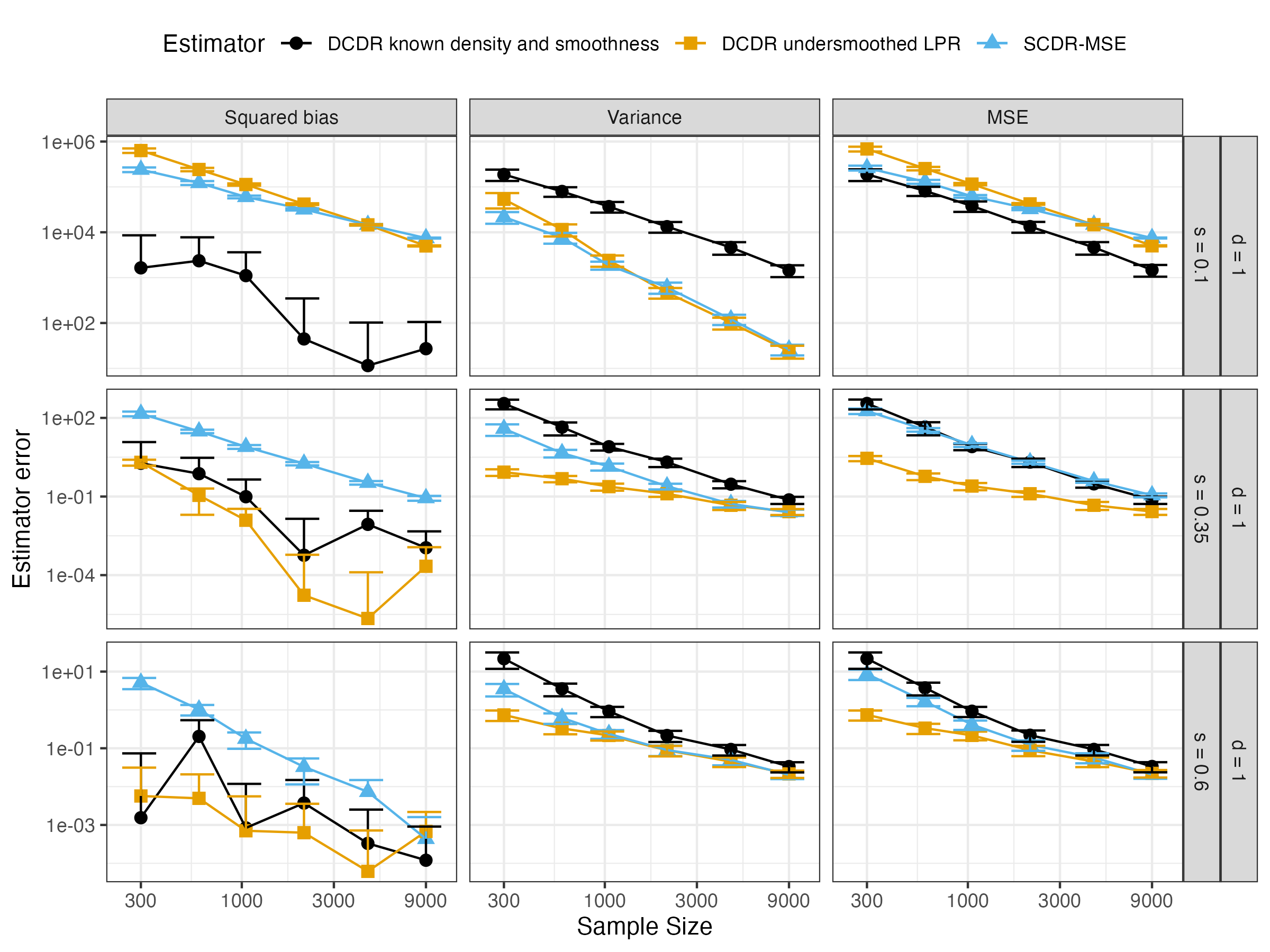}
    \caption{Illustrating the efficiency of double cross-fit estimators. Points represent the average squared bias (left column), variance (middle column), and MSE (right column) over $100$ datasets constructed for different dimensions and smoothnesses (panels) and fold sizes (x-axis). Black circles represent the \textit{DCDR known density and smoothness} estimator, orange squares represent the \textit{DCDR undersmoothed local polynomial regression} estimator, and blue triangles represent the \textit{SCDR-MSE} estimator. 95\% confidence intervals are shown; when the lower bound of the CI equals zero, the lower bound on the CI is excluded.}
    \label{fig:monte_carlo}
\end{figure}


\color{black}
\section{Discussion} \label{sec:discussion}

In this paper, we studied a double cross-fit doubly robust (DCDR) estimator for the Expected Conditional Covariance (ECC). We first provided a novel structure-agnostic error analysis for the DCDR estimator, which holds for generic data generating processes and nuisance function estimators. We observed  that a faster convergence rate is possible by undersmoothing the nuisance function estimators, provided that these estimators satisfy a covariance condition.  We established that several linear smoothers satisfy this covariance condition, and focused on the DCDR estimator with local averaging estimators for the nuisance functions, which had not been studied previously.  We showed that the DCDR estimator based on undermoothed local polynomial regression is $\sqrt{n}$-consistent and asymptotically normal under minimal conditions without knowledge of the covariate density or the smoothness of the nuisance functions.  When the covariate density is known, we demonstrated that the DCDR estimator based on undersmoothed covariate-density-adapted kernel regression is minimax optimal.  Moreover, we proved an undersmoothed DCDR estimator satisfies a slower-than-$\sqrt{n}$ central limit theorem.  Finally, we conducted simulations that support our findings, providing intuition for double cross-fitting and undersmoothing, demonstrating when the DCDR estimator can facilitate $\sqrt{n}$-consistency and asymptotic normality under minimal conditions, and illustrating slower-than-root-n asymptotic normality for the undersmoothed DCDR estimator in the non-$\sqrt{n}$ regime. 

\medskip

There are several potential extensions of our work. While we focus on the ECC, the principles applied here may generalize to wider classes of functionals. Indeed, \citet{newey2018cross} derived general results for the class of ``average linear functionals'' (\citet{newey2018cross}, Section 3). Beyond those results, similarly general results might be possible for the larger class of ``mixed bias functionals" \citep{rotnitzky2021characterization}. Mixed bias functionals satisfy bias decompositions of the form \( \bbE \left( \widehat \psi - \psi \right) = \bbE \left[ f(Z) \{ \widehat \eta_1(Z) - \eta_1(Z) \} \{ \widehat \eta_2(Z) - \eta_2(Z) \} \right] \), where $\eta_1, \eta_2$ are nuisance functions and $f(Z)$ is another function. This is a similar bias decomposition to what we observed for the ECC, and therefore convergence guarantees may be possible using similar arguments to our structure-agnostic analysis in Section~\ref{sec:modelfree}. However, achieving this would entail developing principled approaches for undersmoothing estimators of non-standard nuisance functions --- $\eta_1$ and $\eta_2$ are not always conditional means, and therefore straightforward regression undersmoothing methods may not apply. 

\medskip

Finally, the results in Sections~\ref{sec:modelfree} and~\ref{sec:unknown} can imply practical implementations of the DCDR estimator that achieve faster convergence rates. In Section~\ref{sec:unknown}, we observed that there are simple ad-hoc methods to undersmooth local polynomial regression or series regression at an appropriate rate with sample size by undersmoothing as much as possible. In Section~\ref{sec:simulations}, we observed that this approach worked well in practice. Future work could investigate how these approaches perform with real data and how to generalize these ideas to other machine learning nuisance estimators more rigorously.

\color{black}
\section*{Acknowledgments}

The authors thank several anonymous reviewers, Zach Branson, the CMU causal inference reading group, and participants at ACIC 2023 for helpful comments and feedback.

\bibliographystyle{abbrvnat}
\bibliography{references}

\newpage

\setcounter{section}{0} 
\renewcommand{\thesection}{\Alph{section}} 

\Large \noindent \textbf{Appendix} 

\medskip
\normalsize 

\noindent These supplemental materials are arranged into eight sections:
\begin{itemize}
    \item In Appendix~\ref{app:practical}, we investigate when and how one might conduct undersmoothing with generic machine learning estimators.
    \item  In Appendix~\ref{app:modelfree}, we prove Lemma~\ref{lem:expansion} and Proposition~\ref{prop:spectral} from Section~\ref{sec:modelfree}. 
    \item In Appendix~\ref{app:nf_ests}, we prove bias, variance, and covariance bounds for the nuisance function estimators considered in Section \ref{sec:unknown} --- k-Nearest-Neighbors and local polynomial regression. 
    \item In Appendix~\ref{app:unknown}, we use the results from Appendices~\ref{app:modelfree} and~\ref{app:nf_ests} to prove Lemma~\ref{lem:covariance} and Theorem~\ref{thm:semiparametric} from Section~\ref{sec:unknown}.  
    \item In Appendix~\ref{app:random-forest}, we establish bias, variance, and covariance bounds for centered random forest estimators.
    \item In Appendix~\ref{app:cda_kernel}, we prove a variety of results for covariate-density-adapted kernel regression, including conditional and unconditional variance upper and lower bounds. 
    \item In Appendix~\ref{app:known}, we prove Theorems~\ref{thm:minimax} and~\ref{thm:inference} from Section~\ref{sec:known}, making use of the results in Appendix~\ref{app:cda_kernel}. 
    \item In Appendix~\ref{app:technical}, we prove three technical results regarding properties of the covariate density.  
    \item In Appendix~\ref{app:slln}, we provide a simple strong law of large numbers for triangular arrays of bounded random variables. 
    \item Finally, in Appendix~\ref{app:series}, we review series regression nuisance function estimators, and state and prove several results based on these estimators, which are equivalent to Lemma~\ref{lem:covariance} and Theorem~\ref{thm:semiparametric} in Section~\ref{sec:unknown} of the paper. 
\end{itemize}

\section{Small steps towards undersmoothing in practice} \label{app:practical}

In this appendix, we informally investigate the structure-agnostic results in further detail to gain intuition on how they might inform undersmoothing in practice with generic machine learning estimators. To that end, we consider the following simplifying assumptions:
\begin{enumerate}
    \itemsep0.05in
    \item $\bbE \| \widehat \varphi - \varphi \|_\bbP^2 \lesssim \| b_\pi^2 \|_\infty + \| s_\pi^2 \|_\infty + \| b_\mu^2 \|_\infty + \| s_\mu^2 \|_\infty$. 
    \item The nuisance estimators satisfy the covariance condition $\bbE \Big[ \big| \cov \{ \widehat \eta(X_i), \widehat \eta(X_j) \mid X_i, X_j \} \big| \Big] = O_\bbP( n^{-1} )$.
    \item There is a monotone bias-variance trade-off over the considered range of each tuning parameter, meaning that bias increases and variance decreases as the tuning parameter moves in one direction.
    \item The supremum variance of each nuisance estimator remains bounded within the considered range of tuning parameters.
\end{enumerate}
The first assumption states that the estimation error of the EIF is bounded above by the sum of the supremum squared bias and variance terms from the nuisance estimators; this typically holds under mild conditions (including those used in this paper for the ECC). The second assumption bounds the expected covariance term asymptotically, and must be established for the nuisance estimator under consideration. The third assumption formalizes a known directionality in the bias-variance trade-off associated with tuning parameters. Although this assumption may not strictly hold when performing empirical loss minimization over highly non-convex function classes, it remains plausible in practice with many machine learning methods (e.g., number of boosting steps in gradient-boosted trees), or serves as a reasonable approximation (e.g., tree depth in random forests). The fourth assumption---that the supremum variance is bounded---is similarly reasonable in many practical scenarios.

\medskip

Under the first two assumptions, the spectral radius term from Proposition~\ref{prop:spectral} satisfies
\[
\frac{\rho(\Sigma_n)}{n} = O_\bbP \left(\frac{\bbE \lVert \widehat \varphi - \varphi \rVert_\bbP^2}{n} \right) = O_\bbP\left(\frac{ \lVert b_{\pi}^2 \rVert_\infty + \lVert s_{\pi}^2 \rVert_\infty +   \lVert b_{\mu}^2 \rVert_\infty + \lVert s_{\mu}^2 \rVert_\infty }{n} \right).
\]
Therefore, revisiting the linear expansion from Lemma~\ref{lem:expansion} reveals
\begin{equation} \label{eq:remainder2}
    \widehat \psi_n - \psi_{ecc} = (\bbP_n - \bbP) \{ \varphi(Z) \} + O\left( \| b_\mu \|_\bbP \| b_\pi \|_\bbP \right) + O_\bbP \left( \sqrt{\frac{ \lVert b_{\pi}^2 \rVert_\infty + \lVert s_{\pi}^2 \rVert_\infty +   \lVert b_{\mu}^2 \rVert_\infty + \lVert s_{\mu}^2 \rVert_\infty  }{n}} \right).
\end{equation}
Combining this expansion with the third and fourth assumptions provides a straightforward heuristic for practically minimizing the error terms:
\begin{quote}
    \emph{Undersmooth the nuisance estimators as much as possible.}
\end{quote}
We can see that the first bias term will dominate because the second error term is already standardized by $n^{-1/2}$; hence, undersmoothing the nuisance estimators as much as possible to drive $\| b_\mu \|_\bbP \| b_\pi \|_\bbP$ to zero is imperative.

\medskip

This guideline is actionable with many estimators. For instance, with local polynomial regression, we would drive the bandwidth as small as possible while still retaining a well-defined estimator.  Indeed, this approach is precisely the one we adopted for local polynomial regression estimators in Section~\ref{sec:simulations}. More generally, this heuristic may offer useful practical guidance for more complex estimators commonly employed in modern functional estimation. However, it is not immediately clear whether this approach can be extended to general complex estimators. Nonetheless, in Appendix~\ref{app:random-forest}, we make a first step in this direction by establishing that the regularity conditions above hold for centered random forests, and therefore these estimators could be undersmoothed for faster rates when estimating the ECC \citep{biau2012analysis}.

\medskip

We also note that reducing the bias of nuisance estimators as much as possible may result in a non-negligible asymptotic error if $\| s_\mu^2 \|_\infty, \| s_\pi^2 \|_\infty \asymp 1$. Then, the error in the linear expansion in \eqref{eq:remainder2} is only $O_\bbP(n^{-1/2})$ rather than $o_\bbP(n^{-1/2})$.  Therefore, Wald-style confidence intervals based on the CLT for $(\bbP_n - \bbP) \{ \varphi(Z) \}$ might not have appropriate coverage. In simulations, we found that Wald-style confidence intervals performed well (e.g., in Figure~\ref{fig:smoothness-intro}) even in this scenario. This issue suggests a possible amendment to the heuristic:
\begin{quote}
    \emph{Undersmooth the nuisance estimators as much as possible while retaining consistency.}
\end{quote}
Under this amended guideline, we achieve the more desirable expansion $\widehat \psi_n - \psi_{ecc} = (\bbP_n - \bbP){\varphi(Z)} + o_\bbP(n^{-1/2})$. However, this modified heuristic may be less actionable in practice because it provides limited guidance on precisely how to select tuning parameters. An alternative strategy is to retain our original heuristic---full undersmoothing---and to instead employ inference methods robust to non-negligible bias, such as Adaptive HulC \citep{kuchibhotla2024hulc}.

\color{black}
\section{Section~\ref{sec:modelfree} proofs: Lemma~\ref{lem:expansion} and Proposition~\ref{prop:spectral}} \label{app:modelfree}

\lemexpansion*

\begin{proof}
    We first expand $\widehat \psi_n - \psi_{ecc}$ into the term in the statement of the lemma plus two remainder terms, $R_1$ and $R_2$:
    \begin{align}
        \widehat \psi_n - \psi_{ecc} &= \bbP_n \{ \widehat \varphi(Z) \} - \bbE \{ \varphi(Z) \} \nonumber \\
        &= (\bbP_n - \bbE) \{ \varphi(Z) \} + \underbrace{\bbE \{ \widehat \varphi(Z) - \varphi(Z) \}}_{R_{1,n}} + \underbrace{(\bbP_n - \bbE) \{ \widehat \varphi(Z) - \varphi(Z) \}}_{R_{2,n}} \label{eq:decomp}
    \end{align}
    where $\bbE$ refers to expectation over the estimation \emph{and training} data. The first term in \eqref{eq:decomp} appears in the statement of the lemma, so we manipulate it no further.  
    
    \medskip
    \large    
    \noindent \textbf{$R_{1,n}$ and bounding the bias of $\widehat \psi_n$:}
    
    \normalsize
    \noindent The second term in \eqref{eq:decomp}, $R_{1,n}$, is the bias of the estimator $\widehat \psi_n$.  It is not random.  A simple analysis shows
    \begin{align*}
        \bbE \{ \widehat \varphi(Z) - \varphi(Z) \} &\equiv \bbE \left[ \{ A - \widehat \pi(X) \} \{ Y - \widehat \mu(X) \} - \{ A - \pi(X) \} \{ Y - \mu (X) \} \right] \\
        &= \bbE \big[ \{ A - \widehat \pi(X) \} \{ \mu(X) - \widehat \mu(X) \} + \{ Y - \mu(X) \} \{ \pi(X) - \widehat \pi(X) \} \big] \\
        &= \bbE \big[ \{ \widehat \pi(X) - \pi(X) \} \{ \widehat \mu(X) - \mu(X) \} \big]
    \end{align*}
    where the final line follows by iterated expectations.  By the independence of the training datasets, we have
    $$
    \bbE \big[ \{ \widehat \pi(X) - \pi(X) \} \{ \widehat \mu(X) - \mu(X) \} \big] = \bbE \big[ \bbE \{ \widehat \pi(X) - \pi(X) \mid X \} \bbE \{ \widehat \mu(X) - \mu(X) \mid X \} \big] \leq \lVert b_\pi \rVert_\bbP \lVert b_\mu \rVert_\bbP
    $$
    where the inequality follows by Cauchy-Schwarz and the definition of $b_\eta = \bbE \{ \widehat \eta(X) - \eta(X) \mid X \}$.  
    
    \medskip
    \large    
    \noindent \textbf{$R_{2,n}$ and bounding the variance of $\widehat \psi_n$:}
    
    \normalsize
    \noindent The final term in \eqref{eq:decomp}, $R_{2,n}$, is centered and mean-zero.  The statement in Lemma~\ref{lem:expansion} is implied by Chebyshev's inequality after bounding the variance of $R_{2,n}$.  Thus, the rest of this proof is devoted to a bound on $\bbV (R_{2,n})$, which must account for randomness across both the estimation and training samples.
    
    \medskip
    
    Since $\bbE \{ \widehat \varphi(Z) - \varphi(Z) \}$ is not random, and by successive applications of the law of total variance, we have
    \begin{align}
        \bbV \left[ (\bbP_n - \bbE) \{ \widehat \varphi(Z) - \varphi(Z) \} \right] &= \bbE \left( \bbV \Big[ \bbP_n \{ \widehat \varphi(Z) - \varphi(Z) \} \mid X^n_\varphi, D_\pi, D_\mu \Big] \right) \nonumber \\
        &+ \bbV \left( \bbE \Big[ \bbP_n \{ \widehat \varphi(Z) - \varphi(Z) \} \mid X^n_\varphi, D_\pi, D_\mu \Big] \right) \nonumber \\
        &= \bbE \left( \bbV \Big[ \bbP_n \{ \widehat \varphi(Z) - \varphi(Z) \} \mid X^n_\varphi, D_\pi, D_\mu \Big] \right) \label{eq:app_a1} \\
        &\hspace{0.1in} + \bbE \left\{ \bbV \left( \bbE \Big[ \bbP_n \{ \widehat \varphi(Z) - \varphi(Z) \} \mid X^n_\varphi, D_\pi, D_\mu \Big] \mid X^n_\varphi \right) \right\} \label{eq:app_a2} \\
        &\hspace{0.1in} + \bbV \left\{ \bbE \left( \bbE \Big[ \bbP_n \{ \widehat \varphi(Z) - \varphi(Z) \} \mid X^n_\varphi, D_\pi, D_\mu \Big] \mid X^n_\varphi \right) \right\} \label{eq:app_a3}
    \end{align}
    where $X_\varphi^n$ are the covariates in the estimation data.  Expression \eqref{eq:app_a1} can be upper bounded using the fact that the data are iid and $\bbV(X) \leq \bbE(X^2)$:
    \begin{align*}
        \bbE \left( \bbV \Big[ \bbP_n \{ \widehat \varphi(Z) - \varphi(Z) \} \mid X^n_\varphi, D_\pi, D_\mu \Big] \right) = \bbE \left[ \frac1n \bbV \Big\{ \widehat \varphi(Z) - \varphi(Z) \mid X^n_\varphi, D_\pi, D_\mu \Big\} \right] \leq \frac{\bbE \Big[ \{ \widehat \varphi(Z) - \varphi(Z) \}^2 \Big]}{n}.
    \end{align*}
    Similarly expression \eqref{eq:app_a3} can be upper bounded using linearity of expectation, iid data, and that $\bbV(X) \leq \bbE(X^2)$ and Jensen's inequality:
    \begin{align*}
        \bbV \left\{ \bbE \left( \bbE \Big[ \bbP_n \{ \widehat \varphi(Z) - \varphi(Z) \} \mid X^n_\varphi, D_\pi, D_\mu \Big] \mid X^n_\varphi \right) \right\} &= \bbV \left( \bbE \Big[ \bbP_n \{ \widehat \varphi(Z) - \varphi(Z) \} \mid X^n_\varphi \Big] \right) \\
        &= \bbV \left( \bbP_n \Big[ \bbE \{ \widehat \varphi(Z) - \varphi(Z)\mid X^n_\varphi  \} \Big] \right) \\
        &= \frac{1}{n^2} \sum_{i=1}^{n} \bbV \Big[ \bbE \{ \widehat \varphi(Z_i) - \varphi(Z_i)\mid X^n_\varphi  \} \Big] \\
        &\leq \frac{\bbE \Big[ \{ \widehat \varphi(Z) - \varphi(Z) \}^2 \Big]}{n}.
    \end{align*}
    \noindent Finally, for expression \eqref{eq:app_a2}, by linearity of expectation, and the definition of $\widehat b_\varphi(X_i)$ and $\Sigma_n$, we have 
    \begin{align*}
        \bbE \left\{ \bbV \left( \bbE \Big[ \bbP_n \{ \widehat \varphi(Z) - \varphi(Z) \} \mid X^n_\varphi, D_\pi, D_\mu \Big] \mid X^n_\varphi \right) \right\} &= \bbE \left[ \bbV \left\{ \frac{1}{n} \sum_{i=1}^{n} \widehat b_\varphi(X_i) \mid X_\varphi^n \right\} \right] \\
        &= \frac{1}{n^2} \sum_{i=1}^{n} \sum_{j=1}^{n} \bbE \left[ \cov \left\{ \widehat b_\varphi (X_i), \widehat b_\varphi (X_j) \mid X_\varphi^n \right\} \right] \\
        &= \frac{1}{n^2} \one^T \Sigma_n \one
    \end{align*}
    where $\one$ the $n$-length vector of $1$'s.  Since $\Sigma_n$ is positive semi-definite and symmetric, $\Sigma_n = Q \Lambda Q^T$ where $Q$ is the orthonormal eigenvector matrix and $\Lambda = \text{diag}(\lambda_1, ..., \lambda_n)$ is the diagonal eigenvalue matrix.  Then,
    $$
    \one^T \Sigma_n \one = \one^T Q \Lambda Q^T \one = \sum_{i=1}^{n} \lambda_i \lVert q_i \rVert^2 = \sum_{i=1}^{n} \lambda_i \leq n \rho(\Sigma_n)
    $$
    where the third equality follows because the $q_i$ are normalized, and the inequality follows by the definition of the spectral radius.  Therefore, $\frac{1}{n^2} \one^T \Sigma_n \one \leq \frac1n \rho(\Sigma_n)$, and the result follows.
\end{proof}

\propspectral*

\begin{proof}
    Since the spectral radius of a matrix is less than its Frobenius norm and the data are iid, 
    $$
    \frac{\rho(\Sigma_n)}{n} \leq \frac{1}{n} \bbE \left[ \bbV \left\{ \widehat b_\varphi (X) \mid X_\varphi^n  \right\} \right] + \frac{n-1}{n}  \bbE \left[ \cov_{i\neq j} \left\{ \widehat b_\varphi (X_i), \widehat b_\varphi (X_j) \mid X_\varphi^n  \right\} \right].
    $$
    For the first summand, we have
    $$
    \frac1n  \bbE \left[ \bbV \left\{ \widehat b_\varphi (X) \mid X_\varphi^n  \right\} \right] \leq \frac{\bbE \lVert \widehat \varphi - \varphi \rVert_\bbP^2}{n}
    $$
    because $\bbV(X) \leq \bbE(X^2)$.  For $i \neq j$, we must analyze the covariance term in more detail. Omitting arguments (e.g., $\pi_i \equiv \pi(X_i)$), 
    \begin{align*}
        \bbE &\left[ \cov \left\{ \widehat b_\varphi (X_i), \widehat b_\varphi (X_j) \mid X_\varphi^n \right\} \right] \\
        &= \bbE \left\{ \cov \Big[ \bbE \{ \widehat \varphi(Z_i) - \varphi(Z_i)  \mid X^n_\varphi, D_\pi, D_\mu \}, \bbE \{ \widehat \varphi(Z_j) - \varphi(Z_j)  \mid X^n_\varphi, D_\pi, D_\mu \} \mid X^n_\varphi \Big] \right\} \\
        &= \bbE \left[ \cov \Big\{ ( \widehat \pi_i - \pi_i ) ( \widehat \mu_i - \mu_i ), (\widehat \pi_j - \pi_j) (\widehat \mu_j - \mu_j) \mid X_i, X_j \Big\} \right] \\
        &= \bbE \Big[ \bbE \Big\{ ( \widehat \pi_i - \pi_i ) ( \widehat \mu_i - \mu_i ) (\widehat \pi_j - \pi_j) (\widehat \mu_j - \mu_j) \mid X_i, X_j \Big\} \\
        &\hspace{1in} - \bbE \Big\{ ( \widehat \pi_i - \pi_i ) ( \widehat \mu_i - \mu_i ) \mid X_i, X_j \Big\} \bbE \Big\{ (\widehat \pi_j - \pi_j) (\widehat \mu_j - \mu_j) \mid X_i, X_j \Big\} \Big] \\
        &= \bbE \left[ \bbE \Big\{ ( \widehat \pi_i - \pi_i ) (\widehat \pi_j - \pi_j) \mid X_i, X_j \Big\} \bbE \Big\{ ( \widehat \mu_i - \mu_i )(\widehat \mu_j - \mu_j) \mid X_i, X_j \Big\} \right] \\
        &\hspace{0.5in} - \bbE \left\{ \bbE \big( \widehat \pi_i - \pi_i  \mid X_i \big) \bbE \big( \widehat \mu_i - \mu_i \mid X_i \big) \bbE \big( \widehat \pi_j - \pi_j \mid X_j \big) \bbE \big( \widehat \mu_j - \mu_j \mid X_j \big) \right\} \\
        &\hspace{-0.2in}= \bbE \Big[ \left\{ \cov \Big( \widehat \pi_i, \widehat \pi_j \mid X_i, X_j \Big) + \bbE \Big( \widehat \pi_i - \pi_i \mid X_i \Big) \bbE \Big( \widehat \pi_j - \pi_j \mid X_j \Big) \right\} \cdot \\
        &\hspace{0.2in}\left\{ \cov \Big( \widehat \mu_i, \widehat \mu_j \mid X_i, X_j \Big) + \bbE \Big( \widehat \mu_i - \mu_i \mid X_i \Big) \bbE \Big( \widehat \mu_j - \mu_j \mid X_j \Big) \right\}  \Big] \\
        &- \bbE \left\{ \bbE \big( \widehat \pi_i - \pi_i  \mid X_i \big) \bbE \big( \widehat \mu_i - \mu_i \mid X_i \big) \bbE \big( \widehat \pi_j - \pi_j \mid X_j \big) \bbE \big( \widehat \mu_j - \mu_j \mid X_j \big) \right\} \\
        &\hspace{-0.2in} = \bbE \left\{ \cov \Big( \widehat \pi_i, \widehat \pi_j \mid X_i, X_j \Big)  \bbE (\widehat \mu_i - \mu_i \mid X_i) \bbE(\widehat \mu_j - \mu_j \mid X_j) \right\} + \stepcounter{equation}\tag{\theequation}\label{eq:three_summands1} \\
        &+ \bbE \left\{ \cov \Big( \widehat \mu_i, \widehat \mu_j \mid X_i, X_j \Big)  \bbE (\widehat \pi_i - \pi_i \mid X_i) \bbE( \widehat \pi_j - \pi_j \mid X_j) \right\} \stepcounter{equation}\tag{\theequation}\label{eq:three_summands2} \\
        &+ \bbE \left\{ \cov \Big( \widehat \pi_i, \widehat \pi_j \mid X_i, X_j \Big)  \cov \Big( \widehat \mu_i, \widehat \mu_j \mid X_i, X_j \Big) \right\} \stepcounter{equation}\tag{\theequation}\label{eq:three_summands3}
    \end{align*}
    where the first equality follows by definition, the second and third by the definition of $\widehat \varphi, \varphi$, and covariance, the fourth by the independence of the training datasets, the fifth again by the definition of covariance and because $\pi_i, \pi_j, \mu_i, \mu_j$ are not random conditional on $X_i, X_j$, and the final line by canceling terms. 
    
    \medskip
    
    For \eqref{eq:three_summands1}, 
    \begin{align*}
        &\bbE \left\{ \cov \Big( \widehat \pi_i, \widehat \pi_j \mid X_i, X_j \Big)   \bbE (\widehat \mu_i - \mu_i \mid X_i) \bbE(\widehat \mu_j - \mu_j \mid X_j)  \right\} \leq \\
        &\hspace{1.5in} \bbE \left[ \left| \cov \Big\{ \widehat \pi(X_i), \widehat \pi(X_j) \mid X_i, X_j \Big\} \right| \right] \sup_{x_i, x_j \in \mathcal{X}} \Big| \bbE \{ \widehat \mu(x_i) - \mu(x_i) \} \bbE \{ \widehat \mu(x_j) - \mu(x_j) \} \Big| \\
        &\hspace{1.5in}= \bbE \left[ \left| \cov \Big\{ \widehat \pi(X_i), \widehat \pi(X_j) \mid X_i, X_j \Big\} \right| \right] \left\{ \sup_{x \in \mathcal{X}} \Big| \bbE \{ \widehat \mu(x) - \mu(x) \} \Big|\right\}^2 \\
        &\hspace{1.5in}\leq \bbE \left[ \left| \cov \Big\{ \widehat \pi(X_i), \widehat \pi(X_j) \mid X_i, X_j \Big\} \right| \right] \sup_{x \in \mathcal{X}} \bbE \{ \widehat \mu(x) - \mu(x) \}^2 \\
        &\hspace{1.5in} \equiv  \bbE \left[ \left| \cov \Big\{ \widehat \pi(X_i), \widehat \pi(X_j) \mid X_i, X_j \Big\} \right| \right] \lVert b_\mu^2 \rVert_\infty
    \end{align*}
    where the first inequality is H\"{o}lder's inequality, the second because $|ab| = |a||b|$, the penultimate by Jensen's inequality, and the final by the definition of $\lVert b_\mu \rVert_\infty$. The same result applies for \eqref{eq:three_summands2} with $\mu$ and $\pi$ swapped.  Next, notice that, 
    \begin{align*}
        \cov \Big\{ \widehat \pi(X_i), \widehat \pi(X_j) \mid X_i, X_j \Big\} &= \bbE \left( \Big[ \widehat \pi(X_i) - \bbE \{ \widehat \pi(X_i) \mid X_i, X_j \} \Big] \Big[ \widehat \pi(X_j) - \bbE \{ \widehat \pi(X_j) \mid X_i, X_j \} \Big] \mid X_i, X_j \right)    \\
        &= \bbE \left( \Big[ \widehat \pi(X_i) - \bbE \{ \widehat \pi(X_i) \mid X_i \} \Big] \Big[ \widehat \pi(X_j) - \bbE \{ \widehat \pi(X_j) \mid X_j \} \Big] \mid X_i, X_j \right) \\ 
        &\leq \sqrt{ \bbE \left( \Big[ \widehat \pi(X_i) - \bbE \{ \widehat \pi(X_i) \mid X_i \} \Big]^2 \mid X_i \right) \bbE \left( \Big[ \widehat \pi(X_j) - \bbE \{ \widehat \pi(X_j) \mid X_j \} \Big]^2 \mid X_j \right)   } \\
        &= \sqrt{ \bbV \{ \widehat \pi(X_i) \mid X_i \} \bbV \{ \widehat \pi(X_j) \mid X_j \} } 
    \end{align*}
    where the first line follows by definition, the second because $\widehat \pi(X_i) \ind X_j$ for $X_i \neq X_j$, the third by Cauchy-Schwarz, and the fourth by the definition of the variance. Therefore, for \eqref{eq:three_summands3},  
    \begin{align*}
        \bbE \left\{ \cov \Big( \widehat \pi_i, \widehat \pi_j \mid X_i, X_j \Big)  \cov \Big( \widehat \mu_i, \widehat \mu_j \mid X_i, X_j \Big) \right\} &\leq 
        \bbE \left\{ \sqrt{ \bbV \{ \widehat \pi(X_i) \mid X_i \} \bbV \{ \widehat \pi(X_j) \mid X_j \} } \left| \cov \Big( \widehat \mu_i, \widehat \mu_j \mid X_i, X_j \Big) \right| \right\} \\
        &\leq \sup_{x_i, x_j \in \mathcal{X}} \sqrt{ \bbV \{ \widehat \pi(x_i) \} \bbV \{ \widehat \pi(x_j) \} } \bbE \left\{ \left| \cov \Big( \widehat \mu_i, \widehat \mu_j \mid X_i, X_j \Big) \right| \right\} \\
        &= \sup_{x}  \bbV \{ \widehat \pi(x) \} \bbE \left\{ \left| \cov \Big( \widehat \mu_i, \widehat \mu_j \mid X_i, X_j \Big) \right| \right\} \\
        &\equiv \lVert s_\pi^2 \rVert_\infty \bbE \left\{ \left| \cov \Big( \widehat \mu_i, \widehat \mu_j \mid X_i, X_j \Big) \right| \right\} 
    \end{align*}
    where the first line follows by H\"{o}lder's inequality, the second by the argument in the previous paragraph, the third because $|ab| = |a||b|$, and the last line follows by definition of $\lVert s_\pi^2 \rVert_\infty$. 
    
    \medskip
    
    The result in Proposition~\ref{prop:spectral} follows by repeating the process in the previous paragraph with the roles of $\pi$ and $\mu$ reversed. In fact, Proposition~\ref{prop:spectral} can be improved because we can take the minimum rather than the sum of the variances at the final step so that
    \begin{align}
        \frac{\rho(\Sigma_n)}{n} \leq \frac{\bbE \lVert \widehat \varphi - \varphi \rVert_\bbP^2}{n} &+ \lVert b_{\pi}^2 \rVert_\infty \bbE \Big[ \left| \cov \{ \widehat \mu(X_i), \widehat \mu(X_j) \mid X_i, X_j \} \right| \Big] + \lVert b_{\mu}^2 \rVert_\infty \bbE \Big[ \left| \cov \{ \widehat \pi(X_i), \widehat \pi(X_j) \mid X_i, X_j \} \right| \Big] \nonumber \\
        &\hspace{-1.2in}+ \min \left( \lVert s_\pi^2 \rVert_\infty \bbE \Big[ \big| \cov \{ \widehat \mu(X_i), \widehat \mu(X_j) \mid X_i, X_j \} \big| \Big], \lVert s_\mu^2 \rVert_\infty \bbE \Big[ \big| \cov \{ \widehat \pi(X_i), \widehat \pi(X_j) \mid X_i, X_j \} \big| \Big] \right). \label{eq:spectral_minimum}
    \end{align}
    Proposition~\ref{prop:spectral} follows because the minimum in \eqref{eq:spectral_minimum} is upper bounded by the sum. We will also use \eqref{eq:spectral_minimum} subsequently, referring to it in the proof of Theorems~\ref{thm:minimax} and~\ref{thm:inference}.
\end{proof}

\section{k-Nearest-Neighbors and local polynomial regression} \label{app:nf_ests}

In Sections~\ref{sec:unknown}, we defined two linear smoother estimators.  In this section, we state and prove several results for each estimator, including bounds on their bias and variance, as well as bounds on their expected absolute covariance, $\bbE \left[ \left| \cov \left\{ \widehat \eta(X_i), \widehat \eta(X_j) \mid X_i, X_j \right\} \right| \right]$. In the following, we state and prove the results for $Y$ and $\mu(X)$.  All results also apply to $A$ and $\pi(X)$.

\subsection{k-Nearest-Neighbors}

The analysis of the bias of the k-Nearest-Neighbors estimator relies on control of the nearest neighbor distance.  The nearest neighbor distance is well understood, and general results can be be found in, for example, Chapter 6 of \citet{gyorfi2002distribution}, Chapter 2 of \citet{biau2015lectures}, and \citet{dasgupta2021nearest}. By leveraging Assumption~\ref{asmp:bdd_density}, that the density is upper and lower bounded (which is a stronger assumption than generally required), we provide a simple result that is sufficient for our subsequent analysis, which uses similar techniques to those in the proof of Lemma 6.4 (and Problem 6.7) in \citet{gyorfi2002distribution}.
\begin{lemma} \label{lem:knn_distance}
    Suppose  we observe $\{ X_i \}_{i=1}^{n}$ sampled iid from a distribution satisfying Assumption~\ref{asmp:bdd_density}.  Then, for $0 < p \leq 2d$ and $x \in \mathcal{X}$,
    \begin{equation}
        \bbE \lVert X_{(1)} (x) - x \rVert^p \lesssim n^{-p/d}.
    \end{equation}
\end{lemma}
\begin{proof}
    Let $B_r(x)$ denote a ball of radius $r$ centered at $x$. Then,
    \begin{align*}
        \bbE \lVert X_{(1)} (x) - x \rVert^p &= \int_{0}^{\infty} \bbP \left\{ \lVert X_{(1)} (x) - x \rVert^p > t \right\} dt \\
        &= \int_{0}^{\infty} \bbP \left\{ \lVert X_{(1)} (x) - x \rVert > t^{1/p} \right\} dt \\
        &= \int_{0}^{\infty} \bbP \left\{ \lVert X - x \rVert > t^{1/p} \right\}^n dt \\
        &= \int_{0}^{\infty} \Big[ 1 - \bbP \left\{ X \in B_{t^{1/p}}(x) \right\} \Big]^n dt 
    \end{align*}
    where the third line follows because the observations $\{X_i\}_{i=1}^{n}$ are iid.  Then, by Assumption~\ref{asmp:bdd_density}, for all $r > 0$, $\bbP \{ X \in B_r (x) \} \geq cKr^d \wedge 1$, where $c$ is the lower bound on the density and $K$ is a constant arising from the volume of the d-dimensional sphere.  Therefore,
    \begin{align*}
        \int_{0}^{\infty} \Big[ 1 - \bbP \left\{ X \in B_{t^{1/p}}(x) \right\} \Big]^n dt &\leq \int_{0}^{\infty} \Big\{ \left(1 - cK t^{d/p} \right) \vee 0 \Big\}^n dt \\
        &= \int_{0}^{(cK)^{-p/d}} \left(1 - cK t^{d/p} \right)^n dt \\
        &\leq \int_{0}^{(cK)^{-p/d}} \exp \left(- cK nt^{d/p} \right) dt \\
        &\leq \int_{0}^{\infty} \exp \left(- cK nt^{d/p} \right) dt.
    \end{align*}
    where the penultimate line follows because $1-x \leq e^{-x}$ and the final line because $e^{-x} > 0$.
    
    \medskip
    
    Next, notice that
    \begin{align*}
        \int_{0}^{\infty} \exp \left(- cK nt^{d/p} \right) dt &= -(cKn)^{-p/d} \frac{\Gamma( p/d, cKnt^{d/p})}{d / p} \bigg|_0^\infty \\
        &\lesssim n^{-p/d}
    \end{align*}
    where the first line follows from standard rules of integration and where $\Gamma(s, t)$ is the incomplete gamma function, which satisfies $\Gamma(s,x) = \int_x^\infty t^{s-1} e^{-t} dt$, and the second line follows because $\Gamma(p/d, \infty) = 0$ while $\Gamma(p/d, 0), d/p$, and $cK$ are constants that do not depend on $n$.  Therefore,
    \begin{equation}
        \bbE \lVert X_{(1)} (x) - x \rVert^p \lesssim n^{-p/d}.
    \end{equation}
\end{proof}

\noindent The next result provides pointwise bias and variance bounds for the k-Nearest-Neighbors estimator. Notice that the variance scales at the mean squared error rate due to the randomness over the training data .

\begin{lemma} \label{lem:knn_bounds}
    \emph{(k-Nearest-Neighbors Bounds)} Suppose Assumptions~\ref{asmp:dgp},~\ref{asmp:bdd_density} and~\ref{asmp:holder} hold.  Then, if $\widehat \mu(x)$ is a k-Nearest-Neighbors estimator (Estimator~\ref{est:knn}) for $\mu(x)$ constructed on $D_\mu$, 
    \begin{align}
        \sup_{x \in \mathcal{X}} \left| \bbE \{ \widehat \mu(x) - \mu(x) \} \right| &\lesssim \left( \frac{n}{k} \right)^{- \frac{\beta \wedge 1}{d}} \text{ and } \label{eq:knn_bias} \\
        \sup_{x \in \mathcal{X}} \bbV \{ \widehat \mu(x)  \} &\lesssim \frac1k + \left( \frac{n}{k} \right)^{- \frac{2(\beta \wedge 1)}{d}}. \label{eq:knn_var}
    \end{align}
\end{lemma}

\begin{proof}
    We prove the bounds for generic $x$, and the supremum bounds will follow because $\mathcal{X}$ is assumed compact in Assumption~\ref{asmp:bdd_density}. Note that, if $\mu \in \text{H\"{o}lder}(\beta)$ for $\beta > 1$ then $\mu \in \text{H\"{o}lder}(1)$ (in other words, $\mu$ is Lipschitz). For the bias in  \eqref{eq:knn_bias}, we have
    \begin{align*}
        \left| \bbE \{ \widehat \mu(x) - \mu(x) \} \right| &= \left| \bbE \left\{ \frac1k \sum_{i=1}^{n} \one \Big( \lVert X_i - x \rVert \leq \lVert X_{(k)} (x) - x \rVert \Big) Y_i - \mu(x) \right\} \right| \\
        &= \left| \frac1k \sum_{j=1}^{k} \bbE \left[ \mu \{ X_{(j)} (x) \} - \mu(x)  \right] \right| \\
        &\lesssim  \left|\frac1k \sum_{j=1}^{k} \bbE \{ \lVert X_{(j)}(x) - x \rVert^{\beta \wedge 1}  \} \right| \\
        &\leq \frac1k \sum_{j=1}^{k} \bbE \lVert X_{(j)}(x) - x \rVert^{\beta \wedge 1}
    \end{align*}
    where the first line follows by definition, the second by iterated expectations on the training covariates and then by definition, the first inequality by the smoothness assumption on $\mu$, and the second by Jensen's inequality. 
    
    \medskip
    
    For $k = 1$, one can invoke Lemma~\ref{lem:knn_distance} directly, giving
    \begin{equation}
        \left| \bbE \{ \widehat \mu(x) - \mu(x) \} \right|  \leq n^{\frac{-\beta \wedge 1}{d}}.
    \end{equation}
    Otherwise, split the $n$ datapoints into $k+ 1$ subsets, where the first $k$ subsets are of size $\lfloor n / k \rfloor$. Let $\widetilde{X}_{(1)}^j (x) $ denote the nearest neighbor to $x$ in the $j$th split.  Then, the following deterministic inequality holds:
    $$
    \frac1k \sum_{j=1}^{k} \bbE \lVert X_{(j)}(x) - x \rVert^{\beta \wedge 1} \leq \frac{1}{k} \sum_{j=1}^{k} \bbE\lVert \widetilde{X}_{(1)}^j (x) - x \rVert^{\beta \wedge 1}.
    $$
    Thus, applying Lemma~\ref{lem:knn_distance} to $\bbE\lVert \widetilde{X}_{(1)}^j (x) - x \rVert^{\beta \wedge 1}$ yields
    \begin{equation}
        \left| \bbE \{ \widehat \mu(x) - \mu(x) \} \right|  \lesssim \left( \lfloor n / k \rfloor \right)^{\frac{-\beta \wedge 1}{d}} \asymp \left( n / k \right)^{\frac{-\beta \wedge 1}{d}}.
    \end{equation}
    For the variance in \eqref{eq:knn_var}, we have
    \begin{align*}
        \bbV \{ \widehat \mu(x) \} &= \bbV \big[ \bbE \{ \widehat \mu(x) \mid X_\mu^n \} \big] + \bbE \big[ \bbV \{ \widehat \mu(x) \mid X_\mu^n \} \big] \\
        &= \bbV \big[ \bbE \{ \widehat \mu(x) - \mu(x) \mid X_\mu^n \} \big] + \bbE \big[ \bbV \{ \widehat \mu(x) \mid X_\mu^n \} \big] \\
        &\leq \bbE \big[ \bbE \{ \widehat \mu(x) - \mu(x) \mid X_\mu^n \}^2 \big] + \bbE \big[ \bbV \{ \widehat \mu(x) \mid X_\mu^n \} \big] \\
        &\lesssim  \left( \frac{n}{k} \right)^{- \frac{2(\beta \wedge 1)}{d}}  + \frac{1}{k}
    \end{align*}
    where the first line follows by the law of total variance, the second because $\mu(x)$ is non-random, the third because $\bbV(X) \leq \bbE(X^2)$, the fourth by the bound on the bias, and the final line because $\{ Y_1, \dots, Y_n \}$ are independent conditional on $X_\mu^n$ and have bounded conditional variance by Assumption~\ref{asmp:dgp}.
    
    \medskip
    
    \noindent The supremum bound follows since the proof holds for arbitrary $x$ and $\mathcal{X}$ is compact by Assumption~\ref{asmp:bdd_density}.
\end{proof}

\noindent The final result of this section provides a bound on the covariance term that appears in Proposition~\ref{prop:spectral} and Lemma~\ref{lem:covariance}.

\begin{lemma} \label{lem:knn_covariance} 
    \emph{(k-Nearest-Neighbors covariance bound)} Suppose Assumptions~\ref{asmp:dgp} and~\ref{asmp:bdd_density} hold and $\widehat \mu(x)$ is a k-Nearest-Neighbors estimator (Estimator~\ref{est:knn}) for $\mu(x)$ constructed on $D_\mu$.  Then,
    $$
    \bbE \big[ \left| \cov \{ \widehat \mu(X_i), \widehat \mu(X_j) \mid X_i, X_j \} \right| \big] \lesssim \left\{ \frac1k + \left( \frac{n}{k} \right)^{- \frac{2(\beta \wedge 1)}{d}} \right\} \left( \frac{k}{n} \right).
    $$
\end{lemma}

\begin{proof}
    We have
    \begin{align*}
        \bbE \big[ \left| \cov \{ \widehat \mu(X_i), \widehat \mu(X_j) \mid X_i, X_j \} \right| \big] &= \bbE \big[ \left| \cov \{ \widehat \mu(X_i), \widehat \mu(X_j) \mid X_i, X_j \} \right| \one ( \lVert X_i - X_j \rVert \leq \lVert X_i - X_{(2k)} (X_i) \rVert)  \big] \\
        &\leq \sup_{x_i, x_j} \left| \cov \{ \widehat \mu(x_i), \widehat \mu(x_j) \} \right| \bbP \left( \lVert X_i - X_j \rVert \leq \lVert X_i - X_{(2k)} (X_i) \rVert \right) \\
        &\leq \sup_{x \in \mathcal{X}} \bbV \{ \widehat \mu(x) \}  \bbP \left( \lVert X_i - X_j \rVert \leq \lVert X_i - X_{(2k)} (X_i) \rVert \right) \\
        &\lesssim \left\{ \frac1k + \left( \frac{n}{k} \right)^{- \frac{2(\beta \wedge 1)}{d}} \right\} \bbP \left( \lVert X_i - X_j \rVert \leq \lVert X_i - X_{(2k)} (X_i) \rVert \right)
    \end{align*}
    where the first line follows because $\cov \{ \widehat \mu(X_i), \widehat \mu(X_j) \mid X_i, X_j \} = 0$ when $\lVert X_i - X_j \rVert > \lVert X_i - X_{(2k)} (X_i) \rVert$, the second by H\"{o}lder's inequality, and the final line by Lemma~\ref{lem:knn_bounds}.
    
    \medskip
    
    It remains to bound $\bbP ( \lVert X_i - X_j \rVert \leq \lVert X_i - X_{(2k)} (X_i) \rVert )$. We have
    \begin{align*}
        \bbP ( \lVert X_i - X_j \rVert \leq \lVert X_i - X_{(2k)} (X_i) \rVert ) &= \bbE \left\{ \bbP  ( \lVert X_i - X_j \rVert \leq \lVert X_i - X_{(2k)} (X_i) \rVert  \mid X_i) \right\} \\
        &= \frac{2k}{n + 1} \lesssim \frac{k}{n}.
    \end{align*}
    where the first line follows by iterated expectations.  The second line follows because $\bbP  ( \lVert X_i - X_j \rVert \leq \lVert X_i - X_{(2k)} (X_i) \rVert  \mid X_i)$ is the probability that $X_j$ is one of the $2k$ closest points to $X_i$ out of $X_j$ and the $n$ training data points.  Because $X_j$ and the training data are iid, $X_j$ has an equal chance of being any order neighbor to $X_i$, and therefore the probability it is in the $2k$ closest points is $\frac{2k}{n+1}$.
    
    \medskip
    
    Therefore, we conclude that 
    $$
    \bbE \big[ \left| \cov \{ \widehat \mu(X_i), \widehat \mu(X_j) \mid X_i, X_j \} \right| \big] \lesssim \left\{ \frac1k + \left( \frac{n}{k} \right)^{- \frac{2(\beta \wedge 1)}{d}} \right\} \left( \frac{k}{n} \right).
    $$
\end{proof}

\subsection{Local polynomial regression}

The proofs in this subsection follow closely to those in \citet{tsybakov2009introduction}.  The main difference is that we translate the conditional bounds into marginal bounds, like in \citet{kennedy2023towards}. Let
\begin{align}
    A_n &= \one \left( \widehat Q \text{ is invertible} \right), \label{eq:invertible} \\
    \xi_n &:= \frac{\bbP_n \{ \one ( \lVert X - x \rVert \leq h) \}}{h^d} \text{, and } \\
    \lambda_n &:= \lambda_{\max}\left( \widehat Q^{-1} \right).
\end{align}
First, we note that the weights reproduce polynomials up to degree $\lceil d / 2 \rceil$ by the construction of the estimator in Estimator~\ref{est:lpr} (\citet{tsybakov2009introduction} Proposition 1.12) as long as $A_n = 1$ (i.e., $\widehat Q$ is invertible).

\medskip

We will state results for the bias and variance of the estimator conditionally on the training covariates, assuming $\widehat Q$ is invertible, and keeping $\lambda_n$ and $\xi_n$ in the results. Then, we will argue that $\lambda_n$ and $\xi_n$ are bounded in probability and therefore that (i) $\widehat Q$ is invertible with probability converging to one appropriately quickly, and (ii) the relevant bias and variance bounds hold in probability. Next, we demonstrate that the weights have the desired localizing properties in the following result (\citet{tsybakov2009introduction} Lemma 3).
\begin{proposition} \label{prop:lpr_localization}
    Suppose Assumptions~\ref{asmp:dgp} and~\ref{asmp:bdd_density} hold, $\widehat \mu(x)$ is a local polynomial regression estimator (Estimator~\ref{est:lpr}) for $\mu(x)$ constructed on $D_\mu$, and $\widehat Q$ is invertible. Let 
    $$
    w_i (x; X_\mu^n) = \frac{1}{nh^d} b(0)^T \widehat Q^{-1} b \left( \frac{X_i - x}{h} \right) K \left( \frac{X_i - x}{h} \right).
    $$
    Then,
    \begin{align}
        \sup_{i, x} | w_i (x; X_\mu^n) | &\lesssim \frac{\lambda_n}{ nh^d}, \label{eq:lpr_localization} \\
        \sum_{i=1}^{n} | w_i (x; X_\mu^n) | &\lesssim \lambda_n \xi_n  \label{eq:lpr_bdd_max_wt}, \text{ and } \\
        w_i(x; X_\mu^n) &= 0 \text{ when } \lVert X_i - x \lVert > h. \label{eq:local_kernel}
    \end{align}
\end{proposition}

\begin{proof}
    \eqref{eq:local_kernel} follows by the definition of the kernel in Estimator~\ref{est:lpr}.  For \eqref{eq:lpr_localization}, 
    \begin{align*}
        |  w_i(x; X_\mu^n) | &=  \left| \frac{1}{nh^d} b(0)^T \widehat Q^{-1} b \left( \frac{X_i - x}{h} \right) K \left( \frac{X_i - x}{h} \right)  \right| \\
        &\leq \frac{1}{nh^d} \lVert b(0)^T \rVert \bigg\lVert \widehat Q^{-1} b \left( \frac{X_i - x}{h} \right) K \left( \frac{X_i - x}{h} \right) \bigg\rVert \\
        &\leq \frac{\lambda_n}{ nh^d} \bigg\lVert b \left( \frac{X_i - x}{h} \right) K \left( \frac{X_i - x}{h} \right) \bigg\rVert \\
        &\lesssim  \frac{\lambda_n}{nh^d} \left\lVert b \left( \frac{X_i - x}{h} \right) \right\rVert \one \left( \lVert X_i - x \rVert \leq h \right)  \\
        &\lesssim \frac{\lambda_n \one \left( \lVert X_i - x \rVert \leq h \right)}{ nh^d}
    \end{align*}
    where the first line follows by definition, the second by Cauchy-Schwarz, the third because $\lVert b(0)^T \rVert = 1$ and the definition of $\lambda_n$, the fourth because the kernel is localized by definition in Estimator~\ref{est:lpr}, and the last by Assumption~\ref{asmp:bdd_density} and compact support $\mathcal{X}$. \eqref{eq:lpr_localization} then follows because the indicator function is at most $1$. Finally, for  \eqref{eq:lpr_bdd_max_wt}, 
    \begin{align*}
        \sum_{i=1}^{n} | w_i (x; X_\mu^n) | &= \sum_{i=1}^{n} \left| \frac{1}{nh^d} b(0)^T \widehat Q^{-1} b \left( \frac{X_i - x}{h} \right) K \left( \frac{X_i - x}{h} \right)  \right| \\
        &\lesssim \frac{\lambda_n}{ nh^d} \sum_{i=1}^{n} \one \left( \lVert X_i - x \rVert \leq h \right) = \lambda_n \xi_n 
    \end{align*}
    where the second line follows by the same arguments as before and the definition of $\xi_n$.
\end{proof}

\noindent Next, we prove conditional bias and variance bounds (\citet{tsybakov2009introduction} Proposition 1.13).
\begin{proposition} \label{prop:lpr_fixed_bias_var}
    Suppose Assumptions~\ref{asmp:dgp},~\ref{asmp:bdd_density}, and~\ref{asmp:holder} hold and $\widehat \mu(x)$ is a local polynomial regression estimator (Estimator~\ref{est:lpr}) for $\mu(x)$ constructed on $D_\mu$.  Let $A_n$ denote the event that $\widehat Q$ is invertible, as in \eqref{eq:invertible}. Then,
    \begin{equation} \label{eq:lpr_fixed_bias}
        \left| \bbE \{ \widehat \mu(x) - \mu(x) \mid X_\mu^n, A_n = 1 \} \right| \lesssim \lambda_n \xi_n h^{\beta \wedge \lceil d / 2 \rceil}
    \end{equation}
    and
    $$
    \bbV \{ \widehat \mu(x) \mid X_\mu^n \} \lesssim \frac{\lambda_n^2 \xi_n}{ nh^d}.
    $$
\end{proposition}

\begin{proof}
    Notice first that
    \begin{align*}
        \bbE \{ \widehat \mu(x) - \mu(x) \mid X_\mu^n, A_n = 1 \} &= \bbE \left\{ \sum_{i=1}^{n} w_i (x; X_\mu^n ) Y_i - \mu(x) \mid X_\mu^n, A_n = 1 \right\} \\
        &= \sum_{i=1}^{n} w_i(x; X_\mu^n) \mu(X_i) - \mu(x) \\
        &= \sum_{i=1}^{n} w_i(x; X_\mu^n) \{ \mu(X_i) - \mu(x) \}
    \end{align*} 
    since the weights sum to 1.  Let $\gamma = \beta \wedge \lceil d / 2 \rceil$, and consider the Taylor expansion of $\mu(X_i) - \mu(x)$ up to order $\lfloor \gamma \rfloor$:
    \begin{align*}
        &\left| \bbE \{ \widehat \mu(x) - \mu(x) \mid X_\mu^n, A_n = 1 \} \right| \\
        &= \sum_{i=1}^{n} w_i(x; X_\mu^n) \left[ \sum_{|k| = \lfloor \gamma \rfloor} \int_0^1 (1-t)^{\lfloor \gamma \rfloor - 1} \left\{ D^k \mu(x + t(X_i - x)) - D^k \mu(x) \right\}dt(X_i - x)^k \right]  \\
        &\lesssim \sum_{i=1}^{n} w_i (x; X_\mu^n) \lVert X_i - x \rVert^\gamma \\
        &\leq \sum_{i=1}^{n} |w_i (x; X_\mu^n) | h^\gamma \\
        &\lesssim \lambda_n \xi_n h^{\gamma}  \equiv \lambda_n \xi_n h^{\beta \wedge \lceil d / 2 \rceil}
    \end{align*}
    where the first line follows by a multivariate Taylor expansion of $\mu(X_i) - \mu(x)$ and the reproducing property of local polynomial regression, the second by Assumption~\ref{asmp:holder}, the third by \eqref{eq:local_kernel} and the fourth by \eqref{eq:lpr_bdd_max_wt}.
    
    \medskip
    
    For the variance, we have 
    \begin{align*}
        \bbV \{ \widehat \mu(x) \mid X_\mu^n \} &= \sum_{i=1}^{n}  w_i (x; X_\mu^n)^2 \bbV (Y_i \mid X_i) \\
        &\lesssim \sum_{i=1}^{n}  w_i (x; X_\mu^n)^2 \\
        &\leq \sup_{i,x} |w_i (x; X_\mu^n) | \sum_{i=1}^{n} |w_i (x; X_\mu^n) | \\
        &\lesssim \frac{\lambda_n^2 \xi_n}{ nh^d},
    \end{align*}
    where the second line follows by Assumption~\ref{asmp:dgp}, and the last line by equations \eqref{eq:lpr_localization} and \eqref{eq:lpr_bdd_max_wt}.
\end{proof}

In the next result, we provide a bound on the probability that the minimum eigenvalue of $\widehat Q$ equals zero, which informs both an upper bound on $\lambda_n$ and a bound on the probability that $\widehat Q$ is invertible.

\begin{proposition} \label{prop:matrix_chernoff}
    Suppose Assumption \ref{asmp:bdd_density} holds, $\widehat \mu(x)$ is a local polynomial regression estimator (Estimator~\ref{est:lpr}) for $\mu(x)$ constructed on $D_\mu$. Then, for some $c > 0$
    \begin{equation}
        \bbP \left\{ \lambda_{\min} (\widehat Q) \leq c \right\} \lesssim \exp \left( -nh^d \right).
    \end{equation}
\end{proposition}

\begin{proof}
    By the Matrix Chernoff inequality (e.g., \citet{tropp2015introduction} Theorem 5.1.1),
    $$
    \bbP\left\{ \lambda_{\min}(\widehat Q) \leq \frac{\lambda_{\min} \left\{ \bbE \left(\widehat Q \right)\right\}}{2} \right\} \lesssim \exp \left[ \frac{\lambda_{\min} \left\{ \bbE \left(\widehat Q \right)\right\}}{L} \right]
    $$
    where $L := \max_{i=1}^{n} \rho \left\{  \frac{1}{nh^d} b\left(\frac{X_i - x}{h} \right) K \left( \frac{X_i - x}{h}\right)b\left(\frac{X_i - x}{h} \right)^T \right\}$ and, as a reminder, $\rho(A)$ denotes the spectral radius of a matrix $A$.  By the boundedness of $b$ and the kernel, $L = O\left(\frac{1}{nh^d}\right)$.  Meanwhile,
    \begin{align*}
        \bbE \left(\widehat Q \right) &= \bbE \left\{ \frac{1}{h^d} b\left(\frac{X - x}{h} \right) K \left( \frac{X - x}{h}\right)b\left(\frac{X - x}{h} \right)^T \right\} \\
        &= \int b(u) K(u) b(u)^T f(x + uh) du \\
        &= \int_{\lVert u \rVert \leq 1} b(u) b(u)^T f(x + uh) du \asymp I_{ d + \lceil d /2 \rceil \choose \lceil d / 2 \rceil}
    \end{align*}
    where the first line follows by definition and iid data, the second by a change of variables, the third by the definition of the kernel, and the fourth by the lower bounded covariate density in Assumption \ref{asmp:bdd_density} and the definition of the basis. Therefore, $\bbE  \left(\widehat Q \right)$ is proportional to the identity and thus its minimum eigenvalue is proportional to $1$, and the result follows.
\end{proof}

\begin{corollary} \label{cor:big_o_p}
    Suppose Assumption \ref{asmp:bdd_density} holds, $\widehat \mu(x)$ is a local polynomial regression estimator (Estimator~\ref{est:lpr}) for $\mu(x)$ constructed on $D_\mu$. Then,
    \begin{equation}
        \bbP(A_n = 0) \lesssim \exp(-nh^d)
    \end{equation} 
    and, if $nh^d \to \infty$ and $n \to \infty$, then 
    \begin{equation}
        \lambda_n = O_\bbP(1)
    \end{equation}
\end{corollary}
\begin{proof}
    The first result follows because $\widehat Q$ is positive semi-definite by the construction of the basis.  Therefore, it is invertible if its minimum eigenvalue is positive, and the bound follows from Proposition~\ref{prop:matrix_chernoff}. Meanwhile, the second result follows directly from Proposition~\ref{prop:matrix_chernoff}.
\end{proof}

Next, we demonstrate that $\xi_n$ is bounded in probability.  This result relies on the bandwidth decreasing slowly enough that $nh^d \to \infty$ as $n \to \infty$ and the upper bound on the covariate density.
\begin{proposition} \label{prop:xi_n}
    Suppose Assumption \ref{asmp:bdd_density} holds, $\widehat \mu(x)$ is a local polynomial regression estimator (Estimator~\ref{est:lpr}) for $\mu(x)$ constructed on $D_\mu$, and $nh^d \to \infty$ as $n \to \infty$.  Then, $\xi_n = O_\bbP(1)$. 
\end{proposition}
\begin{proof}
    Notice that $\bbE(\xi_n) \asymp 1$ and $\bbV(\xi_n) \lesssim \frac{1}{nh^d}$ by the construction of the kernel, Assumption \ref{asmp:bdd_density}, and Lemma \ref{lem:sphere}. The result follows by the assumption on the bandwidth and Chebyshev's inequality.
\end{proof}

\begin{lemma} \label{lem:lpr_bounds} \emph{(Local polynomial regression bounds)}
    Suppose Assumptions~\ref{asmp:dgp},~\ref{asmp:bdd_density}, and~\ref{asmp:holder} hold, $\widehat \mu(x)$ is a local polynomial regression estimator (Estimator~\ref{est:lpr}) for $\mu(x)$ constructed on $D_\mu$, and $nh^d \to \infty$ as $n \to \infty$. Then,
    \begin{equation}
        \sup_{x \in \mathcal{X}} \left| \bbE \{ \widehat \mu(x) - \mu(x)  \} \right| \lesssim O_\bbP\left( h^{\beta \wedge \lceil d /2 \rceil} \right) + \exp(-nh^d) \label{eq:lpr_bias} 
    \end{equation}
    and
    \begin{equation}
        \sup_{x \in \mathcal{X}} \bbV \{ \widehat \mu(x) \} \lesssim O_\bbP\left( \frac{1}{nh^d} + h^{2 (\beta \wedge \lceil d/2 \rceil)} \right) + \exp(-nh^d). \label{eq:lpr_variance}
    \end{equation}
\end{lemma}

\begin{proof}
    We prove the bounds for generic $x$, and the supremum bounds will follow because $\mathcal{X}$ is compact by Assumption~\ref{asmp:bdd_density}.  Starting with \eqref{eq:lpr_bias}, \begin{align*}
        \left| \bbE \{ \widehat \mu(x) - \mu(x) \} \right| &\leq \bbE \left[ \left| \bbE\{ \widehat \mu(x) - \mu(x) \mid X_\mu^n \} \right| \right] \\
        &\leq \bbE \bigg[ \left| \bbE\{ \widehat \mu(x) - \mu(x) \mid X_\mu^n, A_n = 1 \} \right| \bbP(A_n = 1 \mid X_\mu^n) \\
        &\hspace{0.2in} + \left| \bbE\{ \widehat \mu(x) - \mu(x) \mid X_\mu^n, A_n = 0 \} \right| \bbP(A_n = 0 \mid X_\mu^n)  \bigg] \\
        &\lesssim \bbE \left( \lambda_n \xi_n h^{\beta \wedge \lceil d / 2 \rceil} \right) + \bbP(A_n = 0) \\
        &\lesssim O_\bbP \left( h^{\beta \wedge \lceil d / 2 \rceil} \right) + \exp(-nh^d),
    \end{align*}
    where the first line follows by iterated expectations and Jensen's inequality, the second  by the law of total probability and the triangle inequality, the third by \eqref{eq:lpr_fixed_bias} in Proposition \ref{prop:lpr_fixed_bias_var} for the first term and because the bias is bounded in the second term (by the construction of the estimator and Assumption \ref{asmp:dgp}) and iterated expectations again, and the final line by Corollary \ref{cor:big_o_p} and Proposition \ref{prop:xi_n}.  		
    
    \medskip
    
    For \eqref{eq:lpr_variance}, we have
    \begin{align*}
        \bbV \{ \widehat \mu(x) \} &= \bbV \Big[ \bbE \{ \widehat \mu(x) \mid X_\mu^n \} \Big] + \bbE \Big[ \bbV \{ \widehat \mu(x) \mid  X_\mu^n \} \Big] \\
        &\lesssim \bbV \Big[ \bbE \{ \widehat \mu(x) \mid  X_\mu^n \} \Big] + \bbE \left( \frac{\lambda_n^2 \xi_n}{nh^d} \right) \\
        &= \bbV \Big[ \bbE \{ \widehat \mu(x) \mid  X_\mu^n \} \Big] + O_\bbP\left(\frac{1}{nh^d} \right),
    \end{align*}
    where the first line follows by the law of total variance, the second by Proposition~\ref{prop:lpr_fixed_bias_var}, and the third by Corollary \ref{cor:big_o_p} and Proposition \ref{prop:xi_n}. It remains to bound $\bbV \Big[ \bbE \{ \widehat \mu(x) \mid X_\mu^n \} \Big]$. We have
    \begin{align*}
        \bbV \Big[ \bbE \{ \widehat \mu(x) \mid X_\mu^n \} \Big] &= \bbV \Big[ \bbE \{ \widehat \mu(x) - \mu(x) \mid X_\mu^n \} \Big] \\
        &\leq \bbE \Big[ \bbE \{ \widehat \mu(x) - \mu(x) \mid X_\mu^n \}^2 \Big] \\
        &= \bbE \Big[ \bbE \{ \widehat \mu(x) - \mu(x) \mid X_\mu^n, A_n = 1\}^2 \bbP(A_n = 1 \mid X_\mu^n) \\
        &\hspace{0.2in}+ \bbE \{ \widehat \mu(x) - \mu(x) \mid X_\mu^n, A_n = 0 \}^2  \bbP(A_n = 0 \mid X_\mu^n) \Big] \\
        &\lesssim \bbE \left( \lambda_n^2 \xi_n^2 h^{2\beta \wedge 2\lceil d / 2 \rceil} \right) + \bbP(A_n = 0) \\
        &\lesssim  O_\bbP\left(h^{2\beta \wedge 2\lceil d / 2 \rceil} \right) + \exp(-nh^d),
    \end{align*}
    where first line follows because $\mu(x)$ is not random, the second line because $\bbV(X) \leq \bbE(X^2)$, the third line by the law of total probability, the fourth by \eqref{eq:lpr_fixed_bias} in Proposition \ref{prop:lpr_fixed_bias_var} for the first term and because the bias is bounded in the second term (by the construction of the estimator and Assumption \ref{asmp:dgp}) and iterated expectations again, and the final line by Corollary \ref{cor:big_o_p} and Proposition \ref{prop:xi_n}.
    
    \medskip
    
    \noindent The supremum bound follows since the proof holds for arbitrary $x$ and $\mathcal{X}$ is compact by Assumption~\ref{asmp:bdd_density}.
\end{proof}

\begin{lemma} \label{lem:lpr_covariance}
    \emph{(Local polynomial regression covariance bound)}  Suppose Assumptions~\ref{asmp:dgp},~\ref{asmp:bdd_density}, and~\ref{asmp:holder} hold, $\widehat \mu(x)$ is a local polynomial regression estimator (Estimator~\ref{est:lpr}) for $\mu(x)$ constructed on $D_\mu$, and $nh^d \to \infty$ as $n \to \infty$. Then,
    $$
    \bbE \big[ \left| \cov \{ \widehat \mu(X_i), \widehat \mu(X_j) \mid X_i, X_j \} \right| \big] \lesssim h^d \left\{ O_\bbP \left( \frac{1}{nh^d} + h^{2 (\beta \wedge \lceil d/2 \rceil)} \right) + \exp(-nh^d) \right\}
    $$
\end{lemma}

\begin{proof}
    We have
    \begin{align*}
        \bbE \big[ \left| \cov \{ \widehat \mu(X_i), \widehat \mu(X_j) \mid X_i, X_j \} \right| \big] &= \bbE \big[ \left| \cov \{ \widehat \mu(X_i), \widehat \mu(X_j) \mid X_i, X_j \} \right| \one \left(\lVert X_i - X_j \rVert \leq 2h \right) \big] \\
        &\leq \sup_{x_i, x_j} \left| \cov \{ \widehat \mu(x_i), \widehat \mu(x_j) \} \right| \bbP \left(\lVert X_i - X_j \rVert \leq 2h \right) \\
        &\leq \sup_{x \in \mathcal{X}} \bbV \{ \widehat \mu(x) \}  \bbP ( \lVert X_i - X_j \rVert \leq 2h ) \\
        &\lesssim  \left\{ O_\bbP \left( \frac{1}{nh^d} + h^{2 (\beta \wedge \lceil d/2 \rceil)} \right) + \exp(-nh^d) \right\} h^d
    \end{align*}
    where the first line follows because $\cov \{ \widehat \mu(X_i), \widehat \mu(X_j) \mid X_i, X_j \} = 0$ when $\lVert X_i - X_j \rVert > 2h$, the second by H\"{o}lder's inequality, and the last line by Lemmas~\ref{lem:lpr_bounds} and~\ref{lem:sphere}.  
\end{proof}

\color{black}
\section{Section~\ref{sec:unknown} proofs: Lemma~\ref{lem:covariance} and Theorem~\ref{thm:semiparametric}} \label{app:unknown}

In this section, we use the results from Appendices \ref{app:modelfree} and \ref{app:nf_ests} to establish Lemma \ref{lem:covariance} and Theorem \ref{thm:semiparametric} from Section~\ref{sec:unknown}.

\lemcovariance*

\begin{proof}
    This follows by Lemmas~\ref{lem:knn_covariance} and~\ref{lem:lpr_covariance}, and by the conditions on the tuning parameters.
\end{proof}

\thmsemiparametric*

\begin{proof}
    By Lemma~\ref{lem:expansion},  
    $$
    \widehat \psi_n - \psi_{ecc} = (\bbP_n - \bbP) \varphi + R_{1,n} + R_{2,n}
    $$
    where 
    $$
    R_{1,n} \leq \lVert b_\pi \rVert_\bbP \lVert b_\mu \rVert_\bbP
    $$
    and 
    $$
    R_{2,n} = O_{\bbP} \left( \sqrt{ \frac{\bbE \lVert \widehat \varphi - \varphi \rVert^2_\bbP + \rho(\Sigma_n)}{n} } \right).
    $$
    The first term, $(\bbP_n - \bbP) \varphi$, satisfies the CLT in the statement of the result, and also satisfies $(\bbP_n - \bbP) \varphi = O_\bbP(n^{-1/2})$.  Therefore, we focus on the two remainder terms in the rest of this proof.
    
    \medskip
    
    By the conditions on the rate at which the number of neighbors and the bandwidth scale, and by Lemma~\ref{lem:covariance}, 
    $$
    \bbE \Big[ \left| \cov \{ \widehat \eta(X_i), \widehat \eta(X_j) \mid X_i, X_j \} \right| \Big] \lesssim \frac{1}{n} \text{ for } \eta \in \{ \pi, \mu \}.
    $$
    Therefore, by Proposition~\ref{prop:spectral},
    $$
    R_{2,n} = O_\bbP \left( \sqrt{ \frac{\bbE \lVert \widehat \varphi - \varphi \rVert^2_\bbP + \lVert b_\pi^2 \rVert_\infty + \lVert s_\pi^2 \rVert_\infty + \lVert b_\mu^2 \rVert_\infty + \lVert s_\mu^2 \rVert_\infty }{n} } \right).
    $$
    Because the EIF for the ECC is Lipschitz in the nuisance functions,
    $$
    \bbE \lVert \widehat \varphi - \varphi \rVert_\bbP^2 \lesssim \bbE \lVert \widehat \pi - \pi \rVert_\bbP^2 + \bbE \lVert \widehat \mu - \mu \rVert_\bbP^2 \leq \lVert b_\pi^2 \rVert_\infty + \lVert s_\pi^2 \rVert_\infty + \lVert b_\mu^2 \rVert_\infty + \lVert s_\mu^2 \rVert_\infty, 
    $$
    and, thus, 
    $$
    R_{2,n} = O_\bbP \left( \sqrt{ \frac{\lVert b_\pi^2 \rVert_\infty + \lVert s_\pi^2 \rVert_\infty + \lVert b_\mu^2 \rVert_\infty + \lVert s_\mu^2 \rVert_\infty }{n} } \right).
    $$
    
    \noindent \large \textbf{Nearest Neighbors:} \\
    \normalsize Next, we consider k-Nearest-Neighbors. By Lemma~\ref{lem:knn_bounds}, when $k_\mu, k_\pi \asymp \log n$,
    \begin{equation} \label{eq:knn_rem_bias}
        R_{1,n} \leq \lVert b_\pi \rVert_\bbP \lVert b_\mu \rVert_\bbP \lesssim \left( \frac{n}{\log n} \right)^{-\frac{(\alpha \wedge 1) + (\beta \wedge 1)}{d}}
    \end{equation}
    while 
    $$
    R_{2,n} = O_\bbP \left( \sqrt{ \frac{\left( n / \log n \right)^{-\frac{(\alpha \wedge 1)}{d}} + \left( n / \log n \right)^{-\frac{(\beta \wedge 1)}{d}} + 1 / \log n}{n}} \right) = o_\bbP(n^{-1/2}).
    $$
    The variance term, $R_{2,n}$, is always asymptotically negligible, while the bias term, $R_{1,n}$, controls when the estimator is $\sqrt{n}$-consistent and the convergence rate in the non-$\sqrt{n}$ regime. The convergence rate in the non-root-n regime follows immediately from \eqref{eq:knn_rem_bias}. For the threshold at which the estimator is $\sqrt{n}$-consistent, notice that 
    $$
    R_{1,n} \leq \left( \frac{n}{\log n} \right)^{-\frac{(\alpha \wedge 1) + (\beta \wedge 1)}{d}} =  \left( \frac{n}{\log n} \right)^{-\frac{\alpha + \beta}{d}} = o_{\bbP} (n^{-1/2})
    $$
    if and only if $\frac{\alpha + \beta}{2} > d/4$ and $\alpha, \beta \leq 1$.
    
    \medskip
    
    \noindent \large \textbf{Local polynomial regression:} \\
    \normalsize
    \noindent For local polynomial regression, by Lemma~\ref{lem:lpr_bounds}, when $h_\mu, h_\pi \asymp \left( \frac{n}{\log n} \right)^{-1/d}$ then
    $$
    R_{1,n} \leq \lVert b_\pi \rVert_\bbP \lVert b_\mu \rVert_\bbP = O_\bbP \left( \frac{n}{\log n} \right)^{-\frac{(\alpha \wedge \lceil d/2 \rceil) + (\beta \wedge \lceil d/2 \rceil)}{d}}
    $$
    while 
    $$
    R_{2,n} = O_\bbP \left( \sqrt{ \frac{\left( n / \log n \right)^{-\frac{(\alpha \wedge \lceil d/2 \rceil)}{d}} + \left( n / \log n \right)^{-\frac{(\beta \wedge \lceil d/2 \rceil)}{d}} + 1 / \log n}{n}} \right) = o_\bbP(n^{-1/2}).
    $$
    Again, the variance term, $R_{2,n}$, is always asymptotically negligible, while the bias term, $R_{1,n}$, controls when the estimator is $\sqrt{n}$-consistent and the convergence rate in the non-$\sqrt{n}$ regime. When $\frac{\alpha + \beta}{2} > \frac{d}{4}$ there are two cases to consider: (1) when $\alpha > d/2$ or $\beta > d/2$, and (2) when $\alpha, \beta < d/2$.  In the first case, then
    $$
    R_{1,n} = O_\bbP \left( \frac{n}{\log n} \right)^{-\frac{(\alpha \wedge \lceil d/2 \rceil) + (\beta \wedge \lceil d/2 \rceil)}{d}} = O_\bbP  \left( \frac{n}{\log n} \right)^{-\frac{\lceil d/2 \rceil}{d}} = o_\bbP (n^{-1/2}).
    $$
    In the second case, 
    $$
    R_{1,n} = O_\bbP \left( \frac{n}{\log n} \right)^{-\frac{\alpha + \beta}{d}} = o_\bbP(n^{-1/2}), 
    $$
    which follows because $\alpha + \beta > d /2$. 
    
    \medskip
    
    When $\frac{\alpha + \beta}{2} \leq d/4$, it follows that $\alpha + \beta \leq d/2 \implies \alpha, \beta \leq \lceil d/2\rceil$.  Therefore, the convergence rate of the DCDR estimator satisfies
    $$
    \bbE | \widehat \psi_n - \psi | = O_\bbP \left( \frac{n}{\log n} \right)^{-\frac{\alpha + \beta}{d}} + o_\bbP(n^{-1/2}).
    $$
\end{proof}

\section{Centered random forests} \label{app:random-forest}

In this section, we analyze the centered random forest proposed by \citet{biau2012analysis}, using the same setup as in the main paper. We first define the estimator and then establish convergence rates for its bias, variance, and expected absolute covariance. These results closely parallel those obtained for the k-NN estimator in the main paper and therefore imply the same conclusions as stated in Theorem~\ref{thm:semiparametric}.

\medskip

Centered random forests differ from Breiman’s original random forest proposal \citep{breiman2001random} and from those typically used in practice. The key distinction is that the tree partitions in a centered random forest are constructed independently of the data, significantly simplifying theoretical analysis. Extending these results to random forests commonly implemented in practice is substantially more challenging and lies beyond the scope of this work. However, \citet[Section 3]{biau2012analysis} discusses connections between centered random forests and practical variants, emphasizing how the results presented here remain relevant to more commonly used implementations, suggesting these results are not merely of theoretical interest.

\medskip

Centered random forests use the whole dataset for each tree, select a feature at random for each node, and then split at the midpoint of that feature.  With centered forests, the $b^{th}$ tree estimator is
\[
\widehat \mu_b(x) = \frac{\sum_{Z_j \in D_\mu} \one\{ X_j \in A_n(x; \theta_b) \} Y_j}{\sum_{Z_j \in D_\mu} \one\{ X_j \in A_n(x; \theta_b) \}}
\]
where $\theta_b$ is the partition induced by tree $b$ and $A_n(x; \theta_b)$ is the leaf of estimation point $x$ in tree $\theta_b$. The centered forest estimator is then given by
\[
\widehat \mu(x) = \lim_{B \to \infty} \frac{1}{B} \sum_{b=1}^{B} \widehat \mu_b(x) = \sum_{Z_k \in D_\mu} \bbE \{ w_k(x; \Theta) \mid D_\mu \} Y_k
\]
where $N_n(x; \Theta) = \sum_{k=1}^{n} \one \{ X_k \in A_n(x; \Theta) \}$ is the number of training samples in the leaf containing $x$ and 
\[
w_k(x; \Theta) = \frac{\one\{ X_k \in A_n(x; \Theta)\}}{N_n(x; \Theta)} \one\{ N_n(x; \Theta) > 0 \}.
\]
Because the error due to using a finite number of trees can be made arbitrarily small by increasing $B$, we focus on the estimator $\widehat \mu(x) = \sum_{Z_k \in D_\mu} \bbE \{ w_k(x; \Theta) \mid D_\mu \} Y_k$. We formally define the estimator next.
\begin{estimator}[Centered random forest]
    The estimator $\widehat \mu(x)$ is constructed as
    \[
    \widehat \mu(x) = \sum_{Z_k \in D_\mu} \bbE \{ w_k(x; \Theta) \mid D_\mu \} Y_k
    \]
    where
    \[
    w_k(x; \Theta) = \frac{\one \{ X_k \in A_n(x; \Theta) \}}{N_n(x; \Theta)} \one \{ N_n(x; \Theta) > 0 \}
    \]
    and
    \begin{itemize}[itemsep=0in]
        \item $D_\mu$ is the training data,
        \item $w_k$ is the weight to data point $X_k$ ,
        \item $\Theta$ is the random partition generated by the splitting procedure in \citet{biau2012analysis},
        \item $A_n(x; \Theta)$ is the leaf of estimation point $x$ in $\Theta$, and 
        \item $N_n(x; \Theta) = \sum_{k=1}^{n} \one \{ X_k \in A_n(x; \theta) \}$ denotes the number of training points in $A_n(x; \Theta)$.
    \end{itemize}
    We consider the simplest version of the splitting procedure in \citet{biau2012analysis}. A fixed parameter $k_n$ controls the number of splits; specifically repeat the following $\lceil \log_2 k_n \rceil$ times:
    \begin{itemize}[itemsep=0in]
        \item At each node, randomly choose a feature on which to split, with probability $d^{-1}$ for each feature.
        \item On the chosen feature, split at the midpoint.
    \end{itemize}
\end{estimator}

This estimator is a simplification of the estimator presented in \citet{biau2012analysis}. In particular, \citeauthor{biau2012analysis} examines sparsity, showing that if the splitting procedure focuses on the ``strong'' variables asymptotically, then the convergence rates of the centered random forest can adapt to strong sparsity and converge faster. We ignore sparsity because it is not our focus in this paper and because the splitting procedure relies on knowledge of which covariates are in the sparsity set, which could be unrealistic in practice. 

\medskip

Compared to the body of this paper, we add another simplifying assumption on the data generating process. Namely, we assume uniform covariates with support the unit hyper-cube. Nonetheless, it seems feasible that this assumption could be relaxed to Assumption~\ref{asmp:bdd_density} from the main paper at the expense of additional complexity in the analysis (see, e.g., Remark 10 in Section 5 of \citet{biau2012analysis} for intuition in this direction). 
\begin{assumption} \label{asmp:uniform}
    The covariate distribution is uniform on the unit hyper-cube. 
\end{assumption}

\subsection{Convergence rates for centered random forests}

In this section, we state and prove several convergence guarantees for centered random forests. The first two results bound the supremum bias and variance of the estimator. They are straightforward corollaries of results in \citet{biau2012analysis}. The primary new result is a bound on the expected absolute covariance of the estimator. The third and fourth results are helper lemmas towards that goal, and the final result provides the bound on the expected absolute covariance. 

\begin{corollary} \label{cor:rf-bias}
    Suppose Assumptions~\ref{asmp:dgp} and \ref{asmp:uniform} hold and $\mu(x)$ is Lipschitz. Then, the supremum of the bias of the centered random forest estimator satisfies
    \[
    \sup_{x \in \mathcal{X}} \bbE \{ \widehat \mu(X) - \mu(X) \mid X = x\}^2 \lesssim k_n^{-1/d} + \exp\left( - \frac{n}{2k_n} \right).
    \]
\end{corollary}

\begin{proof}
    The analysis of the bias of the estimator in \citet[Proposition 4]{biau2012analysis} can be conducted pointwise on arbitrary $X=x$. The supremum bound holds because $\mathcal{X}$ is closed and bounded.
\end{proof}

\begin{corollary} \label{cor:rf-variance}
    Suppose Assumptions~\ref{asmp:dgp} and \ref{asmp:uniform} hold. Then, the supremum of the variance of the estimator satisfies
    \[
    \sup_{x \in \mathcal{X}} \bbV \{ \widehat \mu(X) \mid X = x\} \lesssim \frac{k_n}{n}.
    \]
\end{corollary}

\begin{proof}
    The analysis of the variance of the estimator in \citet[Proposition 2]{biau2012analysis} can be conducted pointwise at arbitrary $X=x$. The supremum bound holds because $\mathcal{X}$ is closed and bounded.
\end{proof}

The next result places a bound on the product that two leaves overlap, which we use to bound the expected absolute covariance.

\begin{lemma} \label{lem:leaf-prod-bound}
    Suppose Assumptions~\ref{asmp:dgp} and \ref{asmp:uniform} hold. Let $X$ and $X^\prime$ denote iid covariate observations and $\Theta$ and $\Theta^\prime$ denote iid partitions according to the procedure outlined above, where $k_n$ increases with sample size. Then,
    \[
    \bbP \left\{ A_n(X; \Theta) \cap A_n(X^\prime; \Theta^\prime) \neq \emptyset \right\} \lesssim \frac{(\log k_n)^{d-1}}{k_n}.
    \]
\end{lemma}

\begin{proof}
    Let $L_a$ for $a \in \{1, \dots, d\}$ denote the side lengths of $A_n(X; \Theta)$ and $L_a^\prime$ denote the side lengths of $A_n(X^\prime; \Theta^\prime)$. Moreover, let $X_a$ and $X_a^\prime$ denote the $a^{th}$ dimensions of $X$ and $X^\prime$, respectively. 
    
    \medskip
    
    We begin by upper bounding the probability in question using two implications.  First,
    \[
    A_n(X; \Theta) \cap A_n(X^\prime; \Theta^\prime) \neq \emptyset \implies \left| X_a - X_a^\prime \right| \leq L_a + L_a^\prime \text{ for all } a \in [d]. 
    \]
    Second,
    \[
    \left| X_a - X_a^\prime \right| \leq L_a + L_a^\prime \text{ for all } a \in [d] \implies \prod_{a=1}^{d} \left| X_a - X_a^\prime \right| \leq \prod_{a=1}^{d} ( L_a + L_a^\prime ).
    \]
    Hence,
    \[
    \bbP \left\{ A_n(X; \Theta) \cap A_n(X^\prime; \Theta^\prime) \neq \emptyset \right\} \leq \bbP \left\{  \prod_{a=1}^{d} \left| X_a - X_a^\prime \right| \leq \prod_{a=1}^{d} ( L_a + L_a^\prime ) \right\}
    \]
    The probability on the right-hand side is amenable to a simple analysis:
    \begin{align*}
        \bbP \left\{  \prod_{a=1}^{d} \left| X_a - X_a^\prime \right| \leq \prod_{a=1}^{d} ( L_a + L_a^\prime ) \right\} &= \bbE \left[ \bbP \left\{  \prod_{a=1}^{d} \left| X_a - X_a^\prime \right| \leq \prod_{a=1}^{d} ( L_a + L_a^\prime ) \mid X, X^\prime \right\} \right] \\
        &= \bbE \left[ \bbP \left\{  \prod_{a=1}^{d} \left| X_a - X_a^\prime \right| \leq \sum_{S \in 2^d} \prod_{a \in S} L_a \prod_{b \notin S} L_b^\prime \mid X, X^\prime \right\} \right] \\
        &= \bbE \left[ \bbP \left\{  \prod_{a=1}^{d} \left| X_a - X_a^\prime \right| \leq \sum_{S \in 2^d} \prod_{a \in S} L_a \prod_{b \notin S} L_b \mid X, X^\prime \right\} \right] \\
        &= \bbE \left[ \bbP \left\{  \prod_{a=1}^{d} \left| X_a - X_a^\prime \right| \leq 2^d \prod_{a =1}^d L_a \mid X, X^\prime \right\} \right] \\
        &= \bbE \left[ \bbP \left\{  \prod_{a=1}^{d} \left| X_a - X_a^\prime \right| \leq 2^{d -\lceil \log_2 k_n \rceil} \mid X, X^\prime \right\} \right] \\
        &= \bbP  \left\{  \prod_{a=1}^{d} \left| X_a - X_a^\prime \right| \leq 2^{d -\lceil \log_2 k_n \rceil} \right\}
    \end{align*}
    where the first line follows by iterated expectations on $X, X^\prime$ and the second by multiplying out the product $\prod_{a=1}^{d} (L_a + L_a^\prime)$. The third follows because, crucially, $\left\{ L_a \right\}_{a=1}^d$ and $\left\{ L_a^\prime \right\}_{a=1}^d$ are independent and identically distributed conditional on $X$ and $X^\prime$ and therefore we can replace $\prod_{b \notin S} L_b^\prime$ by $\prod_{b \notin S} L_b$. The penultimate line follows because the size of $A_n(X; \Theta)$ is $2^{-\lceil \log_2 k_n \rceil}$ by construction (see fact 2 in \citet{biau2012analysis}). 

    \medskip

    To conclude, we can bound the probability at the bottom of the previous display, which is the probability that the volume of the axis-aligned hyper-rectangle defined by $X$ and $X^\prime$ is less than $2^{d - \lceil \log_2 k_n \rceil}$. Suppose $k_n$ increases with sample size so that $2^{d - \lceil \log_2 k_n \rceil} \in (0,1)$ for large enough $n$. Then, Lemma~\ref{lem:hyper-rectangle}, next, yields
    \[
    \bbP  \left\{  \prod_{a=1}^{d} \left| X_a - X_a^\prime \right| \leq 2^{d -\lceil \log_2 k_n \rceil} \right\} \lesssim \left( 2^{d -\lceil \log_2 k_n \rceil} \right) \log^{d-1} \left( 2^{\lceil \log_2 k_n \rceil-d} \right),
    \]
    from which the result follows.
\end{proof}

The next result bounds the size of the axis-aligned hyper-rectangle, which was used in the final step of the previous result.
\begin{lemma} \label{lem:hyper-rectangle}
    Under the setup of Lemma~\ref{lem:leaf-prod-bound}, let
    \[
    \quad V_d = \prod_{a=1}^d |X_a-X'_a|.
    \]
    Then, for all \(t\in(0,1)\),
    \[
    \bbP(V_d \le t) \lesssim t \log^{d-1} \left( \tfrac1t \right).
    \]
\end{lemma}

\begin{proof}
    We proceed by induction on \(d\).
    
    \medskip\noindent\textbf{Base case \(\boldsymbol{d=1}\).} 
    When $d=1$, \(V_1\) follows the triangular distribution. Hence,    \[
    \bbP(V_1 \leq t) = \int_0^t 2(1-u) du \lesssim 2t \lesssim t.
    \]

    \medskip\noindent\textbf{Inductive step.} 
    Assume the statement holds for dimension \(d-1\), i.e.,
    \[
    \bbP\left( V_{d-1} \leq t \right) \leq t \log^{d-2} \left( \tfrac1t \right).
    \]
    Then, we have
    \begin{align*}
        \bbP \left( V_d \leq t \right) &= \bbP \left( V_{d-1} \leq \frac{t}{\left| X_d - X_d^\prime \right|} \right) \\
        &= \bbE \left\{ \bbP \left( V_{d-1} \leq \frac{t}{\left| X_d - X_d^\prime \right|} \right) \mid X_d, X_d^\prime \right\} \\
        &= \int_0^1 \bbP \left( V_{d-1} \leq \tfrac{t}{u} \right) 2(1-u) du \\
        &= \int_0^t 2(1-u) du + \int_t^1 \bbP \left( V_{d-1} \leq \tfrac{t}{u} \right) 2(1-u) du,
    \end{align*}
    where the first line follows by dividing through by $\left| X_d - X_d^\prime \right|$, ignoring the case where $\left| X_d - X_d^\prime \right| = 0$ which occurs almost never, the second line follows by iterated expectations on $X_d, X_d^\prime$, the third line because $X_d - X_d^\prime$ follows the triangular distribution, and the fourth line because the inner probability is at most $1$ when $u < t$. 

    \medskip

    The first summand in the final display above satisfies
    \[
    \int_0^t 2(1-u) du \lesssim t.
    \]
    The second summand is the key. By the assumption on $V_{d-1}$, we have
    \[
    \int_t^1 \bbP \left( V_{d-1} \leq \tfrac{t}{u} \right) 2(1-u) du \lesssim \int_t^1 \frac{t}{u} \log^{d-2} \left( \tfrac{u}{t} \right) 2(1-u) du \lesssim \int_t^1 \frac{t}{u} \log^{d-2} \left( \tfrac{u}{t}\right) du. 
    \]
    By a change of variables with $w = \log(u/k)$, we have 
    \[
    \int_t^1 \frac{t}{u} \log^{d-2} \left( \tfrac{u}{t}\right) du = \int_0^{\log\left (1/t \right)} \exp(-w) w^{d-2} t \exp(w) dw = t \int_0^{\log (1/t)} w^{d-2}dw \lesssim t \log^{d-1}\left(\tfrac1t\right).
    \]
    Hence, $\bbP \left( V_d \leq t \right) \lesssim t \log^{d-1} \left( \tfrac1t \right)$ and the result is proved.
\end{proof}

\noindent The final result bounds the expected absolute conditional covariance term, demonstrating that it scales like $\frac{(\log k_n)^{d-1}}{n}$. 

\begin{lemma} 
    Suppose Assumptions~\ref{asmp:dgp} and \ref{asmp:uniform} hold and $k_n \to \infty$ as $n \to \infty$. Then, 
    \[
    \bbE \left[ \big| \cov \{ \widehat \mu(X_i), \widehat \mu(X_j) \mid X_i, X_j \} \big| \right] \lesssim \frac{(\log k_n)^{d-1}}{n}
    \]
\end{lemma}

\begin{proof}
    We have
    \begin{align*}
        &\bbE \left[ \big| \cov \{ \widehat \mu(X_i), \widehat \mu(X_j) \mid X_i, X_j \} \big| \right] \\
        &= \bbE \left( \left| \cov \left[ \sum_{Z_k \in D_\mu} \bbE \{ w_k(X_i, \Theta) \mid X_i, D_\mu \} Y_k, \sum_{Z_l \in D_\mu} \bbE \{ w_l(X_j; \Theta) \mid X_j, D_\mu \} Y_l \,\Big|\, X_i, X_j \right] \right| \right) \\
        &= \bbE \left( \left| \cov \left[ \sum_{Z_k \in D_\mu} \bbE \{ w_k(X_i; \Theta) \mid X_i, D_\mu \} Y_k, \sum_{Z_l \in D_\mu} \bbE \{ w_l(X_j; \Theta^\prime) \mid X_j, D_\mu \} Y_l \,\Big|\, X_i, X_j \right] \right| \right) \\
        &= \bbE \left( \left| \cov \left[ \sum_{Z_k \in D_\mu} w_k(X_i; \Theta) Y_k, \sum_{Z_l \in D_\mu} w_l(X_j; \Theta^\prime) Y_l \,\Big|\, X_i, X_j \right] \right| \right)
    \end{align*}
    where the first line follows by definition, the second by replacing $\Theta$ by $\Theta^\prime$, an iid partition, in the right-hand side of the covariance, and the third line by the law of total covariance and because $\Theta \ind \Theta^\prime$.

    \bigskip

    Next, notice that the final term on the right-hand side is zero whenever $A_n(X_i; \Theta) \cap A_n(X_j; \Theta^\prime) = \emptyset$; i.e., if the leaves containing $X_i$ and $X_j$ do not intersect, then the estimators cannot share training data. Therefore, \footnotesize
    \begin{align*}
        &\bbE \left[ \big| \cov \{ \widehat \mu(X_i), \widehat \mu(X_j) \mid X_i, X_j \} \big| \right] \\
        &= \bbE \Bigg( \left| \cov \left[ \sum_{Z_k \in D_\mu} w_k(X_i; \Theta) Y_k, \sum_{Z_l \in D_\mu} w_l(X_j; \Theta^\prime) Y_l \,\Big|\, X_i, X_j, A_n(X_i; \Theta) \cap A_n(X_j; \Theta^\prime) \neq \emptyset \right] \right| \\
        &\hspace{0.5in} \cdot \bbP \left\{ A_n(X_i; \Theta) \cap A_n(X_j; \Theta^\prime) \neq \emptyset \mid X_i, X_j \right\} \Bigg). 
    \end{align*} \normalsize
    Then, we can bound the inner covariance:
    \begin{align*}
        &\left| \cov \left[ \sum_{Z_k \in D_\mu} w_k(X_i; \Theta) Y_k, \sum_{Z_l \in D_\mu} w_l(X_j; \Theta^\prime) Y_l \,\Big|\, X_i, X_j, A_n(X_i; \Theta) \cap A_n(X_j; \Theta^\prime) \neq \emptyset \right] \right| \\
        &\leq \left| \cov \left[ \sum_{Z_k \in D_\mu} w_k(X; \Theta) Y_k, \sum_{Z_l \in D_\mu} w_l(X; \Theta^\prime) Y_l \,\Big|\, X_i = X_j = X \right] \right| \\
        &= \left| \cov \left[ \sum_{Z_k \in D_\mu} \bbE \{  w_k(X; \Theta) \mid X, D_\mu \} Y_k, \sum_{Z_l \in D_\mu} \bbE \{  w_l(X; \Theta) \mid X, D_\mu \} Y_l \,\Big|\, X \right] \right| \\
        &= \bbV \{ \widehat \mu(X) \mid X \}.
    \end{align*}
    Note that the first inequality follows follows by setting $X_i = X_j$, which will increase the absolute value of the covariance and noting that $\one \{  A_n(X; \Theta) \cap A_n(X; \Theta^\prime) \neq 0 \} = 1$ because the leaves contain the same point, the second equality follows by the law of total covariance on $\Theta, \Theta^\prime$, and the third line follows by the definition of variance. 

    \medskip
    
    The variance of the estimator can be bounded by its supremum, and therefore H\"{o}lder's inequality and iterated expectations yield
    \[
    \bbE \left[ \big| \cov \{ \widehat \mu(X_i), \widehat \mu(X_j) \mid X_i, X_j \} \big| \right] \leq \sup_{x \in \mathcal{X}} \bbV \{ \widehat \mu(X) \mid X=x \}  \bbP \left\{ A_n(X_i; \Theta) \cap A_n(X_j; \Theta^\prime) \neq \emptyset \right\}.
    \]
    Lemma~\ref{lem:leaf-prod-bound} and Corollary~\ref{cor:rf-variance} imply
    \[
    \bbE \left[ \big| \cov \{ \widehat \mu(X_i), \widehat \mu(X_j) \mid X_i, X_j \} \big| \right] \lesssim \frac{k_n}{n} \cdot \frac{(\log k_n)^{d-1}}{k_n} \lesssim \frac{(\log k_n)^{d-1}}{n}.
    \]
\end{proof}

\color{black}
\section{Covariate-density-adapted kernel regression} \label{app:cda_kernel}

In this section, we establish six results for covariate-density-adapted kernel regression (Estimator~\ref{est:cdalpr}). The first result, Lemma~\ref{lem:cdalpr_upper_bounds}, establishes upper bounds on the variance and covariance.  The second result, Lemma~\ref{lem:cdalpr_lower_bounds}, establishes a lower bound on the unconditional variance. The third result, Lemma~\ref{lem:cdalpr_cond_lower_bounds}, establishes an almost sure limit for the conditional variance while the fourth result, Lemma \ref{lem:cdalpr_third_moment}, establishes an upper bound on the conditional third moment of the estimator. These two results are used in establishing Theorem~\ref{thm:inference} in Appendix~\ref{app:known}. The fifth result, Lemma~\ref{lem:cdakernel_holder}, demonstrates that $\bbE \{ \widehat \mu(x) \}$ is H\"{o}lder smooth when $\widehat \mu$ is the smooth covariate-density-adapted kernel regression (Estimator \ref{est:cdalpr_smooth}), while the sixth result, Lemma~\ref{lem:cdakernel_bounded}, demonstrates this estimator is bounded if the outcome is bounded.

\begin{lemma} \label{lem:cdalpr_upper_bounds} \emph{(Covariate-density-adapted kernel regression variance and covariance upper bounds)}
    Suppose Assumptions~\ref{asmp:dgp}, \ref{asmp:bdd_density}, \ref{asmp:holder}, and \ref{asmp:density_smooth} hold, and $\widehat \mu(x)$ is either a higher-order or smooth covariate-density-adapted kernel regression estimator (Estimator~\ref{est:cdalpr_higher} or~\ref{est:cdalpr_smooth}) for $\mu(x)$ constructed on $D_\mu$.  Then,
    \begin{align}
        \sup_{x \in \mathcal{X}} \bbV \{ \widehat \mu(x) \} &\lesssim \frac{1}{nh^d}\text{, and} \label{eq:cdalpr_variance} \\
        \bbE \big[ \left| \cov \{ \widehat \mu(X_i), \widehat \mu(X_j) \mid X_i, X_j \} \right| \big] &\lesssim \frac1n \label{eq:cdalpr_covariance}
    \end{align}
\end{lemma}

\begin{proof}
    For the variance upper bound, we have
    \begin{align*}
        \bbV \{ \widehat \mu(x) \} &= \bbV \left\{ \sum_{i=1}^{n} \frac{K \left( \frac{X_i - x}{h} \right) \mu(X_i)}{n h^d f(X_i)} \right\} + \bbE \left[ \bbV \left\{ \sum_{i=1}^{n} \frac{K \left( \frac{ X_i - x }{h} \right) Y_i}{n h^d f(X_i)} \mid X_\mu^n \right\} \right]  \\
        &\lesssim \bbE \left\{ \frac{K \left( \frac{ X_i - x }{h} \right)^2 \mu(X_i)^2}{n h^{2d} f(X_i)^2} \right\} + \bbE \left\{ \frac{K \left( \frac{ X_i - x}{h} \right)^2}{n h^{2d} f(X_i)^2} \right\}  \\
        &\lesssim \frac{1}{nh^d},
    \end{align*}
    where the first line follows by the law of total variance, the second by iid data and Assumptions~\ref{asmp:dgp} and~\ref{asmp:bdd_density}, and the third line follows by the assumption on the kernel that $\int K(x)^2 dx \lesssim 1$ and Assumptions~\ref{asmp:dgp} and~\ref{asmp:bdd_density}.  The uniform bound follows because $\mathcal{X}$ is compact.
    
    \medskip
    
    \noindent For the covariance, since the estimator is localized, by the same argument as Lemmas~\ref{lem:knn_covariance} and~\ref{lem:lpr_covariance}
    $$
    \bbE \big[ \left| \cov \{ \widehat \mu(X_i), \widehat \mu(X_j) \mid X_i, X_j \} \right| \big] \leq \sup_{x \in \mathcal{X}} \bbV \{ \widehat \mu(x) \} \bbP ( \lVert X_i - X_j \rVert \leq 2h ) \lesssim \frac1n.
    $$
\end{proof}

\begin{lemma}  \label{lem:cdalpr_lower_bounds} \emph{(Covariate-density-adapted kernel regression variance lower bounds)}
    Suppose Assumptions~\ref{asmp:dgp}, \ref{asmp:bdd_density}, \ref{asmp:density_smooth}, and \ref{asmp:boundedAY} hold and $\widehat \mu(x)$ is a either a higher-order or smooth covariate-density-adapted kernel regression estimator (Estimator~\ref{est:cdalpr_higher} or~\ref{est:cdalpr_smooth}) for $\mu(x)$ constructed on $D_\mu$.  Then, 
    \begin{equation}
        \inf_{x \in \mathcal{X}} \bbV \{ \widehat \mu(x) \} \gtrsim \frac{1}{nh^d}. \label{eq:cdalpr_var_lower}
    \end{equation}
\end{lemma}

\begin{proof}
    We have, 
    \begin{align*}
        \bbV \{ \widehat \mu(x) \} &= \bbV \big[ \bbE \{ \widehat \mu(x) \mid X_\mu^n \} \big] + \bbE \big[ \bbV \{ \widehat \mu(x) \mid X_\mu^n \} \big] \\
        &\geq 0 + \bbE \left[ \bbV \left\{ \frac1n \sum_{i=1}^{n} \frac{ K \left( \frac{ X_i - x }{h} \right) Y_i}{f(X_i) h^d}   \ \Big| X_\mu^n \right\} \right] \\
        &= \frac{1}{n h^{2d}} \bbE \left\{ \frac{K \left( \frac{X - x}{h} \right)^2}{f(X)^2}  \bbV ( Y \mid  X) \right\} \\
        &= \frac{1}{n h^{2d}} \int_{t \in \bbR} K \left( \frac{t - x}{h} \right)^2 \frac{\bbV(Y \mid X = t)}{f(t)} dt \\
        &\gtrsim \frac{1}{n h^{2d}} \int_{t \in \bbR} K \left( \frac{t - x}{h} \right)^2  dt \\
        &= \frac{1}{n h^d}  \int_{u \in \bbR} K \left( u \right)^2 du &u = (t-x)/h \\
        &\gtrsim  \frac{1}{nh^d},
    \end{align*}
    where the second inequality follows by Assumption~\ref{asmp:dgp} and~\ref{asmp:bdd_density} (specifically, because we assume $0 < f(x), \bbV(Y \mid x) < C$ for all $x \in \mathcal{X}$), and the final line by the definition of the kernel in Estimator~\ref{est:cdalpr_higher} and ~\ref{est:cdalpr_smooth} (specifically, because $\int K(u)^2 du \asymp 1$).  These bounds hold for arbitrary $x \in \mathcal{X}$, and thus hold for the infimum over all $x \in \mathcal{X}$ since $\mathcal{X}$ is compact by Assumption~\ref{asmp:bdd_density}.
\end{proof}

\begin{lemma}  \label{lem:cdalpr_cond_lower_bounds} \emph{(Covariate-density-adapted kernel regression conditional variance lower bounds)}
    Suppose Assumptions~\ref{asmp:dgp}, \ref{asmp:bdd_density}, \ref{asmp:density_smooth}, \ref{asmp:boundedAY}, and \ref{asmp:continuous_cond_var} hold and $\widehat \mu(x)$ is a either a higher-order or smooth covariate-density-adapted kernel regression estimator (Estimator~\ref{est:cdalpr_higher} or~\ref{est:cdalpr_smooth}) for $\mu(x)$ constructed on $D_\mu$. Then, when $nh^d \asymp n^{-\alpha}$ for $\alpha > 0$ as $n \to \infty$, 
    \begin{equation} \label{eq:cond_muhat_as}
        nh^d \bbV \{ \widehat \mu(X) \mid D_\mu \} \inas \bbE \left\{ \frac{Y^2}{f(X)} \right\} \bbE \left\{ \frac{K(X)^2}{f(X)} \right\}.
    \end{equation}
\end{lemma}

\begin{proof}
    We will consider the diagonal variance terms and off-diagonal covariance terms separately
    \begin{align*}
        \bbV \{ \widehat \mu(X) \mid D_\mu \} &= \frac{1}{n^2} \sum_{i=1}^{n} \bbV \left\{ \frac{K \left( \frac{X_i - X}{h} \right)}{h^d f(X_i)} Y_i \mid X_i, Y_i \right\} \\
        &+ \frac{1}{n^2} \sum_{i=1}^{n} \sum_{j \neq i} \cov \left\{ \frac{K \left( \frac{X_i - X}{h} \right)}{h^d f(X_i)} Y_i, \frac{K \left( \frac{X_j - X}{h} \right)}{h^d f(X_j)} Y_j \mid X_i, Y_i, X_j, Y_j \right\}.
    \end{align*}
    For the diagonal terms,
    \begin{align*}
        nh^d \frac{1}{n^2} \sum_{i=1}^{n} \bbV \left\{ \frac{K \left( \frac{X_i - X}{h} \right)}{h^d f(X_i)} Y_i \mid X_i, Y_i \right\} &= \frac{1}{n} \sum_{i=1}^{n} \frac{Y_i^2}{f(X_i)^2 h^{d}} \bbV \left\{ K \left( \frac{X_i - X}{h} \right) \mid X_i \right\}.
    \end{align*}
    Notice that the right-hand side is an average of non-negative bounded random variables because $Y^2$ is upper bounded and $f(X)^2$ is lower bounded away from zero by assumption, and because
    \begin{align*}
        0 \leq \frac{1}{h^d} \bbV \left\{ K \left( \frac{X_i - X}{h} \right) \mid X_i \right\} &\leq \frac{1}{h^d} \bbE \left\{ K \left( \frac{X_i - X}{h} \right)^2 \mid X_i \right\} \\
        &= \frac{1}{h^d} \int_{\mathcal{X}} K\left( \frac{X_i - t}{h} \right)^2 f(t) dt  \\
        &= \int_{\mathcal{X}} K\left( u \right)^2 f(X_i - uh) du \asymp 1,
    \end{align*}
    where the final line follows by a change of variables and because the density is upper and lower bounded and $\int K(u)^2 du \asymp 1$ by assumption. 
    
    \medskip
    
    Therefore, the diagonal terms, multiplied by $nh^d$, are a sample average of bounded random variables with common mean. By a strong law of large numbers for triangular arrays of bounded random variables (Lemma \ref{lem:slln}),
    \begin{equation} \label{eq:lim_as}
        nh^d \frac{1}{n^2} \sum_{i=1}^{n} \bbV \left\{ \frac{K \left( \frac{X_i - X}{h} \right)}{h^d f(X_i)} Y_i \mid X_i, Y_i \right\} \inas \lim_{n \to \infty} \bbE\left[ \frac{Y_i^2}{f(X_i)^2 h^{d}} \bbV \left\{ K \left( \frac{X_i - X}{h} \right) \mid X_i \right\} \right],
    \end{equation}
    should the limit on the right-hand side exist.  Indeed, this limit exists. First, notice that
    \begin{align}
        \lim_{n \to \infty} \bbE &\left[ \frac{Y_i^2}{f(X_i)^2 h^{d}} \bbV \left\{ K \left( \frac{X_i - X}{h} \right) \mid X_i \right\} \right] \nonumber \\ 
        &=\lim_{n \to \infty}  \int_{\mathcal{X}} \frac{\bbE(Y^2 \mid X = s)}{f(s) h^d} \left[ \int_{\mathcal{X}} K\left( \frac{s-t}{h} \right)^2 f(t) dt  - \left\{ \int_{\mathcal{X}} K \left(\frac{s-t}{h} \right) f(t) dt \right\}^2 \right] ds \nonumber \\
        &= \lim_{n \to \infty}  \int_{\mathcal{X}} \frac{\bbE(Y^2 \mid X = s)}{f(s)} \left\{ \int_{\mathcal{U}} K\left( u \right)^2 f(s + uh) du \right\} ds \nonumber \\
        &\hspace{1in} - h^d \int_{\mathcal{X}} \frac{\bbE(Y^2 \mid X = s)}{f(s)} \left\{ \int_{\mathcal{U}} K \left(u \right) f(s + uh) du \right\}^2 ds. \label{eq:diff_ker_limit}
    \end{align}
    where the second equality follows by a change of variables, linearity of integration, and the symmetry of $K$.  By the assumed upper bound on $Y$ and lower bound on $f(X)$ and the integrability of $K$, and because $h^d \stackrel{n\to\infty}{\longrightarrow} 0$, the limit of the second summand is zero.
    
    \medskip

    Meanwhile, by the boundedness of $Y$ and $f(X)$, the integrability of $K^2$, and Fubini's theorem,
    $$
    \int_{\mathcal{X}} \frac{\bbE(Y^2 \mid X = s)}{f(s)} \left\{ \int_{\mathcal{U}} K\left( u \right)^2 f(s + uh) du \right\} ds = \int_{\mathcal{X}} \int_{\mathcal{U}}  \bbE(Y^2 \mid X = s) K(u)^2 \frac{f(s + uh)}{f(s)} du ds.
    $$
    Moreover, by the assumed continuity of $f$, $K(u)^2 f(s + uh) \stackrel{n \to \infty}{\longrightarrow} K(u)^2 f(s)$ uniformly in $u$ at all $s$ except for a set of Lebesgue measure-zero on  the boundary of $\mathcal{X}$. Indeed, at those points, if $u$
    ``points'' outside $\mathcal{X}$, then the limit is zero because $f(s + uh) = 0$ for all $h$.  This, combined with the boundedness of $Y$, $f$, and $K$ and the integrability of $K^2$, implies, by the dominated convergence theorem, that 
    \begin{align}
        \lim_{n \to \infty}  \int_{\mathcal{X}} \frac{\bbE(Y^2 \mid X = s)}{f(s)} \left\{ \int_{\mathcal{U}} K\left( u \right)^2 f(s + uh) du \right\} ds &=   \int_{\mathcal{X}} \int_{\mathcal{U}} \bbE(Y^2 \mid X = s) K(u)^2 \lim_{n\to\infty} \frac{f(s + uh)}{f(s)} duds \nonumber \\
        &= \int_{\mathcal{X}} \int_{\mathcal{X}} \bbE(Y^2 \mid X = s) K(u)^2 duds \nonumber \\
        &= \left\{ \int_{\mathcal{X}} \bbE(Y^2 \mid X = s) ds \right\} \left\{ \int_{\mathcal{X}} K(u)^2 du \right\} \nonumber \\
        &= \bbE \left\{ \frac{Y^2}{f(X)} \right\} \bbE \left\{ \frac{K(X)^2}{f(X)} \right\} \label{eq:lim_det}
    \end{align}
    Therefore, because the limits of both summands in \eqref{eq:diff_ker_limit} exist, the limit of the difference is the difference of the limits.  Hence, combining \eqref{eq:lim_as} and \eqref{eq:lim_det} yields
    \begin{equation} \label{eq:var_as}
        nh^d \frac{1}{n^2} \sum_{i=1}^{n} \bbV \left\{ \frac{K \left( \frac{X_i - X}{h} \right)}{h^d f(X_i)} Y_i \mid X_i, Y_i \right\} \inas \bbE \left\{ \frac{Y^2}{f(X)} \right\} \bbE \left\{ \frac{K(X)^2}{f(X)} \right\}.
    \end{equation}
    
    \medskip
    
    Next, consider the sum of off-diagonal covariance terms.  First, because the kernel is localized, notice that when the covariates are far apart such that $\lVert X_i - X_j \rVert > 2h$, then the two terms inside the covariance do not share non-zero support because $K(x/h) \lesssim \one( \lVert x \rVert \leq h)$.  For $f(X)$ and $g(X)$ that do not share non-zero support, $\bbE\{ f(X) g(X) \} = 0$ and so $\left|\cov\{ f(X), g(X) \} \right| = \left| \bbE \{ f(X) \} \bbE\{ g(X) \} \right|.$  In that case,
    \begin{align}
        \left| \cov \left\{ \frac{K \left( \frac{X_i - X}{h} \right)}{h^d f(X_i)} Y_i, \frac{K \left( \frac{X_j - X}{h} \right)}{h^d f(X_j)} Y_j \mid X_i, Y_i, X_j, Y_j \right\} \right| &= \left| \bbE \left\{ \frac{K \left(\frac{X_i - X}{h} \right)}{h^d f(X_i)} Y_i \right\} \bbE \left\{ \frac{K \left(\frac{X_j - X}{h} \right)}{h^d f(X_j)} Y_j \right\} \right| \nonumber \\
        &\lesssim \left| \frac{1}{h^{2d}} \int K \left( \frac{X_i - x}{h} \right)dx \int K \left( \frac{X_j - x}{h} \right)dx \right| \nonumber \\
        &= 1 \label{eq:cov_upper}
    \end{align}
    where the second line follows by lower bounded density and upper bounded outcome, while the final line follows by a change of variables and because $\int K(x) dx = 1$.
    
    \medskip
    
    Otherwise, when the covariates are far apart, the covariance can be upper bounded by the product of standard deviations by Cauchy-Schwarz, i.e., 
    \begin{align}
        &\left| \cov \left\{ \frac{K \left( \frac{X_i - X}{h} \right)}{h^d f(X_i)} Y_i, \frac{K \left( \frac{X_j - X}{h} \right)}{h^d f(X_j)} Y_j \mid X_i, Y_i, X_j, Y_j \right\} \right| \\
        &\leq \sqrt{ \bbV \left\{ \frac{K \left( \frac{X_i - X}{h} \right)}{h^d f(X_i)} Y_i \mid X_i, Y_i \right\} } \sqrt{ \bbV \left\{ \frac{K \left( \frac{X_j - X}{h} \right)}{h^d f(X_j)} Y_j \mid X_j,Y_j \right\} } \nonumber \\
        &\lesssim \frac{1}{h^{2d}} \sqrt{ \bbV \left\{  K \left( \frac{X_i - X}{h} \right) \mid X_i \right\} } \sqrt{ \bbV \left\{  K \left( \frac{X_j - X}{h} \right) \mid X_j \right\} } \nonumber \\
        &\lesssim \frac{1}{h^d}, \label{eq:cov_upper2}
    \end{align}
    where the second line follows because $Y$ and $f(X)$ are upper and lower bounded, respectively, by assumption and $X_i$ and $X_j$ are iid, and the third line follows because $\bbV \left\{  K \left( \frac{X_j - X}{h} \right) \mid X_j \right\} = h^d \int K(u)^2 du - h^{2d} \left\{ \int K(u) du \right\}^2 \lesssim h^d$ by a change of variables because $\int K(u)^2 du \lesssim 1$ by assumption.
    
    \medskip
    
    Then, the sum of off-diagonal covariance terms can be bounded by counting how many covariates are close and multiplying the count by the upper bound $\frac{1}{h^d}$ discussed in the previous paragraph.  Let $P_n$ denote (two times) the number of close covariate pairs as, i.e., 
    \begin{equation} \label{eq:close_covariates}
        P_n = \sum_{i=1}^{n} \sum_{j \neq i} \one \left( \lVert X_i - X_j \rVert \leq 2h \right).
    \end{equation}
    Combining \eqref{eq:cov_upper}, \eqref{eq:cov_upper2}, and \eqref{eq:close_covariates}, we have
    \begin{equation}
        \left| \frac{1}{n^2} \sum_{i=1}^{n} \sum_{j \neq i}  \cov \left\{ \frac{K \left( \frac{X_i - X}{h} \right)}{h^d f(X_i)} Y_i, \frac{K \left( \frac{X_j - X}{h} \right)}{h^d f(X_j)} Y_j \mid X_i, Y_i, X_j, Y_j \right\} \right| \lesssim \frac{P_n}{n^2} \frac{1}{h^d} + 1.
    \end{equation}
    Lemma~\ref{lem:separated_covariates} establishes that $\frac{P_n}{n} \inas 0$ under the assumed condition on the bandwidth that $nh^d \asymp n^{-\alpha}$ for some $\alpha > 0$. Hence, 
    \begin{equation} \label{eq:cov_as}
        nh^d \left[ \left| \frac{1}{n^2} \sum_{i=1}^{n} \sum_{j \neq i}  \cov \left\{ \frac{K \left( \frac{X_i - X}{h} \right)}{h^d f(X_i)} Y_i, \frac{K \left( \frac{X_j - X}{h} \right)}{h^d f(X_j)} Y_j \mid X_i, Y_i, X_j, Y_j \right\} \right| \right] \lesssim \frac{P_n}{n} + nh^d \inas 0.
    \end{equation}
    In conclusion, \eqref{eq:var_as} and \eqref{eq:cov_as} and the continuous mapping theorem imply the result.
\end{proof}

\begin{lemma}  \label{lem:cdalpr_third_moment} \emph{(Covariate-density-adapted kernel regression third moment upper bound)}
    Suppose Assumptions~\ref{asmp:dgp}, \ref{asmp:bdd_density}, \ref{asmp:density_smooth}, and \ref{asmp:boundedAY} hold and $\widehat \mu(x)$ is a either a higher-order or smooth covariate-density-adapted kernel regression estimator (Estimator~\ref{est:cdalpr_higher} or~\ref{est:cdalpr_smooth}) for $\mu(x)$ constructed on $D_\mu$. Then, when $nh^d \asymp n^{-\alpha}$ for $\alpha > 0$ as $n \to \infty$, 
    \begin{equation}
        nh^{\frac{3d}{2}}  \bbE\{ | \widehat \mu(X) |^3 \mid D^n \} \inas 0.
    \end{equation}
\end{lemma}

\begin{proof}
    We have 
    $$
    \bbE \{ \left|\widehat \mu(X) \right|^3 \mid D^n \} \lesssim \frac{1}{n^3 h^{3d}} \sum_{i=1}^{n} \sum_{j=1}^{n} \sum_{k=1}^{n} \left| \bbE \left\{ K \left(\frac{X_i - X}{h} \right) K \left(\frac{X_j - X}{h} \right) K \left(\frac{X_k - X}{h} \right) \mid X_i, X_j, X_k \right\} \right|.
    $$
    by Assumption~\ref{asmp:bdd_density} and Assumption~\ref{asmp:boundedAY} (bounded density and $Y$). By the localizing property of the kernel, all three covariates must be close to share non-zero support, and then the expectation of their product is $\lesssim h^d$ by the boundedness of the covariate density. Otherwise, $ \bbE \left\{ K \left(\frac{X_i - X}{h} \right) K \left(\frac{X_j - X}{h} \right) K \left(\frac{X_k - X}{h} \right) \mid X_i, X_j, X_k \right\} = 0$. Therefore, it suffices to consider the cases when all three covariates are close.  
    
    \medskip
    
    First, notice that the triple sum can be decomposed as
    $$
    \sum_{i=1}^n \sum_{j=1}^n \sum_{k=1}^n = \sum_{i = j = k} + \sum_{i \neq j = k} + \sum_{i = j \neq k} + \sum_{i =k \neq j} + \sum_{i\neq j \neq k},
    $$
    i.e., there are $n$ permutations where the indexes are the same, $3$ sets of double sums where two indexes are the same, left-overs are a U-statistic of order 3. Letting $P_n$ denote twice the number of covariate pairs, as in Lemma \ref{lem:separated_covariates}, and \\ $Q_n := \sum_{i\neq j \neq k}^{n} \one \left( \lVert X_i - X_j \rVert \leq 2h \right) \one \left( \lVert X_i - X_k \rVert \leq 2h \right) \one \left( \lVert X_j - X_k \rVert \leq 2h \right)$, it follows that
    $$
    \sum_{i=1}^{n} \sum_{j=1}^{n} \sum_{k=1}^{n} \one \left( \lVert X_i - X_j \rVert \leq 2h \right) \one \left( \lVert X_i - X_k \rVert \leq 2h \right) \one \left( \lVert X_j - X_k \rVert \leq 2h \right) = n + 3P_n + Q_n
    $$
    because the observations are iid. Hence, 
    \begin{equation}
        \frac{1}{n^3 h^{3d}} \sum_{i=1}^{n} \sum_{j=1}^{n} \sum_{k=1}^{n} \left| \bbE \left\{ K \left(\frac{X_i - X}{h} \right) K \left(\frac{X_j - X}{h} \right) K \left(\frac{X_k - X}{h} \right) \mid X_i, X_j, X_k \right\} \right| \lesssim \frac{h^d}{n^3 h^{3d}} \left(n + 3P_n + Q_n \right).
    \end{equation}
    Therefore,
    \begin{align*}
        nh^{\frac{3d}{2}} \bbE \{ \left|\widehat \mu(X) \right|^3 \mid D^n \} &\lesssim 	\frac{nh^{\frac{3d}{2}}}{n^3 h^{3d}} \sum_{i=1}^{n} \sum_{j=1}^{n} \sum_{k=1}^{n} \left| \bbE \left\{ K \left(\frac{X_i - X}{h} \right) K \left(\frac{X_j - X}{h} \right) K \left(\frac{X_k - X}{h} \right) \mid X_i, X_j, X_k \right\} \right| \\
        &\lesssim \frac{nh^{\frac{5d}{2}}}{n^3 h^{3d}} (n + P_n + Q_n) = \frac{n + P_n + Q_n}{n^2 h^{d/2}} \inas 0,
    \end{align*}	
    where the convergence results follows by Lemmas \ref{lem:separated_covariates} and \ref{lem:separated_triples}, which establish $\frac{P_n}{n} \inas 0$ and $\frac{Q_n}{n} \inas 0$, and the condition on the bandwidth that $\varepsilon < \frac{4(\alpha + \beta)}{d}$, which implies $\frac{1}{nh^{d/2}} = o(1)$. 
\end{proof}

Our penultimate result shows that the smooth covariate-density-adapted kernel regression, averaged over the training points, is itself H\"{o}lder smooth.  Notice that the result relies on the kernel being continuous, which is a mild assumption, but may not hold for the higher-order kernel.
\begin{lemma}
    \label{lem:cdakernel_holder} \emph{(Smooth covariate-density-adapted kernel regression is H\"{o}lder smooth)}
    Suppose Assumptions~\ref{asmp:dgp}, \ref{asmp:bdd_density}, \ref{asmp:holder}, and \ref{asmp:density_smooth} hold, and $\widehat \mu(x)$ is a smooth covariate-density-adapted kernel regression estimator (Estimator~\ref{est:cdalpr_smooth}).  Then, 
    $$
    \bbE \{ \widehat \mu(x) \} \in \text{H\"{o}lder}(\beta).
    $$
\end{lemma}

\begin{proof}
    To establish that $\bbE \{ \widehat \mu(x) \} \in \text{H\"{o}lder}(\beta)$, we will show that (1) it is $\lfloor \beta \rfloor$-times continuously differentiable with bounded partial derivatives, and (2) its $\lfloor \beta \rfloor$ order partial derivatives satisfy the H\"{o}lder continuity condition. 
    
    \medskip
    
    For $x \in \mathcal{X}$,
    $$
    \bbE \{ \widehat \mu(x) \} = \bbE \left\{ \frac{1}{n} \sum_{i=1}^{n} \frac{K\left(\frac{X_i - x}{h} \right)}{h^d f(X_i)} Y_i \right\} = \frac{1}{h^d} \int K \left(\frac{t - x}{h} \right)\mu(t) dt = \int K(u) \mu(uh + x) du,
    $$
    by the definition of the estimator and substitution. Let $D^j$ denote an arbitrary multivariate partial derivative operator of order $j > 0$.  Then, for $j \leq \lfloor \beta \rfloor$,
    $$
    D^j \bbE \{ \widehat \mu(x) \} = D^j \int K(u) \mu(uh + x) du  = \int K(u) D^j \mu(uh + x) du,
    $$
    where the second equality follows by the continuity and integrability assumptions on $K(u)$ and Leibniz' integral rule. Because $\mu \in \text{H\"{o}lder}(\beta)$ by Assumption~\ref{asmp:holder}, $D^j \mu(uh + x)$ exists and is continuous.  Moreover, for any two continuous functions $f$ and $g$, $\int f(x) g(x) dx$ is continuous, and therefore $D^j \bbE \{ \widehat \mu(x) \}$ exists and is continuous.  For boundedness, notice that	
    $$
    \left| D^j \bbE \{ \widehat \mu(x) \} \right| = \left| \int K(u) D^j \mu(uh + x) du \right| \leq \int \left|K(u) \right| \left|D^j \mu(uh + x) \right|du \lesssim 1,
    $$
    because $\mu \in \text{H\"{o}lder}(\beta)$ by Assumption~\ref{asmp:holder} and by the integrability of $K$.  Finally, for the H\"{o}lder continuity condition on the $\lfloor \beta \rfloor$ derivative, notice that for $x, x^\prime \in \mathcal{X}$, 
    \begin{align*}
        \left| D^{\lfloor \beta \rfloor} \bbE \{ \widehat \mu(x) \} - D^{\lfloor \beta \rfloor} \bbE \{ \widehat \mu(x^\prime) \} \right| &= \left| \int K(u) D^{\lfloor \beta \rfloor} \mu(uh + x) du - \int K(u) D^{\lfloor \beta \rfloor} \mu(uh + x^\prime) du \right| \\
        &= \left| \int K(u) \left\{ D^{\lfloor \beta \rfloor} \mu(uh + x) - D^{\lfloor \beta \rfloor} \mu(uh + x^\prime) \right\} du \right| \\
        &\leq  \int \left|K(u)\right|  \left| D^{\lfloor \beta \rfloor} \mu(uh + x) - D^{\lfloor \beta \rfloor} \mu(uh + x^\prime) \right| du \\
        &\lesssim \int \left|K(u)\right|  \lVert x - x^\prime \rVert^{\beta - \lfloor \beta \rfloor} du \\
        &\lesssim   \lVert x - x^\prime \rVert^{\beta - \lfloor \beta \rfloor},
    \end{align*}
    where the first line follows by the same argument as above, the second by linearity of the integral, the penultimate line by the H\"{o}lder assumption of $\mu$, and the final line by the integrability assumption on the kernel. Therefore, $\bbE \{ \widehat \mu(x) \}$ satisfies the conditions of being a H\"{o}lder$(\beta)$ smooth function.
\end{proof}

Our final result establishes that the smooth covariate-density adapted kernel regression estimator is bounded if the relevant outcome is bounded.
\begin{lemma}
    \label{lem:cdakernel_bounded} \emph{(Smooth covariate-density-adapted kernel regression is bounded)}
    Suppose Assumptions~\ref{asmp:dgp}, \ref{asmp:bdd_density}, \ref{asmp:holder}, \ref{asmp:density_smooth}, and \ref{asmp:boundedAY} hold, and $\widehat \mu(x)$ is a smooth covariate-density-adapted kernel regression estimator (Estimator~\ref{est:cdalpr_smooth}).  Then, there exists $M > 0$ such that $\left| \widehat \mu(X) \right| \leq M.$
\end{lemma}

\begin{proof}
    This follows immediately because the covariate density and outcome are bounded by assumption, and the kernel is bounded by construction.
\end{proof}

\section{Section~\ref{sec:known} results: proofs of Theorems~\ref{thm:minimax} and~\ref{thm:inference}} \label{app:known}

For Theorems~\ref{thm:minimax} and~\ref{thm:inference}, we use properties of Sobolev smooth functions. Let $L_p (\bbR^d)$ denote the space of $p$-fold Lebesgue-integrable functions, i.e., 
$$
L_p (\bbR^d) = \left\{ f: \bbR^d \to \bbR : \int_{\bbR^d} \left| f(x) \right|^p dx < \infty \right\}.
$$
We will denote the class of Sobolev$(s, p)$ smooth functions as $H_p^s (\bbR^d)$.  For $s \in \mathbb{N}$, these classes can be defined as
$$
H_p^s (\bbR^d) = \left\{ f \in L_p (\bbR^d) : D^t f \in L_p (\bbR^d) \forall \ |t| \leq s : \left( \int_{\bbR^d} \left| f(x) \right|^p dx \right)^{1/p} + \sum_{|s| = t} \lVert D^t f \rVert_p < \infty \right\},
$$
where $D^t$ is the multivariate partial derivative operator (see Section~\ref{sec:notation}). One can also define Sobolev smooth functions for non-integer $s$ through their Fourier transform (e.g., \citet{gine2021mathematical} Chapter 4). We will omit such a definition here because it requires much additional and unnecessary notation, but still use $H_p^s(\bbR^d)$ to refer to such function classes. Importantly, H\"{o}lder$(s) = H_\infty^s(\bbR^d)$, and $H_\infty^s(\bbR^d) \subseteq H_p^s(\bbR^d)$ for $p \leq \infty$, i.e., H\"{o}lder classes are contained within Sobolev classes of the same smoothness.

\medskip

We begin with the following result, Lemma~\ref{lem:bias_kernel}, which is used in the proof of Theorem~\ref{thm:minimax}.  Lemma~\ref{lem:bias_kernel} follows very closely from Theorem 1 in \citet{gine2008simple} (also, Lemmas 4.3.16 and 4.3.18 in \citet{gine2021mathematical}).  The higher order property of the kernel in Estimator~\ref{est:cdalpr_higher} allows us to generalize the result to higher smoothness.
\begin{lemma} \label{lem:bias_kernel}
    Suppose Assumptions~\ref{asmp:dgp},~\ref{asmp:bdd_density},~\ref{asmp:holder}, and~\ref{asmp:density_smooth} hold, and $\widehat \mu(x)$ is a higher-order covariate-density-adapted kernel regression estimator (Estimator~\ref{est:cdalpr_higher}) for $\mu(x)$ constructed on $D_\mu$.  Let $g \in \text{H\"{o}lder}(\alpha)$.  Then,
    $$
    \sup_{x \in \mathcal{X}}  \left| \bbE \left( g(X) \Big[ \bbE \{ \widehat \mu(X) \mid X \} - \mu(X) \Big] \mid X = x \right) \right| \lesssim h_\mu^{\alpha + \beta}.
    $$
\end{lemma}

\begin{proof}
    Let $h \equiv h_\mu$ throughout. First note that
    \begin{align*}
        \bbE \{\widehat \mu (x)  \} &= \bbE \left\{ \sum_{Z_i \in D_\mu}  \frac{K \left( \frac{X_i - x}{h} \right)}{nh^d f(X_i)} Y_i \right\}  \\
        &= \bbE \left\{ \frac{K \left( \frac{X - x}{h} \right)}{h^d f(X)} \mu(X) \right\} \\
        &= \int_{t \in \mathcal{X}} \frac{K \left( \frac{t - x}{h} \right)}{h^d} \mu(t) dt. 
    \end{align*}
    
    Since $\mathcal{X}$ is compact in $\bbR^d$, we evaluate the following integrals over $\bbR^d$, with the understanding that outside the relevant sets the integrand evaluates to zero (e.g., after the change of variables).  Then, letting $g(x)f(x) = g f(x)$, $\overline h(x) = h(-x)$, and $\ast$ denote convolution,
    \begin{align*}
        \bbE &\left( g(X) \Big[ \bbE \{ \widehat \mu(X) \mid X \} - \mu(X) \Big] \mid X = x \right) \\
        &= \int_{x \in \bbR^d} g f(x) \left\{ \int_{t \in \bbR^d} \frac{1}{h^d} K \left( \frac{t - x}{h} \right) \mu(t) dt - \mu(x) \right\} dx \\
        &= \int_{x \in \bbR^d} g f(x) \left\{ \int_{u \in \bbR^d} K(-u) \mu(x - uh) du - \mu(x) \right\} dx &u = (x-t)/h \\
        &= \int_{x \in \bbR^d} g f(x) \left\{ \int_{u \in \bbR^d} K(u) \mu(x - uh) du - \mu(x) \right\} dx \\
        &= \int_{x \in \bbR^d} gf(x) \left[ \int_{u \in \bbR^d} K(u) \{ \mu(x - uh) - \mu(x) \} du \right] dx \\
        &= \int_{u \in \bbR^d} K(u) \left[ \int_{x \in \bbR^d} g f(x) \left\{ \mu(x - uh) - \mu(x) \right\} dx \right] du \\
        &= \int_{u \in \bbR^d} K(u) \left[ \int_{x \in \bbR^d} g f(x) \overline \mu(uh - x) - g f(x) \overline \mu(-x) dx \right] du \\
        &= \int_{u \in \bbR^d} K(u) \left\{ g f \ast \overline \mu (uh) - g f \ast \overline \mu (0) \right\} du. 
    \end{align*}
    where the first line follows by definition, the second by substitution, the third because $K$ is symmetric, the fourth because $\int K = 1$, the fifth by Fubini's theorem, and the last two again by definition. 
    
    \medskip
    
    Next, notice that $g f \in \text{H\"{o}lder}(\alpha) \subseteq H_2^\alpha (\bbR)$ because $g \in \text{H\"{o}lder}(\alpha)$ and $f \in \text{H\"{o}lder}(\alpha \vee \beta)$ by Assumption~\ref{asmp:density_smooth}, and $\mu \in \text{H\"{o}lder}(\beta) \implies \overline \mu \in \text{H\"{o}lder}(\beta) \subseteq H_2^\beta (\bbR)$.  Therefore, by Lemma 12 and Remark 11i in \citet{gine2008uniform}, $gf \ast \overline{\mu} \in \text{H\"{o}lder}(\alpha + \beta)$.  
    
    \medskip
    
    The rest of the proof continues by a standard Taylor expansion analysis of higher-order kernels.  See, e.g., \citet{scott2015multivariate} Chapter 6.  Let $D^j f$ denote the multivariate partial derivative of $f$ of order $j$ and let $\eta(x) = gf \ast \overline \mu(x)$ for simplicity. Then, we have
    \begin{align*}
        \int_{u} &K(u) \left\{ \eta(uh) - \eta(0) \right\} du \\
        &\hspace{-0.3in}= \int_{u} K(u) \Bigg[\sum_{0 < |j| < \lfloor \alpha + \beta \rfloor - 1} \frac{D^j \eta(0)}{j^!} (uh)^j   \\
        &\hspace{0.4in} + \sum_{|k| = \lfloor \alpha + \beta \rfloor} \frac{\lfloor \alpha + \beta \rfloor}{k^!} \int_{0}^{1} (1-t)^{\lfloor \alpha + \beta \rfloor - 1} \left\{ D^k \eta (tuh) - D^k \eta(0) \right\} (uh)^{\lfloor \alpha + \beta \rfloor} dt \Bigg]  du \\
        &\lesssim   \int_{u \in \bbR^d} K(u) (h \lVert u \rVert)^{\alpha + \beta - \lfloor \alpha + \beta \rfloor} (h \lVert u \rVert)^{\lfloor \alpha + \beta \rfloor} du  \\
        &= h^{\alpha + \beta} \int_{u \in \bbR^d} K(u) \lVert u \rVert^{\alpha + \beta} du \lesssim h^{\alpha + \beta},
    \end{align*}
    where the first line follows by a Taylor expansion of the difference $\eta(uh) - \eta(0)$; the second because (1) $\eta \in$ H\"{o}lder$(\alpha + \beta)$, (2) the kernel is of order at least $\lceil \alpha + \beta \rceil$, (3) $|u^k| \leq \lVert u \rVert^k$ (where $\lVert \cdot \rVert$ is the Euclidean norm), and (4) $\int_{0}^{1} (1-t)^{\lfloor \beta \rfloor - 1} = \frac{1}{\lfloor \beta \rfloor}$; and the final line follows again by assumption on the kernel.	
    
    \medskip
    
    \noindent The supremum over $x \in \mathcal{X}$ follows because $\mathcal{X}$ is compact by assumption.
\end{proof}

\color{black}

\thmminimax*

\begin{proof}
    Assume without loss of generality that $\widehat \pi$ is the consistent estimator and $\widehat \mu$ the undersmoothed estimator, with $h_\mu \asymp n^{-\frac{2}{2\alpha + 2\beta + d}}$.   Since $\widehat \mu$ and $\widehat \pi$ were trained on separate independent samples, the bias satisfies 
    $$
    \bbE \left( \widehat \psi_n - \psi \right) = \bbE \left \{ \widehat \varphi(Z) - \varphi(Z) \right\} = \bbE \left( \big[ \bbE \{ \widehat \mu(X) \mid X \} - \mu(X) \big] \big[ \bbE \{ \widehat \pi(X) \mid X \} - \pi(X) \big] \right).
    $$
    Lemma~\ref{lem:cdakernel_holder} demonstrates that $\bbE \{ \widehat \pi(x) \} \in \text{H\"{o}lder}(\alpha)$ under the assumptions given on the kernel in Estimator~\ref{est:cdalpr_smooth}. Therefore, $\bbE \{ \widehat \pi(x) \} - \pi(x) \in \text{H\"{o}lder}(\alpha) \subseteq H_2^\alpha(\bbR^d)$. Thus, by Lemma~\ref{lem:bias_kernel},  \begin{equation}\label{eq:dr_bias}
        \left| \bbE \left( \widehat \psi_n - \psi \right) \right| \lesssim h_\mu^{\alpha + \beta} \asymp n^{-\frac{2\alpha + 2\beta}{2\alpha + 2\beta + d}}.
    \end{equation}
    Because $\varphi$ is Lipschitz in its nuisance functions, and by the same arguments as in Lemma~\ref{lem:expansion} and Proposition~\ref{prop:spectral}, and by \eqref{eq:cdalpr_covariance} in Lemma~\ref{lem:cdalpr_upper_bounds}, the remainder term in Lemma~\ref{lem:expansion} satisfies
    $$
    R_{2,n} = O_\bbP \left( \frac{\lVert b_{\pi}^2 \rVert_\infty +  \lVert s_{\pi}^2 \rVert_\infty + \lVert b_{\mu}^2 \rVert_\infty + \lVert s_{\mu}^2 \rVert_\infty}{n} \right),
    $$
    Then, by \eqref{eq:cdalpr_variance} in Lemma~\ref{lem:cdalpr_upper_bounds},
    \begin{equation} \label{eq:minimax_var}
        R_{2,n} = O_\bbP \left( \frac{1}{n^2 h_\mu^d} \right) = O_\bbP \left( n^{-\frac{4\alpha + 4\beta}{2\alpha + 2\beta + d}} \right).
    \end{equation} 
    Hence, when $\frac{\alpha + \beta}{2} > d/4$, the CLT term dominates the expansion --- as in Theorem \ref{thm:semiparametric} --- whereas in the non-$\sqrt{n}$ regime bias and variance are balanced.
\end{proof}

\thminference*
        
\begin{proof}
    The proof relies on several helper lemmas stated after this proof.  We focus on the regime where $\frac{\alpha + \beta}{2} < \frac{d}{4}$, although a standard CLT could apply in the smoother regime.  In this non-$\sqrt{n}$ regime, the undersmoothed DCDR estimator does not achieve $\sqrt{n}$-convergence and we must instead prove slower-than-$\sqrt{n}$ convergence.
        
    \medskip
        
    We omit $Z$ arguments (e.g., $\varphi(Z) \equiv \varphi$) and let $D^n = \{ D_\mu, D_\pi \}$ denote all the training data.  First, note that by Lemma \ref{lem:cond_variance_clt}, $\bbV(\widehat \varphi \mid D^n) > 0$ almost surely, so that division by $\bbV(\widehat \varphi \mid D^n)$ is well-defined almost surely. Then, by the definition of $\widehat \psi_n, \psi_{ecc}, \widehat \varphi,$ and $\varphi$ and adding zero and multiplying by one, we have the following decomposition:
    \begin{align*}
        \frac{\widehat \psi_n - \psi_{ecc}}{\sqrt{\bbV ( \widehat \varphi \mid D^n ) / n}} &= \frac{\bbP_n \widehat \varphi - \bbE(\widehat \varphi)}{\sqrt{\bbV ( \widehat \varphi \mid D^n ) / n}} +\frac{\bbE (\widehat \varphi - \varphi)}{\sqrt{\bbV ( \widehat \varphi \mid D^n ) / n}} \\
        &= \frac{\bbP_n \widehat \varphi - \bbE(\widehat \varphi \mid D^n)}{\sqrt{\bbV ( \widehat \varphi \mid D^n ) / n}} + \frac{\bbE (\widehat \varphi \mid D^n) - \bbE(\widehat \varphi)}{\sqrt{\bbV ( \widehat \varphi \mid D^n ) / n}} + \frac{\bbE (\widehat \varphi - \varphi)}{\sqrt{\bbV ( \widehat \varphi \mid D^n ) / n}} \\
        &= \underbrace{\frac{\bbP_n \widehat \varphi - \bbE(\widehat \varphi \mid D^n)}{\sqrt{\bbV ( \widehat \varphi \mid D^n ) / n}}}_{\text{CLT}} + \underbrace{\sqrt{\frac{\bbV ( \widehat \varphi )}{\bbV ( \widehat \varphi \mid D^n )}}}_{T_1} \left\{ \underbrace{\frac{\bbE (\widehat \varphi \mid D^n) - \bbE(\widehat \varphi)}{\sqrt{\bbV ( \widehat \varphi ) / n}}}_{T_2} +  \underbrace{\frac{\bbE (\widehat \varphi - \varphi)}{\sqrt{\bbV ( \widehat \varphi ) / n}}}_{T_3} \right\}
    \end{align*}
    where the expectation and variance are over both the test and training data unless otherwise indicated by conditioning.  As the text underneath the underbraces indicates, we will show the limiting result for the first term --- the conditional standardized average.  That the unconditional standardized average converges to the conditional average in probability follows by Lemmas~\ref{lem:t1},~\ref{lem:t2}, and~\ref{lem:t3}, which establish that $T_1 = O_\bbP (1)$, $T_2 = o_\bbP(1)$, and $T_3 = o(1)$, respectively.  Therefore,
    $$
    T_1 (T_2 + T_3) = O_{\bbP}(1) \{ o_\bbP(1) + o(1) \} = o_\bbP (1).
    $$
    Returning to the CLT term, let $\Phi(\cdot)$ denote the cumulative distribution function of the standard normal. By iterated expectations and Jensen's inequality, 
    $$
    \lim_{n\to\infty} \sup_t \left| \bbP \left\{ \frac{\bbP_n \widehat \varphi - \bbE(\widehat \varphi \mid D^n)}{\sqrt{\bbV ( \widehat \varphi \mid D^n ) / n}} \leq t \right\} - \Phi(t) \right| \leq \lim_{n\to\infty} \bbE \left[ \sup_t \left|  \bbP \left\{ \frac{\bbP_n \widehat \varphi - \bbE(\widehat \varphi \mid D^n)}{\sqrt{\bbV ( \widehat \varphi \mid D^n ) / n}} \leq t \mid D^n \right\}  - \Phi(t) \right| \wedge 1 \right].
    $$
    Conditional on $D^n$, the summands in $\bbP_n \left\{ \frac{\widehat \varphi - \bbE(\widehat \varphi \mid D^n)}{\sqrt{\bbV ( \widehat \varphi \mid D^n ) / n}} \right\}$ are iid with mean zero and unit variance (almost surely). Therefore, by the Berry-Esseen inequality (Theorem 1.1, \citet{bentkus1996berry}), 
    $$
    \sup_t \left| \bbP \left\{ \frac{\bbP_n \widehat \varphi - \bbE(\widehat \varphi \mid D^n)}{\sqrt{\bbV ( \widehat \varphi \mid D^n ) / n}} \leq t \mid D^n \right\}  - \Phi(t) \right| \lesssim \frac{\bbE \Big[ \left| \widehat \varphi(Z) - \bbE \{ \widehat \varphi(Z) \mid D_\pi, D_\mu \} \right|^3 \mid D_\pi, D_\mu \Big]}{\sqrt{n}\ \bbV \{ \widehat \varphi(Z) \mid D_\pi, D_\mu \}^{3/2}} \inas 0,
    $$
    where convergence almost surely to zero follows by Lemma \ref{lem:lyapunov}.  Then, because \\ $\sup_t \left| \bbP \left\{ \frac{\bbP_n \widehat \varphi - \bbE(\widehat \varphi \mid D^n)}{\sqrt{\bbV ( \widehat \varphi \mid D^n ) / n}} \leq t \mid D^n \right\}  - \Phi(t) \right| \wedge 1$ is uniformly integrable and converges almost surely to zero, convergence in $L^1$ follows (Theorem 4.6.3, \citet{durrett2019probability}), i.e., 
    $$
    \lim_{n\to\infty} \bbE \left[ \sup_t \left|  \bbP \left\{ \frac{\bbP_n \widehat \varphi - \bbE(\widehat \varphi \mid D^n)}{\sqrt{\bbV ( \widehat \varphi \mid D^n ) / n}} \leq t \mid D^n \right\}  - \Phi(t) \right| \wedge 1 \right] = 0.
    $$
    Clearly, \eqref{eq:slow_clt} is satisfied.  Meanwhile, \eqref{eq:limiting_var} follows from Lemma \ref{lem:cond_variance_clt}.
\end{proof}

\begin{lemma} \label{lem:cond_variance_clt}
    Under the conditions of Theorem~\ref{thm:inference}, suppose without loss of generality that $\widehat \mu$ is the estimator with higher-order kernel $K_\mu$ and bandwidth scaling as $h_\mu \asymp n^{-\frac{2 + \varepsilon}{2\alpha + 2\beta + d}}$ while $\widehat \pi$ is consistent, smooth, and bounded. Then,
    \begin{equation}
        nh_\mu^d \bbV \{ \widehat \varphi(Z) \mid D^n \} \inas \bbE \left\{ \frac{\bbV (A \mid X)Y^2}{f(X)} \right\} \bbE\left\{ \frac{K_\mu (X)^2}{f(X)} \right\}.
    \end{equation}	
    If the roles of $\widehat \mu$ and $\widehat \pi$ were reversed, then
    \begin{equation}
        nh_\pi^d \bbV \{ \widehat \varphi(Z) \mid D^n \} \inas \bbE \left\{ \frac{\bbV (Y \mid X)A^2}{f(X)} \right\} \bbE\left\{ \frac{K_\pi (X)^2}{f(X)} \right\}.
    \end{equation}	
\end{lemma}
    
\begin{proof}
    Unless they are necessary for clarity, we omit $X$ and $Z$ arguments throughout for brevity (e.g., $\pi \equiv \pi(X)$).  By definition, 
    \begin{align}
        \bbV(\widehat \varphi \mid D^n) &= \bbV \{ (A - \widehat \pi) (Y - \widehat \mu) \mid D^n \} \nonumber \\
        &= \bbV \{ (A - \widehat \pi) Y \mid D^n \} + \bbV \{ (A - \widehat \pi) \widehat \mu \mid D^n \} + 2 \cov \{ (A - \widehat \pi) Y, (\widehat \pi - A) \widehat \mu \mid D^n \}. \label{eq:lem12a}
    \end{align}
    Since $\widehat \mu$ is the undersmoothed estimator, one might expect the second term in \eqref{eq:lem12a} to dominate this expansion and scale like $\bbV \{ \widehat \mu(X) \mid D^n \}$. We show this below.
    
    \medskip
    
    Starting with the first term in \eqref{eq:lem12a}, we have
    $$
    \bbV \{ (A - \widehat \pi) Y \mid D^n \} = O(1)
    $$
    by the boundedness assumption on $A$ and $Y$ in Assumption~\ref{asmp:boundedAY} and because $\widehat \pi$ is bounded by construction (Lemma \ref{lem:cdakernel_bounded}). Then, notice that the third term in \eqref{eq:lem12a} is upper bounded by the square root of the second term: by Cauchy-Schwarz and because $\bbV \{ (A - \pi) Y\} = O(1)$, 
    $$
    2 \left| \cov \{  (A - \pi)Y, (\widehat \pi - A) \widehat \mu \mid  D^n \} \right| \lesssim \sqrt{\bbV \{ (\widehat \pi - A) \widehat \mu \mid D^n \}}.
    $$
    Hence, demonstrating that the second term in \eqref{eq:lem12a} satisfies the almost sure limit when standardized by $nh_\mu^d$ ensures it will dominate the expansion.
    
    \medskip
    
    We have
    \begin{equation} \label{eq:varphi_cond_var}
        \bbV \{ ( A - \widehat \pi) \widehat \mu \} = \bbV\{ (\pi - \widehat \pi) \widehat \mu \mid D^n \} + \bbE \{ \bbV(A \mid X) \widehat \mu^2 \mid D^n \}.
    \end{equation}
    We will show that the first summand, when scaled by $nh^d$, converges to zero almost surely while the second summand satisfies the result.
    
    \medskip
    
    For the first summand in \eqref{eq:varphi_cond_var}, we have
    \begin{equation} \label{eq:pi_pihat_small}
        nh_\mu ^d \bbV\left\{ \left( \pi - \widehat \pi \right) \widehat \mu \mid D^n \right\} = \frac{1}{n} \sum_{D_\mu} \frac{Y_i^2}{f(X_i)^2 h_\mu^d} \bbV \left[ \{ \pi(X) - \widehat \pi(X) \} K \left(\frac{X_i - X}{h_\mu} \right) \mid X_i \right] + A_n
    \end{equation}
    where $A_n$ is the off-diagonal covariance terms.  $A_n \inas 0$ because $( \pi - \widehat \pi )$ is bounded by Assumption \ref{asmp:boundedAY} and Lemma \ref{lem:cdakernel_bounded}, and by the same argument as in Lemma \ref{lem:cdalpr_cond_lower_bounds}.
    
    \medskip
    
    The diagonal terms in \eqref{eq:pi_pihat_small} are a sample average of bounded random variables with common mean. Hence, by the strong law of large numbers for triangular arrays of bounded random variables (Lemma \ref{lem:slln}) and the continuous mapping theorem,
    \begin{equation} 
        nh_\mu^d \bbV\left\{ \left( \pi - \widehat \pi \right) \widehat \mu \mid D^n \right\} \inas \lim_{n \to \infty} \bbE\left(\frac{Y^2}{f(X)^2 h_\mu^d} \bbV \left[ \{ \pi(X^\prime) - \widehat \pi(X^\prime) \} K \left(\frac{X- X^\prime}{h} \right) \mid X \right] \right)  + 0,
    \end{equation}
    should the limit on the right-hand side exist.  Indeed, this limit exists, and is zero.  Notice that the expectation is taken over all the training data --- both $D_\mu$ and $D_\pi$.  Therefore,
    \begin{align*}
         &\bbE\left(\frac{Y^2}{f(X)^2 h_\mu^d} \bbV \left[ \{ \pi(X^\prime) - \widehat \pi(X^\prime) \} K \left(\frac{X- X^\prime}{h_\mu} \right) \mid X \right] \right) \\
         &\leq  \bbE\left(\frac{Y^2}{f(X)^2 h_\mu^d} \bbE \left[ \{ \pi(X^\prime) - \widehat \pi(X^\prime) \}^2 K \left(\frac{X- X^\prime}{h_\mu} \right)^2 \mid X \right] \right)  \\
         &= \bbE\left\{ \frac{Y^2}{f(X)^2 h_\mu^d} \bbE \left( \bbE_{D_\pi} \big[ \{ \pi(X^\prime) - \widehat \pi(X^\prime) \}^2 \mid D_\mu, X, X^\prime \big] K \left(\frac{X- X^\prime}{h_\mu} \right)^2 \mid X \right) \right\} \\
         &\leq \sup_{x^\prime \in \mathcal{X}} \bbE_{D_\pi} \big[ \{ \widehat \pi(x^\prime) - \pi(x^\prime) \}^2 \big] \bbE\left[ \frac{Y^2}{f(X)^2 h_\mu^d} \bbE \left\{ K \left(\frac{X - X^\prime}{h_\mu} \right)^2 \mid X \right\} \right] \\
         &= o(1),
    \end{align*}
    where the last line follows because $\sup_{x^\prime \in \mathcal{X}} \bbE_{D_\pi} \big[ \{ \widehat \pi(x^\prime) - \pi(x^\prime) \}^2 \big] = o(1)$ by Lemma \ref{lem:cdalpr_upper_bounds} and because the second multiplicand in the penultimate line is upper bounded (we added the $D_\pi$ subscript to emphasize that this expectation is over the training data for $\widehat \pi$).
    
    \medskip
    
    For the second summand in \eqref{eq:varphi_cond_var}, 
    \begin{align*}
        nh_\mu^d \bbE \{ \bbV(A \mid X) \widehat \mu^2 \mid D^n  \} &= \frac{1}{n} \sum_{i=1}^{n} \frac{Y_i^2}{f(X_i)^2 h_\mu^d} \bbE \left\{ K \left(\frac{X_i - X}{h_\mu}\right)^2 \bbV(A \mid X) \mid X_i \right\} + A_n
    \end{align*}
    where $A_n$ is the off-diagonal product terms.  $A_n \inas 0$ because $\bbV(A\mid X)$ is bounded by Assumption \ref{asmp:dgp} and by the same argument as in Lemma \ref{lem:cdalpr_cond_lower_bounds}.
    
    \medskip
    
    For the diagonal terms, because they are a sample average of bounded random variables with common mean, by a strong law of large numbers for triangular arrays (Lemma \ref{lem:slln}), 
    $$
    \hspace{-0.3in}\frac{1}{n} \sum_{i=1}^{n} \frac{Y_i^2}{f(X_i)^2 h_\mu^d} \bbE \left\{ K \left(\frac{X_i - X}{h_\mu} \right)^2 \bbV(A \mid X) \mid X_i \right\} \inas \lim_{n \to \infty} \bbE\left[ \frac{Y^2}{f(X)^2 h_\mu^d} \bbE \left\{ K \left( \frac{X - X^\prime}{h_\mu} \right)^2 \bbV(A \mid X^\prime) \mid X \right\} \right],
    $$
    should the limit on the right-hand side exist.  The rest of the proof follows by the same argument as in Lemma \ref{lem:cdalpr_cond_lower_bounds}. We have, by a change of variables and the symmetry of $K$,
    \begin{align*}
        \lim_{n \to \infty} &\bbE\left[ \frac{Y^2}{f(X)^2 h_\mu^d} \bbE \left\{ K \left( \frac{X - X^\prime}{h_\mu} \right) \bbV(A \mid X^\prime) \mid X \right\} \right] = \\
        &\hspace{1in} \lim_{n \to \infty} \int_{\mathcal{X}} \frac{\bbE(Y^2 \mid s)}{f(s)} \left\{ \int_{\mathcal{U}} K\left( u \right)^2 \bbV(A \mid s + uh) f(s + uh) du \right\} ds.
    \end{align*}
    By the boundedness of $Y$ and $f(X)$, the integrability of $K^2$, and Fubini's theorem, we can exchange integrals.  Then, by the assumed continuity of $f$ and $\bbV(A \mid x)$, $$
    K(u)^2 f(s + uh) \bbV(A \mid s + uh) \stackrel{n \to \infty}{\longrightarrow} K(u)^2 f(s) \bbV(A \mid s)
    $$
    uniformly in $u$ at all $s$ except for a set of Lebesgue measure-zero on  the boundary of $\mathcal{X}$. Indeed, at those points, if $u$ ``points'' outside $\mathcal{X}$, then the limit is zero because $f(s + uh) \bbV(A \mid s + uh) = 0$ for all $h$.  This, combined with the boundedness of $Y$, $f$, $A$, and $K$ and the integrability of $K^2$, implies, by the dominated convergence theorem, that 
    \begin{align}
        \frac{1}{n} \sum_{i=1}^{n} \frac{Y_i^2}{f(X_i)^2 h_\mu^d} \bbE \left\{ K \left(\frac{X_i - X}{h_\mu} \right)^2 \bbV(A \mid X) \mid X_i \right\} &\inas \int_{\mathcal{X}} \int_{\mathcal{U}} \bbE(Y^2 \mid X = s) K\left( u \right)^2 \bbV(A \mid s) du ds \nonumber \\
        &= \bbE \left\{ \frac{\bbE(Y^2 \mid X) \bbV(A \mid X)}{f(X)} \right\} \bbE \left\{ \frac{K(X)^2}{f(X)} \right\}. \label{eq:as_limit}
    \end{align}
    Then, plugging \eqref{eq:as_limit} into \eqref{eq:varphi_cond_var} and by the continuous mapping theorem,
    $$
    nh_\mu^d \bbV \{ ( A - \widehat \pi) \widehat \mu \} \inas \bbE\left\{ \frac{\bbV(A \mid X) Y^2}{f(X)} \right\} \bbE\left\{\frac{K(X)^2}{f(X)}\right\}.
    $$
    The result follows because $nh_\mu^d \bbV \{ ( A - \widehat \pi) \widehat \mu \}$ dominates the expansion in \eqref{eq:lem12a}.  The same argument follows with the roles of $\widehat \pi$ and $\widehat \mu$ reversed, but swapping the roles of $Y$ and $A$ and swapping $h_\mu$ and $K_\mu$ for $h_\pi$ and $K_\pi$.
\end{proof}

\begin{lemma} \label{lem:lyapunov}
    Under the setup from Theorem~\ref{thm:inference},
    \begin{equation} \label{eq:lyapunov}
        \frac{\bbE \Big[ \left| \widehat \varphi(Z) - \bbE \{ \widehat \varphi(Z) \mid D_\pi, D_\mu \} \right|^3 \mid D_\pi, D_\mu \Big]}{\sqrt{n}\ \bbV \{ \widehat \varphi(Z) \mid D_\pi, D_\mu \}^{3/2}} \inas 0.
    \end{equation} 
\end{lemma}

\begin{proof}
     Assume without loss of generality that $\widehat \pi$ is the smooth estimator (Estimator \ref{est:cdalpr_smooth}) and $\widehat \mu$ is the higher-order kernel estimator (Estimator \ref{est:cdalpr_higher}) so that $nh_\mu^d \to 0$ as $n \to \infty$, where $h_\mu$ is the bandwidth of the covariate-density-adapted kernel regression estimator. By Lemma~\ref{lem:cond_variance_clt}, the denominator in \eqref{eq:lyapunov} satisfies
    \begin{equation} \label{eq:lyapunov_denominator}
        nh_\mu^{\frac{3d}{2}} \sqrt{n} \bbV \{ \widehat \varphi(Z) \mid D^n \}^{3/2} = \left[ nh_\mu^d \bbV \{ \widehat \varphi(Z) \mid D^n \}\right]^{3/2} \inas \bbE\left\{ \frac{\bbV(A\mid X) Y^2}{f(X)} \right\}^{3/2} \bbE\left\{\frac{K(X)^2}{f(X)}\right\}^{3/2}.
    \end{equation}
    Meanwhile, the numerator in \eqref{eq:lyapunov} satisfies
    \begin{align*}
        \bbE \Big[ \left| \widehat \varphi(Z) - \bbE \{ \widehat \varphi(Z) \mid D^n \} \right|^3 \mid D^n \Big] &= \bbE \left[ \left| AY - \bbE(AY) + \widehat \pi(X) \{ \mu(X) - Y \} + \widehat \mu(X) \{ \pi(X) - A \} \right|^3 \mid D^n \right] \\
        &\lesssim \bbE \left[ \left| AY - \bbE(AY) \right|^3 \mid D^n \right] \\
        &+ \bbE \left[ \left|\widehat \pi(X) \{ \mu(X) - Y \} \right|^3 \mid D^n \right]  \\
        &+  \bbE \left[ \left|\widehat \mu(X) \{ \pi(X) - A \} \right|^3 \mid D^n \right] \\
        &= O \left[ 1 + \bbE \left\{ \left|\widehat \mu(X) \right|^3 \mid D^n \right\} \right]
    \end{align*}
    where the first line follows by definition and canceling terms and the last because $A$, $Y$, and $\widehat \pi$ are bounded by Assumption \ref{asmp:boundedAY} and construction (Lemma \ref{lem:cdakernel_bounded}). Lemma \ref{lem:cdalpr_third_moment} establishes that		
    \begin{equation} \label{eq:lyapunov_numerator}
        nh_\mu^{\frac{3d}{2}} \bbE \{ \left|\widehat \mu(X) \right|^3 \mid D^n \} \inas 0.
    \end{equation}
    Therefore, by the continuous mapping theorem,
    $$
    \frac{\bbE \Big[ \left| \widehat \varphi(Z) - \bbE \{ \widehat \varphi(Z) \mid D_\pi, D_\mu \} \right|^3 \mid D_\pi, D_\mu \Big]}{\sqrt{n}\ \bbV \{ \widehat \varphi(Z) \mid D_\pi, D_\mu \}^{3/2}} = \frac{nh_\mu^{\frac{3d}{2}} \bbE \Big[ \left| \widehat \varphi(Z) - \bbE \{ \widehat \varphi(Z) \mid D_\pi, D_\mu \} \right|^3 \mid D_\pi, D_\mu \Big]}{nh_\mu^{\frac{3d}{2}}\sqrt{n}\ \bbV \{ \widehat \varphi(Z) \mid D_\pi, D_\mu \}^{3/2}} \inas 0.
    $$
\end{proof}

\begin{lemma} \label{lem:marginal_variance_clt}
    Under the conditions of Theorem~\ref{thm:inference}, suppose without loss of generality that $\widehat \mu$ is the higher-order kernel estimator with bandwidth scaling as $h_\mu \asymp n^{-\frac{2 + \varepsilon}{2\alpha + 2\beta + d}}$ while $\widehat \pi$ is the smooth kernel estimator which is consistent. Then,
    $$
    \bbV \{ \widehat \varphi(Z) \} \asymp \frac{1}{nh_\mu^d}.
    $$
\end{lemma}

\begin{proof}
    
    Since $\bbV\{ \varphi(Z) \}$ is a constant by Assumptions~\ref{asmp:dgp} and~\ref{asmp:bdd_density}. Therefore, if $\bbV \{ \widehat \varphi(Z) - \varphi(Z) \}$ is increasing with sample size then $\bbV \{ \widehat \varphi(Z) \} \asymp \bbV \{ \widehat \varphi(Z) - \varphi(Z) \}$. 
    We have 
    $$
    \bbV \{ \widehat \varphi(Z) - \varphi(Z) \} = \bbE [ \{ \widehat \varphi(Z) - \varphi(Z) \}^2 ] - \bbE \{ \widehat \varphi(Z) - \varphi(Z) \}^2.
    $$
    By the analysis in Theorem~\ref{thm:minimax}, 
    $$
    \bbE \{ \widehat \varphi(Z) - \varphi(Z) \}^2 \lesssim h_\mu^{2(\alpha + \beta)}
    $$
    Omitting $X$ arguments, 
    \begin{align*}
        \bbE [ \{ \widehat \varphi(Z) - \varphi(Z) \}^2 ] &= \bbE \left[ \big\{ ( A - \widehat \pi )( \mu - \widehat \mu ) + ( Y - \mu )( \pi - \widehat \pi ) \big\}^2 \right] \\
        &= \bbE \big\{ ( A - \widehat \pi )^2 ( \mu - \widehat \mu )^2 \big\} + \bbE \big\{ ( Y - \mu )^2 ( \pi - \widehat \pi )^2 \big\} \\
        &\hspace{0.5in} + 2 \bbE \left\{ ( A - \widehat \pi )( Y - \mu )( \pi - \widehat \pi )( \mu - \widehat \mu ) \right\} \\
        &= \bbE \big\{ ( A - \pi + \pi - \widehat \pi )^2 ( \mu - \widehat \mu )^2 \big\} + \bbE \left[ \bbE \big\{ ( Y - \mu )^2 \mid X \big\} ( \pi - \widehat \pi )^2 \right] \\
        &\hspace{0.5in} + 2 \bbE \left[ \big\{ A ( Y - \mu ) - \widehat \pi ( Y - \mu ) \big\} ( \mu - \widehat \mu )( \pi - \widehat \pi ) \right] \\
        &= \bbE \left[ \big\{ ( A - \pi )^2 + ( \pi - \widehat \pi )^2 \big\} ( \mu - \widehat \mu )^2 \right] + \bbE \Big( \bbV(Y \mid X) \{ \pi - \widehat \pi \}^2 \Big) \\
        &\hspace{0.5in} + 2 \bbE \big\{ \cov(A, Y \mid X)  ( \mu - \widehat \mu )( \pi - \widehat \pi ) \big\} \\
        &= \bbE \big\{ ( \pi - \widehat \pi )^2 ( \mu - \widehat \mu )^2 \big\} \\
        &\hspace{0.5in} + \bbE \Big\{ \bbV(A \mid X) ( \mu - \widehat \mu )^2 \Big\} + \bbE \Big\{ \bbV(Y \mid X) ( \pi - \widehat \pi )^2 \Big\} \\
        &\hspace{0.5in} + 2 \bbE \Big\{ \cov(A, Y \mid X)  ( \mu - \widehat \mu )( \pi - \widehat \pi ) \Big\} 
    \end{align*}
    where the first line follows by definition; the second by multiplying the square; the third by adding and subtracting $\pi(X)$ in the first term, iterated expectation on the second term, and multiplying out the third term; the fourth by multiplying out the square in the first term and iterated expectations on $X$ and the training data, by definition of $\bbV (Y \mid X)$ on the second term, and by iterated expectation on $X$ and the training data and by definition of $\cov(A, Y \mid X)$ on the third term; and the final line follows by iterated expectations on $X$, the definition of $\bbV(A \mid X)$, and rearranging.
    
    \medskip
    
    Notice that $\bbE \big\{ ( \pi - \widehat \pi )^2 ( \mu - \widehat \mu )^2 \big\} = O \left[ \bbE \{ (\widehat \mu - \mu)^2 \} \right]$ and $\bbE \Big\{ \bbV(Y \mid X) ( \pi - \widehat \pi )^2 \Big\} = O(1)$ because $\widehat \pi$ and $\pi$ are bounded by Assumption \ref{asmp:boundedAY} and construction (Lemma \ref{lem:cdakernel_bounded}), while  $2 \bbE \Big\{ \cov(A, Y \mid X)  ( \mu - \widehat \mu )( \pi - \widehat \pi ) \Big\}  = O \left[ \sqrt{\bbE \{ (\widehat \mu - \mu)^2 \}} \right]$ by Cauchy-Schwarz and Assumption~\ref{asmp:boundedAY}.  Finally, by Assumptions~\ref{asmp:dgp} and~\ref{asmp:bdd_density}, and Lemma~\ref{lem:cdalpr_lower_bounds},
    $$
    \bbE \left\{ \bbV(A \mid X) (\widehat \mu - \mu)^2 \right\} \gtrsim \frac{1}{nh_\mu^d}.
    $$
    Since $\frac{1}{nh_\mu^d}$ is increasing with sample size, this final term then dominates the expression and 
    $$
    \bbV \{ \widehat \varphi(Z) - \varphi (Z) \} \gtrsim \frac{1}{nh_\mu^d}.
    $$
    Moreover, because $\frac{1}{nh_\mu^d}$ is increasing with sample size, $\bbV \{ \widehat \varphi(Z) \} \asymp \bbV \{ \widehat \varphi(Z) - \varphi(Z) \} \gtrsim \frac{1}{nh_\mu^d}$.  The upper bound, $\bbV \{ \widehat \varphi(Z) - \varphi(Z) \} \lesssim \frac{1}{nh_\mu^d}$, follows by the same decomposition as above, but applying the upper bounds from Lemma~\ref{lem:cdalpr_upper_bounds}.
\end{proof}

\begin{lemma} \label{lem:t1}
    Under the conditions of Theorem~\ref{thm:inference}, 
    $$
    \frac{\bbV \{ \widehat \varphi(Z) \}}{\bbV \{ \widehat \varphi(Z) \mid D^n \}} = O_\bbP (1).
    $$
\end{lemma}

\begin{proof}
    Suppose without loss of generality that $\widehat \mu$ is the estimator with bandwidth scaling as $h_\mu \asymp n^{-\frac{2 + \varepsilon}{2\alpha + 2\beta + d}}$ while $\widehat \pi$ is consistent. By Lemma~\ref{lem:marginal_variance_clt}, 
    $$
    nh_\mu^d \bbV \{ \widehat \varphi(Z) \} \asymp 1.
    $$
    By Lemma~\ref{lem:cond_variance_clt}, 		
    $$	
    nh_\mu^d \bbV \{ \widehat \varphi(Z) \mid D^n \} \inas \bbE \left\{ \frac{\bbV (A \mid X)Y^2}{f(X)} \right\} \bbE\left\{ \frac{K_\mu (X)^2}{f(X)} \right\}.
    $$
    The result follows from these two combined.  The same holds if the roles of $\widehat \pi$ and $\widehat \mu$ were reversed.
\end{proof}

\begin{lemma} \label{lem:t2}
    Under the conditions of Theorem~\ref{thm:inference}, 
    $$
    \frac{\bbE \{ \widehat \varphi(Z) \mid D^n \} - \bbE\{ \widehat \varphi(Z) \}}{\sqrt{\bbV \{ \widehat \varphi(Z) \} / n}} \inprob 0.
    $$
\end{lemma}

\begin{proof}
    We prove convergence in quadratic mean.  The expression in the lemma is mean zero by iterated expectations, 
    $$
    \bbE \left[ \frac{\bbE \{ \widehat \varphi(Z) \mid D^n \} - \bbE\{ \widehat \varphi(Z) \}}{\sqrt{\bbV \{ \widehat \varphi(Z) \} / n}}  \right] = 0.
    $$
    Therefore, it suffices to show that the variance of the expression in the lemma converges to zero; i.e., 
    $$
    \frac{n \bbV \left[ \bbE \{ \widehat \varphi (Z) \mid D^n \} \right]}{\bbV \{ \widehat \varphi(Z) \}} \to 0.
    $$
    By Lemma~\ref{lem:marginal_variance_clt}, 
    $$
    \bbV \{ \widehat \varphi(Z) \} \asymp \frac{1}{nh_\mu^d}.
    $$
    Consider $Z_i, Z_j$ drawn iid from the same distribution as $Z$, and which are independent of $D^n$ (like $Z$). Then,
    \begin{align*}
        \bbV \left[ \bbE \{ \widehat \varphi (Z) \mid D^n \} \right] &= \cov \Big[ \bbE \{ \widehat \varphi(Z_i) \mid D^n \}, \bbE \{ \widehat \varphi(Z_j) \mid D^n \} \Big] \\
        &= \cov \Big[ \bbE \{ \widehat \varphi(Z_i) - \varphi(Z_i) \mid D^n \}, \bbE \{ \widehat \varphi(Z_j) - \varphi(Z_j) \mid D^n \} \Big] \\
        &= \cov \{ \widehat \varphi(Z_i) - \varphi(Z_i), \widehat \varphi(Z_j) - \varphi(Z_j) \} - \bbE \Big[ \cov \{ \widehat \varphi(Z_i) - \varphi(Z_i), \widehat \varphi(Z_j) - \varphi(Z_j) \mid D^n \} \Big] \\
        &= \cov \{ \widehat \varphi(Z_i) - \varphi(Z_i), \widehat \varphi(Z_j) - \varphi(Z_j) \}
    \end{align*}
    where the first line follows because $Z, Z_i, Z_j$ are identically distributed, the second line because $\bbE \{ \varphi(Z) \mid D^n \}$ is not random because $\varphi$ does not depend on the training data, the third by the law of total covariance, and the last because $Z_i$ and $Z_j$ are independent. Like in the proof of Lemma~\ref{lem:expansion} in Appendix~\ref{app:modelfree}, we have 
    \begin{align*}
        &\cov \{ \widehat \varphi(Z_i) - \varphi(Z_i), \widehat \varphi(Z_j) - \varphi(Z_j) \} \\
        &= \cov \left[ \bbE \{ \widehat \varphi(Z_i) - \varphi(Z_i) \mid X_i, X_j, D^n \}, \bbE \{ \widehat \varphi(Z_j) - \varphi(Z_j) \mid X_i, X_j, D^n \} \right] + 0 \\
        &= \bbE \left( \cov \left[ \bbE \{ \widehat \varphi(Z_i) - \varphi(Z_i) \mid X_i, D^n \}, \bbE \{ \widehat \varphi(Z_j) - \varphi(Z_j) \mid X_j, D^n \} \mid X_i, X_j \right] \right) + 0 \\
        &\equiv \bbE \left[ \cov \left\{ \widehat b_\varphi(X_i), \widehat b_\varphi(X_j) \mid X_i, X_j \right\} \right]
    \end{align*}
    by successive applications of the law of total covariance, and where $\widehat b_\varphi(X_i)$ is defined in Lemma~\ref{lem:expansion}. From here, because $X_i \neq X_j$, we can use the same argument as in the proof of Proposition~\ref{prop:spectral} (see \eqref{eq:spectral_minimum}), and conclude
    \begin{equation*}
        \bbE \left[ \cov \left\{ \widehat b_\varphi(X_i), \widehat b_\varphi(X_j) \mid X_i, X_j \right\} \right] \lesssim \frac{\lVert b_\pi^2 \rVert_\infty + \lVert b_\mu^2 \rVert_\infty + \min ( \lVert s_\pi^2 \rVert_\infty, \lVert s_\mu^2 \rVert_\infty )}{n} \lesssim \frac{1}{n}.
    \end{equation*}
    where the first inequality follows by Proposition~\ref{prop:spectral} and Lemma~\ref{lem:cdalpr_upper_bounds}, and the second by Lemma~\ref{lem:cdalpr_upper_bounds}. Therefore,
    $$
    n \bbV \left[ \bbE \{ \widehat \varphi (Z) \mid D^n \} \right] \lesssim 1,
    $$
    and so
    $$
    \frac{n \bbV \left[ \bbE \{ \widehat \varphi (Z) \mid D^n \} \right]}{\bbV \{ \widehat \varphi(Z) \}} \lesssim nh_\mu^d \to 0 \text{ as } n \to \infty,
    $$
    where convergence to zero follows because $h_\mu \asymp n^{-\frac{2 + \varepsilon}{2\alpha + 2\beta + d}}$.
\end{proof}

\begin{lemma} \label{lem:t3}
    Under the conditions of Theorem~\ref{thm:inference}, 
    $$
    \frac{\bbE \{ \widehat \varphi(Z) - \varphi(Z) \}}{\sqrt{\bbV \{ \widehat \varphi(Z) \} / n}} \to 0.
    $$
\end{lemma}

\begin{proof}
    The ratio $\frac{\bbE \{ \widehat \varphi(Z) - \varphi(Z) \}}{\sqrt{\bbV \{ \widehat \varphi(Z) \} / n}}$ is not random because the expectation and variance are over the estimation and training data.  By the analysis in Theorem~\ref{thm:minimax}, 
    $$
    \bbE \{ \widehat \varphi(Z) - \varphi(Z) \} \lesssim h_\mu^{\alpha + \beta} \lesssim n^{-\frac{(2 + \varepsilon)(\alpha + \beta)}{2\alpha +2\beta + d}}
    $$
    Assume without loss of generality that $\widehat \mu$ is the undersmoothed nuisance function estimator, then by Lemma~\ref{lem:marginal_variance_clt}, 
    $$
    \bbV \{ \widehat \varphi(Z) \} \asymp \frac{1}{nh_\mu^d}.
    $$
    Therefore, 
    $$
    \frac{\bbE \{ \widehat \varphi(Z) - \varphi(Z) \}}{\sqrt{\bbV \{ \widehat \varphi(Z) \} / n}} \lesssim n h_\mu^{d/2} n^{-\frac{(2 + \varepsilon)(\alpha + \beta)}{2\alpha + 2\beta + d}} \to 0 \text{ as } n \to \infty
    $$
    because $h_\mu \asymp n^{-\frac{2 + \varepsilon}{2\alpha + 2\beta + d}}$.
\end{proof}

\section{Technical results regarding the covariate density} \label{app:technical}

Below, we state and prove three technical lemmas about the covariates $\{X_i\}_{i=1}^{n}$ if their density is bounded above and below as in Assumption~\ref{asmp:bdd_density}.

\begin{lemma} \label{lem:sphere} \emph{(Sphere Lemma)} Assume $X$ has density $f(X)$ that satisfies Assumption~\ref{asmp:bdd_density} and let $B_h (x)$ denote a ball of radius $h$ around a fixed point $x \in \mathcal{X}$. Then
    \begin{equation}
        \bbP \{ X \in B_h (x) \} \asymp h^d
    \end{equation}
\end{lemma}

\begin{proof}
    The volume of a ball with radius $r$ in $d$ dimensions scales like $r^d$.  The result follows because the density is upper and lower bounded. 
\end{proof}

\begin{lemma}\label{lem:separated_covariates} \emph{(Well separated training covariates)}.  
Let $\{ X_i \}_{i=1}^{n}$ be $n$ covariate data points satisfying Assumption~\ref{asmp:bdd_density} (bounded density).  Let $P_n$ denote the random variable counting (twice) all pairs of covariates closer than $2h$ where $h$ is a bandwidth scaling with sample size; i.e.,
$$
P_n = \sum_{i=1}^{n} \sum_{j\neq i}^{n} \one \left( \lVert X_i - X_j \rVert \leq 2h \right).
$$
If $h$ satisfies $nh^d \asymp n^{-\alpha}$ for $\alpha > 0$ as $n \to \infty$, then 
\begin{equation} \label{eq:cov_prob_bound}
    \frac{P_n}{n} \inas 0.
\end{equation}
\end{lemma}

\begin{proof}
The result follows by a moment inequality for U-statistics and the Borel-Cantelli lemma. First, we relate the un-decoupled U-statistic, $P_n$, to the relevant decoupled U-statistic. Let $\{ X_i^{(1)} \}_{i=1}^{n}$ and $\{X_j^{(2)} \}_{j=1}^{n}$ denote two independent sequences drawn from the same distribution as $\{ X_i \}_{i=1}^{n}$. Let
\begin{equation}
    P_n^\prime := \sum_{i=1}^{n} \sum_{j \neq i}^{n} \one \left(\lVert X_i^{(1)} - X_j^{(2)} \rVert \leq 2h \right).
\end{equation}
By Theorem 3.1.1 in \citet{de1999decoupling}, for $p \geq 1$,
\begin{equation}
    \bbE \left\{ \left( \frac{P_n}{n} \right)^p \right\} \lesssim \bbE\left\{ \left(\frac{P_n^\prime}{n}\right)^p \right\}.
\end{equation}
Then, by Proposition 2.1 and the right-hand side of (2.2) in \citet{gine2000exponential}, for all $p > 1$, 
\begin{equation} \label{eq:expansion}
    \bbE \left\{ \left( \frac{P_n^\prime}{n} \right)^p \right\} \lesssim (nh^d)^p + nh^{dp} + n^{2-p}h^d.
\end{equation}
This follows because the kernel is $\frac{\one\left( \lVert X_i^{(1)} - X_j^{(2)} \rVert \leq 2h \right)}{n}$, which satisfies
\begin{align*}
    \bbE \left\{ \frac{\one\left( \lVert X_i^{(1)} - X_j^{(2)} \rVert \leq 2h \right)}{n}  \right\}^p &\lesssim \left(\frac{h^d}{n} \right)^p, \\
    \bbE \left\{ \frac{\one\left( \lVert X_i^{(1)} - X_j^{(2)} \rVert \leq 2h \right)}{n} \mid X_i  \right\}^p &\lesssim \left(\frac{h^d}{n} \right)^p\text{, and} \\
    \bbE \left[ \left\{ \frac{\one\left( \lVert X_i^{(1)} - X_j^{(2)} \rVert \leq 2h \right)}{n} \right\}^p \right] &\lesssim \frac{h^d}{n^p}.
\end{align*}
To conclude, we prove an infinitely summable concentration inequality directly. Let $\epsilon > 0$. By \eqref{eq:expansion} and Markov's inequality, for all $p \geq 2$,
\begin{equation}
    \mathbb{P} \left( \frac{P_n}{n} \geq \epsilon \right) \lesssim (nh^d)^p + nh^{dp} + n^{2-p}h^d \asymp n^{-\alpha p} + o(n^{-(1 + \alpha)}) + o(n^{-(1+\alpha)}),
\end{equation}
where the right-hand side follows by the conditions on the bandwidth. Hence, for $p > \frac{1 + \delta}{\alpha}$ for any $\delta > 0$, $\mathbb{P} \left( \frac{P_n}{n} \geq \epsilon \right) = o(n^{-(1+\delta)})$ for all $\epsilon > 0$, and therefore the result follows by the Borel-Cantelli lemma.
\end{proof}

\begin{lemma} \label{lem:separated_triples} \emph{(Triply well separated training covariates)}. Let $\{ X_i \}_{i=1}^{n}$ be $n$ covariate data points satisfying Assumption~\ref{asmp:bdd_density} (bounded density).  Let $Q_n$ denote the random variable counting (six times) all triples of covariates closer than $2h$ where $h$ is a bandwidth scaling with sample size; i.e.,
\begin{equation}
    Q_n = \sum_{i \neq j \neq k}^n \one \left( \lVert X_i - X_j \lVert \leq 2h \right)\one \left(\lVert X_i - X_k \rVert \leq 2h \right)\one\left( \lVert X_j - X_k \rVert \leq 2h\right).
\end{equation}
If $h$ satisfies $nh^d \asymp n^{-\alpha}$ for $\alpha > 0$ as $n \to \infty$, then
\begin{equation}
    \frac{Q_n}{n} \inas 0.
\end{equation}
\end{lemma}

\begin{proof}
The result follows by the same approach as the previous lemma, but applying a moment inequality for U-statistics of order 3. First, let $\{ X_i^{(1)} \}_{i=1}^{n}$, $\{X_j^{(2)} \}_{j=1}^{n}$, and $\{ X_k^{(3)} \}_{k=1}^{n}$ denote three independent sequences drawn from the same distribution as $\{ X_i \}_{i=1}^{n}$. Moreover, let
\begin{equation}
    Q_n^\prime :=  \sum_{i \neq j \neq k}^n \one \left( \lVert X_i^{(1)} - X_j^{(2)} \lVert \leq 2h \right)\one \left(\lVert X_i^{(1)} - X_k^{(3)} \rVert \leq 2h \right)\one\left( \lVert X_j^{(2)} - X_k^{(3)} \rVert \leq 2h\right).
\end{equation}
Then, by Theorem 3.1.1 in \citet{de1999decoupling} and Proposition 2.1 and the right-hand side of (2.2) in \citet{gine2000exponential}, for all $p > 1$, 
\begin{equation} \label{eq:expansion2}
    \bbE \left\{ \left( \frac{Q_n}{n} \right)^p \right\} \lesssim (nh^d)^{2p} + n (nh^{2d})^p + n^2 h^{dp} + n^{3-p} h^d.
\end{equation}
This follows because the kernel is $\frac{ \one \left( \lVert X_i^{(1)} - X_j^{(2)} \lVert \leq 2h \right)\one \left(\lVert X_i^{(1)} - X_k^{(3)} \rVert \leq 2h \right)\one\left( \lVert X_j^{(2)} - X_k^{(3)} \rVert \leq 2h\right)}{n}$, which satisfies
\begin{align*}
    &\bbE\left\{ \frac{ \one \left( \lVert X_i^{(1)} - X_j^{(2)} \lVert \leq 2h \right)\one \left(\lVert X_i^{(1)} - X_k^{(3)} \rVert \leq 2h \right)\one\left( \lVert X_j^{(2)} - X_k^{(3)} \rVert \leq 2h\right)}{n} \right\}^p \lesssim \left(\frac{h^{2d}}{n}\right)^p, \\
    &\bbE\left\{ \frac{\one \left( \lVert X_i^{(1)} - X_j^{(2)} \lVert \leq 2h \right)\one \left(\lVert X_i^{(1)} - X_k^{(3)} \rVert \leq 2h \right)\one\left( \lVert X_j^{(2)} - X_k^{(3)} \rVert \leq 2h\right)}{n} \mid X_i^{(1)} \right\}^p \lesssim \left(\frac{h^{2d}}{n}\right)^p, \\
    &\hspace{-0.3in}\bbE\left\{ \frac{\one \left( \lVert X_i^{(1)} - X_j^{(2)} \lVert \leq 2h \right)\one \left(\lVert X_i^{(1)} - X_k^{(3)} \rVert \leq 2h \right)\one\left( \lVert X_j^{(2)} - X_k^{(3)} \rVert \leq 2h\right)}{n} \mid X_i^{(1)}, X_j^{(2)} \right\}^p \lesssim \left(\frac{h^{d}}{n}\right)^p\text{, and } \\
    &\bbE \left[ \left\{ \frac{\one \left( \lVert X_i^{(1)} - X_j^{(2)} \lVert \leq 2h \right)\one \left(\lVert X_i^{(1)} - X_k^{(3)} \rVert \leq 2h \right)\one\left( \lVert X_j^{(2)} - X_k^{(3)} \rVert \leq 2h\right)}{n} \right\}^p \right] \lesssim \frac{h^d}{n^p}.
\end{align*}    
Let $\epsilon > 0$. Then, by Markov's inequality, for all $p \geq 3$,
\begin{equation}
    \mathbb{P} \left( \frac{Q_n}{n} \geq \epsilon \right) \lesssim (nh^d)^{2p} + n (nh^{2d})^p + n^2 h^{dp} + n^{3-p} h^d \asymp n^{-2\alpha p} + o(n^{-(1+\alpha)}),
\end{equation}
where the right-hand side follows by the conditions on the bandwidth.  Hence, for $p > \frac{1 + \delta}{2\alpha}$ for any $\delta > 0$, $\mathbb{P} \left( \frac{Q_n}{n} \geq \epsilon \right) = o(n^{-(1+\delta)})$ for all $\epsilon > 0$, and therefore the result follows by the Borel-Cantelli lemma.
\end{proof}	

\section{A strong law of large number for a triangular array of bounded random variables} \label{app:slln}

The following result is a simple strong law of large numbers for a triangular array of bounded random variables.

\begin{lemma} \label{lem:slln}
Let $\{ \xi_{i,n} \}_{i=1}^{n} \stackrel{iid}{\sim} \mathcal{P}_n$ for $n \in \mathbb{N}$ denote a triangular array of random variables which are row-wise iid. If the random variables satisfy
\begin{enumerate}
    \item $|\xi_{i,n}| < B$ for all $i$ and $n$ and some $B < \infty$, and 
    \item $\bbE ( \xi_{1,n} ) \stackrel{n \to \infty}{\longrightarrow} \mu$ for some $\mu \in \bbR$,
\end{enumerate}
then
\begin{equation}
    \frac{1}{n} \sum_{i=1}^{n} \xi_{i,n} \inas \mu.
\end{equation}
\begin{proof}
    The proof follows by a combination of Hoeffding's inequality and the Borel-Cantelli lemma.
    
    \medskip
    
    Let $t > 0$. Because $\bbE ( \xi_{1,n} ) \stackrel{n \to \infty}{\longrightarrow} \mu$, there exists some $N \in \mathbb{N}$ such that $| \bbE ( \xi_{1,n} ) - \mu| < \frac{t}{2}$ for all $n \geq N$. Hence, for $n \geq N$, by the triangle inequality,
    \begin{align}
        \bbP \left( \left| \frac{1}{n} \sum_{i=1}^{n}  \xi_{i,n} - \mu \right| \geq t \right) &= \bbP \left( \left| \frac{1}{n} \sum_{i=1}^{n} \xi_{i,n} - \bbE ( \xi_{1,n} ) + \bbE ( \xi_{1,n} ) - \mu \right| \geq t \right) \\
        &\leq \bbP \left( \left| \frac{1}{n} \sum_{i=1}^{n} \xi_{i,n} - \bbE ( \xi_{1,n} ) \right| + \left| \bbE ( \xi_{1,n} ) - \mu \right| \geq t \right) \\
        &= \bbP \left( \left| \frac{1}{n} \sum_{i=1}^{n} \xi_{i,n} - \bbE ( \xi_{1,n} ) \right| \geq t - \left| \bbE ( \xi_{1,n} ) - \mu \right| \right) \\
        &\leq \bbP \left( \left| \frac{1}{n} \sum_{i=1}^{n} \xi_{i,n} - \bbE ( \xi_{1,n} ) \right| \geq \frac{t}{2} \right).
    \end{align}
    Applying Hoeffding's inequality to the final line gives
    \begin{equation}
        \bbP \left( \left| \frac{1}{n} \sum_{i=1}^{n} \xi_{i,n} - \mu \right| \geq t \right) \leq 2\exp \left\{ - \frac{2nt^2}{16B^2} \right\}.
    \end{equation}
    The result then follows because $\sum_{n=1}^\infty \bbP \left( \left| \frac{1}{n} \sum_{i=1}^{n} \xi_{i,n} - \mu \right| \geq t \right) < \infty$ and by the Borel-Cantelli lemma.
\end{proof}
\end{lemma}

\section{Series regression} \label{app:series}

In this section, we consider series regression for the nuisance function estimators, and establish equivalent results to Lemma~\ref{lem:covariance} and Theorem~\ref{thm:semiparametric}.  Series regression is well studied and includes bases such as the Legendre polynomial series, the local polynomial partition series, and the Cohen-Daubechies-Vial wavelet series \citep{belloni2015some, hansen2022econometrics}.   Here, we focus on regression splines \citep{newey2018cross, fisher2023threeway} and wavelet estimators \citep{mcgrath2022undersmoothing}.  Regression splines are a natural global averaging estimator to consider because, like the local averaging estimators we considered in Section~\ref{sec:unknown}, they do not require knowledge of the covariate density.  The wavelet estimators are a natural alternative because, like the covariate-density-adapted kernel regression we considered in Section~\ref{sec:known}, they can achieve the minimax rate in the non-$\sqrt{n}$ regime.  From a technical perspective, our examination of each of these estimators may be of interest because our proofs that they achieve $\sqrt{n}$-consistency and minimax optimality are different from those considered previously.

\subsection{Regression splines}

First, we review regression splines.
\begin{estimator} \textbf{\emph{(Regression Splines)}} \label{est:spline}   
    The regression spline estimator for $\mu(x) = \bbE(Y \mid X = x)$ is
    \begin{equation}
        \widehat \mu(x) = \sum_{Z_i \in D_\mu} \frac{g(x)^T \widehat Q^{-1} g(X_i)}{n} Y_i
    \end{equation}
    where $g: \bbR^d \to \bbR^{k_\mu}$ is a $k_\mu$ order polynomial spline basis, and
    $$
    \widehat Q = \frac1n \sum_{X_i \in X^n_\mu} g(X_i) g(X_i)^T.
    $$
    Additionally, the spline neighborhoods are approximately evenly sized (see, Assumption 3 in \citet{fisher2023threeway}), so that the distance between two points within a neighborhood scales like $\lesssim k_\mu^{-1/d}$. The regression spline estimator for $\pi(x) = \bbE(A \mid X = x)$ is defined analogously on $D_\mu$.
\end{estimator}

The additional condition we impose, that the neighborhoods are approximately evenly sized, can be enforced under Assumption~\ref{asmp:bdd_density} that the covariate density and covariate support are bounded. 

\subsection{Wavelet estimators}

Here, we review wavelet estimators.  For simplicity, we focus on the case where the covariate density is known and sufficiently smooth, as in Assumption~\ref{asmp:density_smooth}, and propose the same estimator as that considered in \citet{mcgrath2022undersmoothing}. 

\begin{estimator} \textbf{\emph{(Wavelet estimator)}} \label{est:wavelet}
    The wavelet estimator for $\mu(x) = \bbE(Y \mid X = x)$ is
    \begin{equation}
        \widehat \mu(x) = \sum_{Z_i \in D_\mu} \frac{K_{V_{k_\mu}} (x, X_i)}{n f(X_i)} Y_i
    \end{equation}
    where $K_{V_{k_\mu}}(x, X_i)$ denotes the orthogonal projection kernel onto the linear subspace $V_{k_\mu}$ as defined in Appendix A of \citet{mcgrath2022undersmoothing}.  The wavelet estimator for $\pi(x) = \bbE(A \mid X = x)$ is defined analogously on $D_\pi$.
\end{estimator}

\subsection{Lemma~\ref{lem:covariance} and Theorem~\ref{thm:semiparametric} for series regression}

For brevity, we simply assume standard bias and variance bounds for regression splines and wavelet estimators hold.
\begin{assumption}[Bias and variance bounds] \label{asmp:series_bounds}
    For regression splines and wavelet estimators, we suppose
    \begin{align}
        \sup_{x \in \mathcal{X}} \left| \bbE \{ \widehat \mu(x) - \mu(x) \} \right| &\lesssim k_\mu^{-\beta / d}, \text{and} \\
        \sup_{x \in \mathcal{X}} \bbV \{ \widehat \mu(x) \} &\lesssim \frac{k_\mu}{n}
    \end{align}
    and analogous results hold for $\pi(x)$ and $\widehat \pi(x)$.
\end{assumption}

These are the typical bias and variance bounds from the series regression literature. Further assumptions are typically necessary to establish them, analogous to those we enforced for the local polynomial regression estimator, so that the Gram matrix is invertible. The next assumption is a typical example for the design matrix.
\begin{assumption} \label{asmp:ortho_design} \emph{\textbf{(Bounded Minimum Eigenvalue)}}    
    For Estimator~\ref{est:spline}, there exists $\lambda_0 > 0$ such that, uniformly over all $n$,
    $$
    \lambda_{\min} \left[ \bbE \left\{ g(X) g(X)^T \right\} \right] \geq \lambda_0.
    $$
\end{assumption}

This assumption requires that the regressors $g_1(X), ..., g_k(X)$ are not too co-linear, and corresponds to Condition A.2 in \citet{belloni2015some} and Assumption 5 in \citet{fisher2023threeway}.  This assumption implicitly constrains the number of bases to grow no faster than the sample size, and constrains the convergence rate of the DCDR estimator in the non-$\sqrt{n}$ regime. We do not investigate this further, but see, e.g., \citet{belloni2015some} and \citet{fisher2023threeway} for comprehensive analyses.

\medskip

In the next result, we prove that the expected absolute covariance term from Lemma~\ref{lem:covariance} decreases inversely with sample size with both regression splines and wavelet estimators.

\begin{lemma} \label{lem:series_covariance}
    Suppose Assumptions~\ref{asmp:dgp},~\ref{asmp:bdd_density}, and~\ref{asmp:holder} hold.  If $\widehat \mu(x)$ is a regression spline (Estimator~\ref{est:spline}) and Assumption~\ref{asmp:ortho_design} holds or $\widehat \mu$ is a wavelet estimator (Estimator~\ref{est:wavelet}) and Assumption~\ref{asmp:density_smooth} holds, then 
    $$
    \bbE \left[ \big| \cov \left\{ \widehat \mu(X_i), \widehat \mu(X_j) \mid X_i, X_j \right\} \big| \right] \lesssim \frac1n.
    $$
    Analogous results hold for $\widehat \pi(X)$.
\end{lemma}
\begin{proof}
    For regression splines, the proof follows by the same technique as for local averaging estimators (e.g., Lemma~\ref{lem:lpr_covariance}) because regression splines partition the covariate space into neighborhoods: if $X_i$ and $X_j$ are far enough apart, then they do not share training data.  Specifically, let $A_{ij}$ denote the event that $X_i$ and $X_j$ are in the same neighborhood according to the basis $g$ in Estimator~\ref{est:spline}.  Then,
    \begin{align*}
        \bbE \left[ \big| \cov \left\{ \widehat \mu(X_i), \widehat \mu(X_j) \mid X_i, X_j \right\} \big| \right] &= \bbE\left[ \left|\cov \left\{ \widehat \mu(X_i), \widehat \mu(X_j) \mid X_i, X_j \right\} \right| A_{ij} \right] \\
        &\leq \sup_{x_i, x_j} \left| \cov \{ \widehat \mu(x_i), \widehat \mu(x_j) \} \right| \bbP ( A_{ij} ) \\
        &\lesssim \sup_{x} \bbV \{ \widehat \mu(x) \}  k_\mu^{-1}  \\
        &\lesssim \frac1n.
    \end{align*}
    where the first line follows because $\cov \{ \widehat \mu(X_i), \widehat \mu(X_j) \mid X_i, X_j \} = 0$ when  $X_i$ and $X_j$ are not in the same neighborhood, the second by H\"{o}lder's inequality, the third by the definition of the size of the neighborhoods in Estimator~\ref{est:spline} and Lemma~\ref{lem:sphere}, and the final line by Assumption~\ref{asmp:series_bounds}.
    
    \medskip
    
    For wavelet estimators, the proof is different.  It follows by the same analysis as in Lemma 15 (i) in \citet{mcgrath2022undersmoothing}, which we repeat here for completeness.  Notice that 
    \begin{align*}
        \bbE \{ \widehat \mu(X_i) \widehat \mu(X_j) \mid X_i, X_j \} &=  \bbE \left[ \sum_{Z_k, Z_l \in D_\mu} \frac{K_{V_{k_\mu}}(X_i, X_k) K_{V_{k_\mu}}(X_j, X_l) Y_k Y_l}{n^2f(X_k) f(X_l)} \mid X_i, X_j \right] \\
        &= \frac{1}{n} \bbE \left[ \frac{K_{V_{k_\mu}}(X_i, X_k) K_{V_{k_\mu}}(X_j, X_k) Y_k^2}{f(X_k)^2} \mid X_i, X_j \right] + \left( 1 - \frac1n \right) \bbE \{ \widehat \mu(X) \mid X \}^2 \\
        &= \bbE \{ \widehat \mu(X) \mid X \}^2 \\
        &+ \frac1n \left( \bbE \left[ \frac{K_{V_{k_\mu}}(X_i, X_k) K_{V_{k_\mu}}(X_j, X_k) Y_k^2}{f(X_k)^2} \mid X_i, X_j \right] - \bbE \{ \widehat \mu(X) \mid X \}^2 \right), \\
    \end{align*}
    where the first line follows by definition, the second by iid datapoints, and the third by rearranging. By the definition of covariance, 
    \begin{align*}
        \cov \{ \widehat \mu(X_i), \widehat \mu(X_j) \mid X_i, X_j \} &= \bbE \{ \widehat \mu(X_i) \widehat \mu(X_j) \mid X_i, X_j \} -  \bbE \{ \widehat \mu(X) \mid X \}^2 \\
        &=  \frac1n \left( \bbE \left[ \frac{K_{V_{k_\mu}}(X_i, X_k) K_{V_{k_\mu}}(X_j, X_k) Y_k^2}{f(X_k)^2} \mid X_i, X_j \right] - \bbE \{ \widehat \mu(X) \mid X \}^2 \right).
    \end{align*}
    Therefore, 
    $$
    \bbE \left[ \big| \cov \left\{ \widehat \mu(X_i), \widehat \mu(X_j) \mid X_i, X_j \right\} \big| \right] \lesssim \frac1n,
    $$
    where the inequality follows by Assumptions~\ref{asmp:dgp} and~\ref{asmp:bdd_density} and because $K_{V_{k_\mu}}(x, y)$ is bounded.  
\end{proof}

\noindent By Assumption~\ref{asmp:series_bounds} and Lemma~\ref{lem:series_covariance}, we have an analogous result to Theorem~\ref{thm:semiparametric}, which we state without proof.

\begin{theorem}
    \emph{\textbf{(Series regression)}}
    Suppose Assumptions~\ref{asmp:dgp},~\ref{asmp:bdd_density}, and~\ref{asmp:holder} hold and $\psi_{ecc}$ is estimated with the DCDR estimator $\widehat \psi_n$ from Algorithm~\ref{alg:dcdr}.  If the nuisance functions $\widehat \mu$ and $\widehat \pi$ are estimated with regression splines (Estimator~\ref{est:spline}), Assumption~\ref{asmp:ortho_design} holds, and the bases scale like $k_\mu, k_\pi \asymp \frac{n}{\log n}$, or if the nuisance functions are estimated with wavelet estimators (Estimator~\ref{est:wavelet}), Assumption~\ref{asmp:density_smooth} holds, and $k_\mu, k_\pi \asymp \frac{n}{\log n}$, then
    \begin{equation}
        \begin{cases}
            \sqrt{\frac{n}{\bbV \{ \varphi(Z) \}}} (\widehat \psi_n - \psi_{ecc}) \indist N(0, 1) &\text{ if } \frac{\alpha + \beta}{2} > d/4 \text{, and} \\[10pt]
            \bbE | \widehat \psi_n - \psi_{ecc} | \lesssim \left( \frac{n}{\log n} \right)^{-\frac{\alpha + \beta}{d}} &\text{ otherwise.}
        \end{cases}
    \end{equation}    
\end{theorem}

This result is optimal for regression splines -- to ensure the Gram matrix is invertible, they cannot be undersmoothed any further, and so the bias of the DCDR estimator cannot be reduced.  For wavelet estimators with known covariate density, this result can be improved in the non-$\sqrt{n}$ regime by undersmoothing even further only one of the two nuisance function estimators and carefully analyzing the bias of the DCDR estimator (see, \citet{mcgrath2022undersmoothing}, Proposition 2).

\end{document}